\documentclass{amsart}

\usepackage{xcolor}
\definecolor{lgray}{RGB}{240,240,240}
\definecolor{webgreen}{rgb}{0,.5,0}
\definecolor{webbrown}{rgb}{.6,0,0}
\definecolor{RoyalBlue}{cmyk}{1, 0.50, 0, 0}
\usepackage[colorlinks=true, breaklinks=true, urlcolor=webbrown, linkcolor=RoyalBlue, citecolor=webgreen, backref=page]{hyperref}

\usepackage{epsfig, graphicx, subfigure}
\usepackage{verbatim, setspace}
\usepackage{amsmath, amssymb, bm}

\usepackage{newtxtext,newtxmath}

\usepackage{mathabx}
\usepackage{pifont}

\makeatletter
\renewcommand{\tocsection}[3]{%
  \indentlabel{\@ifnotempty{#2}{\bfseries\ignorespaces#1 #2\quad}}\bfseries#3}
\renewcommand{\tocsubsection}[3]{%
  \indentlabel{\@ifnotempty{#2}{\ignorespaces#1 #2\quad}}#3}
\newcommand\@dotsep{4.5}
\def\@tocline#1#2#3#4#5#6#7{\relax
  \ifnum #1>\c@tocdepth 
  \else
    \par \addpenalty\@secpenalty\addvspace{#2}%
    \begingroup \hyphenpenalty\@M
    \@ifempty{#4}{%
      \@tempdima\csname r@tocindent\number#1\endcsname\relax
    }{%
      \@tempdima#4\relax
    }%
    \parindent\z@ \leftskip#3\relax \advance\leftskip\@tempdima\relax
    \rightskip\@pnumwidth plus1em \parfillskip-\@pnumwidth
    #5\leavevmode\hskip-\@tempdima{#6}\nobreak
    \leaders\hbox{$\m@th\mkern \@dotsep mu\hbox{.}\mkern \@dotsep mu$}\hfill
    \nobreak
    \hbox to\@pnumwidth{\@tocpagenum{\ifnum#1=1\bfseries\fi#7}}\par
    \nobreak
    \endgroup
  \fi}
\AtBeginDocument{%
\expandafter\renewcommand\csname r@tocindent0\endcsname{0pt}
}
\def\l@subsection{\@tocline{2}{0pt}{2.5pc}{5pc}{}}
\makeatother

\newtheorem{theorem}{Theorem}[section]

\newtheorem{proposition}[theorem]{Proposition}
\newtheorem{lemma}[theorem]{Lemma}

\newcommand{\T}		{\mathbb{T}}
\newcommand{\D}		{\mathbb{D}}
\newcommand{\R}		{\mathbb{R}}
\newcommand{\C}		{\mathbb{C}}
\newcommand{\N}		{\mathbb{N}}
\newcommand{\Z}		{\mathbb{Z}}

\newcommand{\n}{{\vec n}}

\newcommand{\dist}{\mathrm{dist}}

\newcommand{\supp}{\mathrm{supp}\,}

\newcommand{\re}{\mathrm{Re}}
\newcommand{\im}{\mathrm{Im}}

\newcommand{\out}{\Omega}

\newcommand{\qasq}{\quad \text{as} \quad}
\newcommand{\qandq}{\quad \text{and} \quad}

\newcommand{\ic}{\mathrm{i}}

\newcommand{\RS}{{\mathfrak S}}
\newcommand{\z}	{{\boldsymbol z}}

\newcommand{\h}	{{\boldsymbol h}}

\numberwithin{equation}{section}

\begin{document}

\title[Strong Asymptotics of Angelesco MOPs: Non-Marginal Directions]{Strong Asymptotics of Multiple Orthogonal Polynomials for Angelesco Systems. Part I: Non-Marginal Directions.}

\author[A.I. Aptekarev]{Alexander I. Aptekarev}
\address{Keldysh Institute of Applied Mathematics, Russian Academy of Science, Miusskaya Pl. 4, Moscow, 125047 Russian Federation}
\email{\href{mailto:aptekaa@keldysh.ru}{aptekaa@keldysh.ru}}

\author[S. Denisov]{Sergey A. Denisov}
\address{Department of Mathematics, University of Wisconsin-Madison, 480n Lincoln Dr., Madison, WI 53706, USA}
\email{\href{mailto:denissov@math.wisc.edu}{denissov@math.wisc.edu}}

\author[M. Yattselev]{Maxim L. Yattselev}
\address{Department of Mathematical Sciences, Indiana University Indianapolis, 402~North Blackford Street, Indianapolis, IN 46202. USA}

\email{\href{mailto:maxyatts@iu.edu}{maxyatts@iu.edu}}

\thanks{The research of MY was supported in part by a grant from the Simons Foundation, CGM-706591. The research of SD was supported by NSF Grant DMS-2054465 and by Van Vleck Professorship Research Award.}

\subjclass[2010]{47B36, 47A10, 42C05}

\begin{abstract}
In this work, we establish strong asymptotics of multiple orthogonal polynomials of the second type for Angelesco systems with measures that satisfy Szeg\H{o} conditions. We consider multi-indices that converge to infinity in the non-marginal directions.
\end{abstract}

\maketitle

\tableofcontents

\section{Introduction}

\subsection{Orthogonal Polynomials}

Let \( \mu \) be a compactly supported Borel measure on the real line with infinitely many points in \( \supp\mu \), its support. The \emph{\( n \)-th monic orthogonal polynomial} with respect to \( \mu \) is a monic polynomial \( P_n(x) \) of degree \( n \) such that
\[
\int x^kP_n(x)d\mu(x) = 0, \quad k\in\{0,1,\ldots,n-1\}.
\]
It is known that \( P_n(x) \) is unique and that all of its zeros are simple and belong to \( \Delta(\mu)=[\alpha(\mu),\beta(\mu)] \), the \emph{convex hull} of \( \supp\mu \), i.e., the smallest interval containing the support of \( \mu \). One of the central questions of the analytic theory of orthogonal polynomials is to identify their asymptotic behavior in the complex plane as the degree $n\to\infty$. There are three well-established ways to study such a behavior: obtain weak, ratio, and strong asymptotics. The case of orthogonal polynomials on a segment has been thoroughly investigated \cite{Szego,Nevai,StahlTotik,Ismail,Simon,Simon2} and  we outline some of these results for weak and strong asymptotics before discussing multiple orthogonality.

\subsection{Weak Asymptotics}

Weak convergence is intimately related to the logarithmic potential theory in the complex plane. Recall that the \emph{logarithmic potential} of a compactly supported positive Borel measure \( \omega \) is given by
\[
V^\omega(z) = -\int\log|z-t|d\omega(t).
\]
It is a superharmonic function in the complex plane \( \C \),  harmonic away from \( \supp\omega \), and it behaves like \( -|\omega| \log |z| + o(1) \) as \( z\to\infty \), where \( |\omega| \) is the mass of \( \omega \). One can readily notice that
\[
\frac1n \log|P_n(z)| = -V^{\mu_n}(z), \quad \mu_n := \frac1n\sum_{x:P_n(x)=0} \delta_x,
\]
where \( \delta_x \) is the Dirac's delta distribution centered at \( x \) and \( \mu_n \) is the \emph{normalized (probability) counting measure} of the zeros of \( P_n(x) \). Denote by \( I[\omega] \) the logarithmic energy of \( \omega \), that is, \( I[\omega] :=\int V^\omega d\omega \). Every compact set \( K\subset \C \) is either ``small enough'' so that \( I[\omega] = +\infty \) for every probability Borel measure supported on \( K \), in which case \( K \) is called {\it polar}, or there exists a unique minimizer of the logarithmic energy among all probability Borel measures supported on \( K \), say \( \omega_K \), called the \emph{logarithmic equilibrium distribution} on \( K \). The measure \( \mu \) is called \emph{UST-regular} (Ullman-Stahl-Totik) precisely when its support is non-polar and
\[
\frac1n \log|P_n(z)| + V^{\omega_{\supp\mu}}(z) = o(1) \quad \text{in} \quad D_{\Delta(\mu)} := \overline\C\setminus\Delta(\mu),\quad n\to\infty,
\]
where the error term is locally uniform in \( D_{\Delta(\mu)} \), see \cite[Chapter~3, p.61]{StahlTotik}, and \( \overline\C \) is the extended complex plane. In this case the normalized counting measures of zeros \( \mu_n \) converge weak$^*$ to \( \omega_{\supp\mu} \), that is, \( \int fd\mu_n \to \int f d\omega_{\supp\mu} \) for any function \( f \) continuous on \( \Delta(\mu) \). There is a number of  criteria that ensure the UST-regularity of a measure; see \cite[Chapter~4]{StahlTotik}.  For example, write
\begin{equation}
\label{Lebesgue}
d\mu(x) = \mu^\prime(x)dx + d\mu^s(x),
\end{equation}
where \( \mu^s\) is singular to the Lebesgue measure. If \( \supp\mu=\Delta(\mu) \) and $\mu^\prime>0$ a.e. on $\Delta(\mu)$ then  \( \mu \) is UST-regular, see \cite[Chapter~4, p.101]{StahlTotik}, or more generally,  if \( \supp\mu=\Delta(\mu) \) and
\[
\liminf_{r\downarrow 0} r\log\mu([x-r,x+r]) \geq 0
\]
for almost every \( x\in\Delta(\mu) \), then \( \mu \) is UST-regular, see \cite[Chapter~4, p.110]{StahlTotik}.

In the case of measures supported on an interval, i.e., when \( \supp\mu=\Delta=[\alpha,\beta] \), the above asymptotic formulae can be made very explicit: it holds that
\begin{equation}
\label{arcsine}
d\omega_\Delta(x) = \frac{dx}{\pi\sqrt{(x-\alpha)(\beta-x)}} \qandq V^{\omega_\Delta}(z) = \log|\phi_\Delta(z)| - \log\frac4{\beta-\alpha},
\end{equation}
where \( \phi_\Delta(z) \) is the conformal map of \( D_\Delta \) onto $\D$ such that \( \phi_\Delta(\infty) =0 \) and \( \phi_\Delta(\beta)=1 \). That is, the logarithmic equilibrium distribution on \( \Delta \) is simply the arcsine distribution on \( \Delta \). One can also readily verify that
\begin{equation}
\label{phi-w}
\phi_\Delta(z) = \frac2{\beta-\alpha}\left(z - \frac{\beta+\alpha}{2} - w_\Delta(z) \right) \qandq w_\Delta(z) := \sqrt{(z-\alpha)(z-\beta)},
\end{equation}
where the branches are holomorphic off \( \Delta \) and \( w_\Delta(z) = z+\mathcal O(1) \) as \( z\to\infty \).

\subsection{Strong Asymptotics}

Strong asymptotics of orthogonal polynomials is related to the function theory of Hardy spaces. Given a closed interval \( \Delta \), we denote by \( L^p(\omega_\Delta) \) the space of real-valued functions whose moduli are \( p \)-summable with respect to \(\omega_\Delta \). We further denote by \( H^2(D_\Delta) \) the Hardy space of functions holomorphic in \( D_\Delta \) whose squared moduli possess harmonic majorants in \( D_\Delta \), see \cite[Chapter~10, p.168]{Duren}. This definition is conformally invariant meaning that \( f \in H^2(\D) \), the standard Hardy space on the unit disk, if and only if \( f\circ\phi_\Delta\in H^2(D_\Delta) \).  In particular, any \( g\in H^2(D_\Delta) \) possesses non-tangential limits from above and below \( \Delta \), say \( g_\pm \), and  \( g_\pm \in L^2(\omega_\Delta) \) while \( \log |g_\pm| \in L^1(\omega_\Delta) \). We shall say that \( G \) is an \emph{outer function} in \( H^2(D_\Delta) \) if \( G\circ \phi_\Delta^{-1} \) is an outer function in \( H^2(\D) \), see \cite[Section~2.4]{Duren}. For instance, given a non-negative function \( f \in L^2(\omega_\Delta) \) such that \( \log f \in L^1(\omega_\Delta) \), the function
\begin{equation}
\label{outer}
\out_\Delta(f,z) := \exp\left( w_\Delta(z) \int_\Delta \log f(x) \frac{d\omega_\Delta(x)}{z-x} \right)
\end{equation}
is an outer function in \( H^2(D_\Delta) \), $\out_\Delta(f,\infty)>0$, and it  is also conjugate-symmetric (this formula is obtained through conformal equivalence with \( H^2(\D) \) and the known integral representation of outer functions in that space). For such functions it holds that
\begin{equation}
\label{outer-b}
|\out_{\Delta\pm}(f,x)| = f(x) \quad \text{for almost every} \quad x\in\Delta.
\end{equation}
In particular, any conjugate-symmetric outer function in \( H^2(D_\Delta) \) that is positive at infinity can be recovered through the modulus of its boundary values via \eqref{outer}--\eqref{outer-b}. Notice also that
\begin{equation}
\label{geom-mean}
\log \out_\Delta(f,\infty) = \int_\Delta \log fd\omega_\Delta.
\end{equation}

Given a compactly supported Borel measure \( \mu \), it is said that \( \mu \) is a \emph{Szeg\H{o}  measure} on an interval \( \Delta\subseteq\Delta(\mu) \), which we denote by \( \mu\in\mathrm{Sz}(\Delta) \), if \( \log\mu^\prime \in L^1(\omega_\Delta) \), see \eqref{Lebesgue}. In this case it also holds that \( \log v_\Delta \in L^1(\omega_\Delta) \), where
\begin{equation}
\label{szego}
d\mu_{|\Delta}(x) = v_\Delta(x)d\omega_\Delta(x) + d\mu_{|\Delta}^s(x)
\end{equation}
(that is, \( v_\Delta(x) = \pi\mu^\prime(x)\sqrt{(x-\alpha)(\beta-x)} \) for \(x\in\Delta=[\alpha,\beta] \)). When studying strong asymptotics of polynomials orthogonal on the real line, one usually assumes  that \( \Delta=\Delta(\mu) \) in \eqref{szego}. However, in the case of multiple orthogonality, which is the main subject of this work, it will be important for us to take restrictions of \( \mu \) onto proper subintervals of \( \Delta(\mu) \), and this is the reason why we write \( \mu_{|\Delta} \) in \eqref{szego}. When \( \mu\in\mathrm{Sz}(\Delta) \) one can define the so-called \emph{Szeg\H{o} function} of \( \mu_{|\Delta} \), which depends only on the absolutely continuous part of \( \mu \), by setting
\begin{equation}
\label{SzegoFun}
G(\mu_{|\Delta},z) := \out_\Delta\big(\sqrt{v_\Delta},z\big), \quad z\in D_\Delta.
\end{equation}
By its very definition, $G$ is an outer function in \( H^2(D_\Delta) \) whose traces from above and below of \( \Delta \) satisfy \( |G_\pm(\mu_{|\Delta},x)|^2 = v_\Delta(x) \) for almost every \( x\in \Delta \). In what follows, we remove the subscript \( |\Delta \) from \( \mu \) in \eqref{SzegoFun} if \( \Delta=\Delta(\mu) \).  It is known, see \cite[Chapter~XII]{Szego} and \cite[Section~13.3]{Simon}, that if \( \mu \) is a Szeg\H{o} measure on \( \Delta(\mu) \), then
\[
P_n(z) = (1+o(1))\left( \frac{4}{\beta-\alpha}\phi_{\Delta(\mu)}^{-1}(z) \right)^n \frac{G(\mu,\infty)}{G(\mu,z)}
\] 
locally uniformly in \( D_{\Delta(\mu)} \). 

\subsection{Multiple Orthogonal Polynomials}

Our primary goal is an extension of the above results to multiple orthogonal polynomials that can be defined as follows. Let \( \mu_i \), \( i\in I_d:=\{1,2,\ldots,d\} \), \( d\geq2\), be positive compactly supported Borel measures on the real line. Given a multi-index \( \n=(n_1,n_2,\ldots,n_d) \in\Z_+^d \), we denote by \( P_\n(x) \) a non-identically zero monic polynomial of minimal degree such that 
\begin{equation}
\label{ortho}
\int x^kP_\n(x)d\mu_i(x) = 0, \quad k\in\{0,1,\ldots,n_i-1\}, \quad i\in I_d.
\end{equation}
Such a polynomial always exists and is unique. We say that a multi-index  \( \n \) is \emph{normal} if \( \deg P_\n = |\n| \), where  \( |\n| = n_1+n_2+\cdots+n_d \). Moreover, the system \( \vec\mu = (\mu_1,\mu_2,\ldots,\mu_d) \) is called \emph{perfect} if all the multi-indices are normal. The notion of a MOP was elaborated in the constructive Diophantine approximations, see, e.g., \cite{ap2}, and it goes back to the famous proof by Hermite that the number $e$ is transcendental \cite{ap1}.

The questions of the asymptotic behavior of MOPs for an arbitrary system \( \vec\mu \) are hard, especially if such a system is not perfect. Below, we restrict ourselves to the so-called \emph{Angelesco systems}. These are systems of measures \( \vec\mu \) that satisfy conditions
\begin{equation}
\label{Angelesco}
\Delta(\mu_i)\cap \Delta(\mu_j) = \varnothing, \quad i,j\in I_d, \quad i\neq j.
\end{equation}
It is customary to label measures \( \mu_i \) so that \( \beta(\mu_i)<\alpha(\mu_j) \) when \( i<j \). It was shown by Angelesco \cite{Ang19} that such systems are always perfect (this system was later rediscovered in \cite{ap3}).  Moreover, each \( P_\n(x) \) has precisely \( n_i \) zeros on \( \Delta(\mu_i) \), \( i\in I_d \). Hence, for Angelesco systems we can always write
\begin{equation}
\label{partition}
 P_\n(x) = P_{\n,1}(x)P_{\n,2}(x)\cdots P_{\n,d}(x),
\end{equation}
where each \( P_{\n,i}(x) \) is monic and has all its zeros on \( \Delta(\mu_i) \). The existence of various asymptotic limits of MOPs depends on the way multi-index \( |\n| \) approaches infinity. Therefore, given a vector \( \vec c=(c_1,c_2,\ldots,c_d)\in(0,1)^d \) such that \( |\vec c| := c_1+c_2+\cdots+c_d =1 \), we restrict our attention to \emph{ray sequence} of multi-indices defined by
\begin{equation}
\label{Nc}
\mathcal N(\vec c)= \big\{\n:n_i/|\n|\to c_i \text{ as } |\n|\to\infty,~~i\in I_d \big\}.
\end{equation}
Of course, there are many ray sequences corresponding to a given vector \( \vec c \).  In our analysis it will sometimes be important to distinguish \emph{non-marginal} ray sequences, i.e., ray sequences corresponding to \( \vec c\in(0,1)^d \), and \emph{marginal} ones, i.e., those for which at least one coordinate \( c_i \) vanishes. In the current paper, we handle non-marginal sequences only and the marginal sequences will be studied in the forthcoming work.

  There is a large body of literature on asymptotics of MOPs. We shall provide some relevant references further below in Section~\ref{sec:sa} related to Angelesco systems. Besides them another well-studied class of vector-measures is known as Nikishin systems, see \cite{ap5,ap6,ap0}. For Nikishin systems and their generalizations,  weak, ratio, and strong asymptotics were obtained in \cite{ap9,ApLLRocha05,ap10,ap11,FidLopezLLSor10,GonLL24} (also see the references therein).  Certain extensions of Totik's results \cite{Totik} on asymptotics of orthogonal polynomials with varying weights are essential to our approach. We discuss these extensions in Section~\ref{sec:vw}. In turn, the material in Section~\ref{sec:vw} relies on extensions of pioneering results in \cite{ap7,ap8} on ratio asymptotics and the work by de~la Calle Ysern and L\'opez Lagomasino \cite{CYLL98}, see also Stahl's paper \cite{St00} on strong asymptotics of orthogonal polynomials with respect to reciprocal polynomial weights, which we derive in Section~\ref{sec:pw}.

\section{Weak Asymptotics of MOPs}
\label{sec:wa}

Given \( \vec c\in(0,1)^d \), \( |\vec c|=1 \), it was shown by Gonchar and Rakhmanov \cite{Grakh81} that there exists a unique vector of positive Borel measures \( \vec\omega_{\vec c} = (\omega_{\vec c,1}, \omega_{\vec c,2}, \ldots, \omega_{\vec c,d}) \) such that
\begin{equation}
\label{vec_eq1}
\supp\omega_{\vec c,i}=:\Delta_{\vec c,i}\subseteq \Delta(\mu_i), \quad \Delta_{\vec c,i} = [\alpha_{\vec c,i},\beta_{\vec c,i}], \quad |\omega_{\vec c,i}|=c_i, \quad i\in I_d,
\end{equation}
and
\begin{equation}
\label{vec_eq2}
V^{\omega_{\vec c,j}}(x) + \sum_{i\in I_d}V^{\omega_{\vec c,i}}(x) \begin{cases} = \ell_{\vec c,j}, & x\in\Delta_{\vec c,j}, \smallskip \\ > \ell_{\vec c,j}, & x\in\Delta(\mu_j)\setminus\Delta_{\vec c,j}, \end{cases}
\end{equation}
for some constants \( \ell_{\vec c,j} \), \( j\in I_d \) (the presence of a strict inclusion \( \Delta_{\vec c,i} \subsetneq \Delta(\mu_i) \) is colloquially known as a \emph{pushing effect}; it can happen to none, some, or all but one intervals, see Proposition~\ref{prop:Ya} further below). The vector-equilibrium measure \( \vec\omega_{\vec c} \) can also be defined via the energy minimization process  similar to the logarithmic equilibrium measures. The central result of \cite{Grakh81}  is the following theorem.

\begin{theorem}
\label{thm:GR}
Let  \( \{P_\n(z)\}_\n \) be the table of multiple orthogonal polynomials with respect to an Angelesco system of measures \( \vec \mu \). Fix a non-marginal ray sequence \( \mathcal N(\vec c) \). Assume that each \( \mu_i \) is absolutely continuous with respect to the Lebesgue measure on \( \Delta(\mu_i) \) and \( \mu_i^\prime(x)>0 \) almost everywhere on \( \Delta(\mu_i) \).  Then, it holds for each \( i\in I_d \) that
\[
\lim_{\mathcal N(\vec c)} \frac1{n_i}\log |P_{\n,i}(z)| = -c_i^{-1}V^{\omega_{\vec c,i}}(z)
\]
locally uniformly in \( D_{\Delta(\mu_i)} \). Moreover, the normalized counting measures of the zeros of the polynomials \( P_{\n,i}(x) \) converge weak$^*$ to \( c_i^{-1}\omega_{\vec c,i} \) along \( \mathcal N(\vec c) \) for any \( i\in I_d \). 
\end{theorem}

In the remaining part of this section we provide a more detailed description of the vector-equilibrium measure \( \vec\omega_{\vec c} \).  What follows is taken from \cite[Section~2]{Ya16}.

Given pairwise disjoint closed intervals \( (\Delta_1,\Delta_2,\ldots,\Delta_d) \), define \( \RS \) to be a \( (d+1) \)-sheeted compact Riemann surface realized as follows: take \( d+1 \) copies of the extended complex plane; cut the zeroth copy along \( \cup_{i\in I_d}\Delta_i \) and denote it by $\RS_0$; cut the \( i \)-th  copy along a single interval \( \Delta_i \) and denote it by $\RS_i$, \( i\in I_d \); glue $\RS_i$, \( i\in I_d \), to \( \RS_0 \) crosswise along the corresponding cut. Denote by \( \pi \) the natural projection from \( \RS \) onto \( \overline \C \) that takes a point on \( \RS_k \) and maps it into the corresponding point in the cut plane. 

Let \( \RS_{\vec c} \) be the surface corresponding to \( (\Delta_{\vec c,1},\Delta_{\vec c,2},\ldots,\Delta_{\vec c,d}) \). Denote by \( E_{\vec c} \) the set of ramification points of \( \RS_{\vec c} \), i.e., \( E_{\vec c}=\{\boldsymbol \alpha_{\vec c,1},\boldsymbol\beta_{\vec c,1},\ldots,\boldsymbol \alpha_{\vec c,d},\boldsymbol\beta_{\vec c,d} \} \subset \RS_{\vec c,0} \), where \( \pi(\boldsymbol \alpha_{\vec c,i}) = \alpha_{\vec c,i} \) and \( \pi(\boldsymbol \beta_{\vec c,i}) = \beta_{\vec c,i} \). Define
\begin{equation}
\label{CauchyInt}
\boldsymbol h_{\vec c,i}(z) := \int\frac{d\omega_{\vec c,i}(x)}{x-z}, \quad z\in D_{\Delta_{\vec c,i}}, \quad i\in I_d,
\end{equation}
and put \( \boldsymbol h_{\vec c,0}(z) := - \sum_{i\in I_d} \boldsymbol h_{\vec c,i}(z) \). Define \( h_{\vec c} \) to be the function on \( \RS_{\vec c} \) such that \( \boldsymbol h_{\vec c|\RS_{\vec c,k}} = \boldsymbol h_{\vec c,k}\circ \pi_{|\RS_{\vec c,k}} \) for all \( k\in\{0,1,\ldots,d \} \).

\begin{proposition}
\label{prop:Ya}
The function \( \boldsymbol h_{\vec c} \) extends continuously to \( \RS_{\vec c}\setminus E_{\vec c} \) and is in fact a rational function on \( \RS_{\vec c} \). It has a simple zero at each point on top of infinity, a single simple zero
\[
\z_{\vec c,i} \in \RS_{\vec c,0} \cap \pi^{-1}([\beta_{\vec c,i},\alpha_{\vec c,i+1}])
\]
for each \( i\in\{1,2,\ldots,d-1\} \), a simple pole at each element of \( E_{\vec c} \) (if \( \z_{\vec c,i} \) coincides with either \( \boldsymbol\beta_{\vec c,i} \) or \(\boldsymbol \alpha_{\vec c,i+1} \), then it cancels the corresponding pole), and otherwise it is non-vanishing and finite. Moreover, \( \z_{\vec c,i}=\boldsymbol\beta_{\vec c,i} \) (resp. \( \z_{\vec c,i}=\boldsymbol\alpha_{\vec c,i+1} \)) if and only if
\[
V^{\omega_{\vec c,i}}(z) + \sum_{j\in I_d}V^{\omega_{\vec c,j}}(z) - \ell_{\vec c,i}>0
\]
for \( x\in(\beta_{\vec c,i},\beta_{\vec c,i}+\epsilon) \) (resp. \( x\in(\alpha_{\vec c,i+1}-\epsilon,\alpha_{\vec c,i+1}) \)) for some \( \epsilon>0 \), \( i\in\{1,2,\ldots,d-1\} \). Furthermore, we have for every \( i\in I_d \) that
\begin{equation}
\label{future_kappas}
V^{\omega_{\vec c,i}}(z) + \sum_{j\in I_d}V^{\omega_{\vec c,j}}(z) - \ell_{\vec c,i} = \re\left(\int_{\alpha_{\vec c,i}}^z \big( \boldsymbol h_{\vec c,i}(s) - \boldsymbol h_{\vec c,0}(s) \big) ds \right),
\end{equation}
where \( \alpha_{\vec c,i} \) can be replaced by any point on \( \Delta_{\vec c,i} \) as the integrand is purely imaginary on \( \Delta_{\vec c,i} \). It also holds that
\begin{equation}
\label{eq_meas}
d\omega_{\vec c,i}(x) = \big( \boldsymbol h_{\vec c,i+}(x) - \boldsymbol h_{\vec c,i-}(x) \big) \frac{dx}{2\pi\ic}, \quad i\in I_d.
\end{equation}
Finally, if \( \{\vec c_n\}\subset(0,1)^d \), \( |\vec c_n| =1\), is sequence of vectors that converge to \( \vec c \), then the measures \( \omega_{\vec c_n,i} \) converge weak$^*$ to \( \omega_{\vec c,i} \) for each \( i\in I_d \).
\end{proposition}

It can be deduced from this proposition that there is a one-to-one correspondence between \( \vec c\in(0,1)^d \), \( |\vec c|=1 \), and vectors \( (z_1,z_2,\ldots,z_{d-1}) \) such that \( z_1<z_2<\cdots <z_{d-1} \) and \( z_i\in(\alpha(\mu_i),\beta(\mu_{i+1})) \). On the one hand, \( \vec c \) corresponds to \( (\pi(\z_{\vec c,1}),\pi(\z_{\vec c,2}),\ldots,\pi(\z_{\vec c,d})) \). On the other hand, let \( (z_1,z_2,\ldots,z_{d-1}) \) be as described. Set
\[
\alpha_1=\alpha(\mu_1), \;\; \alpha_{i+1} = \max\{\alpha(\mu_{i+1}),z_i\},  \;\; \beta_i = \min\{\beta(\mu_i),z_i\}, \qandq \beta_d=\beta(\mu_d),
\]
where \( i\in I_d\setminus\{d\} \). Define \( \RS \) with respect to the intervals \( \Delta_i=[\alpha_i,\beta_i] \), \( i\in I_d \). Let \( \boldsymbol h \) be a rational function on \( \RS \) with the zero/pole divisor as described in Proposition~\ref{prop:Ya} where the simple zero in the gap \( \RS_0 \cap \pi^{-1}([\beta_i,\alpha_{i+1}]) \) has natural projection \( z_i \), \( i\in I_d\setminus\{d\} \). Normalize this function to have residue \( 1 \) at the point on top of infinity on \( \RS_0 \). Define measures \( \omega_i \) via \eqref{eq_meas} and let \( c_i=|\omega_i| \), \( i\in I_d \). Then, \( |\vec c|=1 \) and one can use \eqref{future_kappas} to show that \( \omega_i = \omega_{\vec c,i} \), \( i\in I_d \).

The following two facts about measures \( \omega_{\vec c,i} \) are important for the forthcoming analysis. Notice that it readily follows from \eqref{eq_meas} that these measures are absolutely continuous with respect to the Lebesgue measure, i.e., \( d\omega_{\vec c,i}(x) = \omega_{\vec c,i}^\prime(x)dx \).

\begin{proposition}
\label{prop:gaps}
For each \( \vec c\in(0,1)^d, |\vec c|=1 \) and \( i\in I_d \), the density \( \omega_{\vec c,i}^\prime(x) \) is non-vanishing on \( (\alpha_{\vec c,i},\beta_{\vec c,i}) \). Moreover, 
\[
\lim_{ x\uparrow \beta_{\vec c,i} } \omega_{\vec c,i}^\prime(x) (\beta_{\vec c,i}-x)^{\pm 1/2}
\]
exists and is positive and finite, where one needs to take exponent \( 1/2 \) if \( \boldsymbol h_{\vec c} \) has a pole at \( \boldsymbol \beta_{\vec c,i} \) and exponent \( -1/2 \) if \( \boldsymbol h_{\vec c} \) is finite at \( \boldsymbol \beta_{\vec c,i} \), \( i\in I_d \). Furthermore, analogous claims hold at each \( \alpha_{\vec c,i} \).
\end{proposition}
\begin{proof}
Let \( \RS_{\vec c} \) be the Riemann surface defined after Theorem  \ref{thm:GR}. Denote by \( \infty_k \) the point on top of infinity that belongs to \( \RS_{\vec c,k}\), \( k\in\{0,1,\ldots,d \} \). Let \( \chi_{\vec c}:\RS_{\vec c}\to \overline\C \) be the conformal map such that \( \chi_{\vec c,0}(z) = z + \mathcal O(1/z) \) as \( z\to\infty \), where, as before, \( \chi_{\vec c,k}(z) \) is the pull-back to \( \overline\C \) of \( \chi_{\vec c} \) from \( \RS_{\vec c,k} \), \( k\in\{0,1,\ldots,d \} \). That is, \( \chi_{\vec c} \) is a rational function on \(  \RS_{\vec c} \) with a simple pole at \( \infty_0 \) and no constant term in its Laurent expansion there. The specified behavior at \( \infty_0 \) determines \( \chi_{\vec c} \) uniquely (the difference of any two such functions must be a constant function as it has no poles and is analytic everywhere on \( \RS_{\vec c} \); as it vanishes at \( \infty_0 \), this difference is identically zero). The uniqueness immediately yields that
\[
\overline{\chi_{\vec c,k}(\bar z)} = \chi_{\vec c,k}(z), \quad k\in\{0,1,\ldots,d \}.
\]
Thus, the preimage of the real line is the cycle that proceeds along segments of the real line when sheets \( \RS_{\vec c,k} \) are identified with cut planes in the following manner:
\[
\infty_0 \stackrel{\RS_{\vec c,0}}{\to} \boldsymbol\alpha_{\vec c,1} \stackrel{\RS_{\vec c,1}}{\to} \infty_1 \stackrel{\RS_{\vec c,1}}{\to} \boldsymbol\beta_{\vec c,1} \stackrel{\RS_{\vec c,0}}{\to} \boldsymbol\alpha_{\vec c,2} \stackrel{\RS_{\vec c,2}}{\to} \cdots \stackrel{\RS_{\vec c,d}}{\to} \infty_d \stackrel{\RS_{\vec c,d}}{\to} \boldsymbol\beta_{\vec c,d} \stackrel{\RS_{\vec c,0}}{\to} \infty_0.
\]
In particular, we have that \( \pm\im(\chi_{\vec c,0\pm}(x))>0 \) for \( x\in (\alpha_{\vec c,i},\beta_{\vec c,i}) \), \( i\in I_d \). Let
\[
p_{2i-1} := \chi_{\vec c}(\boldsymbol\alpha_{\vec c,i}), \quad p_{2i} := \chi_{\vec c}(\boldsymbol\beta_{\vec c,i}), \quad x_{2i-1} := \chi_{\vec c}(\infty_i), \quad i\in I_d, 
\]
and \( x_{2i} := \chi_{\vec c}(\z_{\vec c,i}) \), \( i\in I_d\setminus\{d\} \). Then, it holds that
\begin{equation}
\label{px_ineq}
p_{2i-1} < x_{2i-1} < p_{2i}, \;\; i \in I_d, \qandq p_{2i} \leq x_{2i} \leq p_{2i+1}, \;\; i \in I_d\setminus\{d\}.
\end{equation}
It also can be readily seen that
\begin{equation}
\label{h_repr}
\boldsymbol h_{\vec c} = \frac{(\chi_{\vec c}-x_1)(\chi_{\vec c}-x_2)\cdots(\chi_{\vec c}-x_{2d-1})}{(\chi_{\vec c}-p_1)(\chi_{\vec c}-p_2)\cdots(\chi_{\vec c}-p_{2d})} = \sum_{i=1}^{2d} \frac{\gamma_i}{\chi_{\vec c}-p_i}
\end{equation}
(this is not the reduced form because \( x_{2i} \) can be equal to either \( p_{2i} \) or \( p_{2i+1} \) as explained in Proposition~\ref{prop:Ya}). Clearly, it holds that
\[
\gamma_k = \prod_{i=1}^{k-1} \frac{p_k-x_i}{p_k-p_i} \prod_{i=k+1}^{2d} \frac{p_k-x_{i-1}}{p_k-p_i} \geq 0,
\]
where the last inequality follows from \eqref{px_ineq} (notice also that \( \sum_{i=1}^{2d}\gamma_i =1 \) as follows from the normalization of \( \boldsymbol h_{\vec c} \) at \( \infty_0 \)). Since \( \boldsymbol h_{\vec c} \) is a rational function on \( \RS_{\vec c} \), it necessarily holds that \( \boldsymbol h_{\vec c,0\pm}(x) = \boldsymbol h_{\vec c,i\mp}(x) \), \( x\in\Delta_{\vec c,i} \) for each \( i\in I_d \). Thus, we get from \eqref{eq_meas} and \eqref{h_repr} that
\[
\pi\omega_{\vec c,i}^\prime(x) = \frac{\boldsymbol h_{\vec c,0-}(x)-\boldsymbol h_{\vec c,0+}(x)}{2\ic} = -\im(\boldsymbol h_{\vec c,0+}(x)) = \sum_{i=1}^{2d} \frac{\gamma_i \im(\chi_{\vec c,0+}(x))}{|\chi_{\vec c,0+}(x)-p_i|^2} >0
\]
for \( x\in(\alpha_{\vec c,i},\beta_{\vec c,i}) \) and each \( i\in I_d \) as claimed. Moreover, we get for each \( i\in I_d \) that
\[
\chi_{\vec c,0}(z) = p_{2i} + q_{2i}(z-\beta_{\vec c,i})^{1/2} + \mathcal O((z-\beta_{\vec c,i}))
\]
for \( z\not\in\Delta_{\vec c,i} \) and sufficiently close to \( \beta_{\vec c,i} \), where we take the principal branch of the square root and \( q_{2i} > 0 \)  since \( \chi_{\vec c} \) is conformal. Then, we get for each \( l\in I_d \) that
\[
\pi\omega_{\vec c,l}^\prime(x) = (\gamma_{2l}/q_{2l}) (\beta_{\vec c,l}-x)^{-1/2} + \mathcal O(1)
\]
if \( \gamma_{2l} \neq 0 \), i.e., \( \boldsymbol z_{\vec c,l} \neq \boldsymbol\beta_{\vec c,l} \), or
\[
\pi\omega_{\vec c,l}^\prime(x) = (\beta_{\vec c,l}-x)^{1/2} \sum_{i=1,i\neq 2l}^{2d} \frac{\gamma_i q_{2l}}{(p_{2l}-p_i)^2} + \mathcal O((\beta_{\vec c,l}-x))
\]
otherwise. Since analogous claims hold at each \( \alpha_{\vec c,l} \), this finishes the proof of the proposition.
\end{proof}

\section{Strong Asymptotics of MOPs}
\label{sec:sa}

We keep all the notation given in the introduction and Section~\ref{sec:wa}.  

\subsection{Uniformity of Szeg\H{o} Measures}

To describe the strong limits of multiple orthogonal polynomials we need to define an  analog of the Szeg\H{o} functions \eqref{SzegoFun}.

\begin{proposition}
\label{prop:sa1}
Let \( \{\Delta_i\}_{i\in I_d} \) be a collection of pairwise disjoint closed intervals and \( \{ \mu_i \}_{i\in I_d} \) be positive Borel measures such that \( \mu_i\in\mathrm{Sz}(\Delta_i) \) for each \( i\in I_d \). There exists a unique collection of functions \( \{S_i(z)\}_{i\in I_d} \) such that each \( S_i(z) \) is a conjugate-symmetric outer function in \( H^2(D_{\Delta_i}) \) with \( S_i(\infty)>0 \) and it holds that
\begin{equation}
\label{SzegoonRiemann}
|S_{i\pm}(x)|^2 \prod_{j\in I_d,j\neq i}S_j(x) = v_{\Delta_i}(x)  \quad \text{for a.e.} \quad x\in\Delta_i, \quad i\in I_d,
\end{equation}
where \( v_{\Delta_i}(x) \) is the Radon-Nikodym derivative of \( \mu_i \) w.r.t. \( \omega_{\Delta_i} \), see \eqref{szego}.
\end{proposition}

We point out that for absolutely continuous measures \( d\mu_i = \mu_i^\prime(x)dx \) with smooth non-vanishing densities \( \mu_i^\prime(x) \), the integral representation of functions \( S_i(z) \) was obtained in \cite{ap11}, see also \cite[Section~6]{Ya16}. This representation is similar to formulae  \eqref{outer}, \eqref{SzegoFun}  and it uses a Cauchy-like kernel on the Riemann surface \( \RS \) that corresponds to \( (\Delta_1,\Delta_2,\ldots,\Delta_d) \).

To account for the pushing effect, we further strengthen the notion of a Szeg\H{o} measure.  We shall say that a measure \( \mu \) is \emph{uniformly Szeg\H{o}} on a closed interval \( \Delta \), and denote this by \( \mu\in\mathrm{USz}(\Delta) \), if \( \mu\in\mathrm{Sz}(\Delta) \) and for any sequence of closed intervals \( \{\Delta_n\} \) such that \( \Delta_n\subseteq \Delta(\mu) \) and \( \Delta_n\to\Delta \) as \( n\to\infty \), there is $n_0$ such that \( \mu\in\mathrm{Sz}(\Delta_n) \) for $n\ge n_0$ and 
\begin{equation}
\label{UnSzego}
\lim_{n\to\infty}\int \big|\log \mu^\prime(x) - \log \mu^\prime(l_{\Delta\to\Delta_n}(x)) \big| d\omega_\Delta(x) \to 0 \,,
\end{equation}
where \( l_{\Delta\to\Delta_n}(x) \) is a linear function with a positive leading coefficient that maps \( \Delta \) onto \( \Delta_n \).  We call the class ``uniform Szeg\H{o}'' to emphasize that small perturbations of the endpoints of \( \Delta \) result in small changes of the value of the Szeg\H{o} function at infinity. In fact, this is true for the whole Szeg\H{o} function locally uniformly away from \( \Delta \).

\begin{proposition}
\label{prop:sa2}
Let \( \{\Delta_i\}_{i\in I_d} \) be a collection of pairwise disjoint closed intervals and \( \{\mu_i\}_{i\in I_d} \) be a collection of positive Borel measures such that \( \mu_i\in \mathrm{USz}(\Delta_i) \) for each \( i\in I_d \). Further, let \( \{\Delta_{n,i}\}_{i\in I_d} \), \( n\in \N \), be collections of pairwise disjoint closed intervals such that \( \Delta_{n,i}\subseteq \Delta(\mu_i) \) and \( \Delta_{n,i}\to\Delta_i \) as \( n\to\infty \) for each \( i\in I_d \). Then
\begin{equation}
\label{SzegoCont}
S_{n,i}(z) \to S_i(z) \qasq  n\to\infty
\end{equation}
locally uniformly in \( D_{\Delta_i} \) for each \( i\in I_d \), where \( \{S_i(z)\}_{i\in I_d} \) and \( \{ S_{n,i}(z)\}_{i\in I_d} \) are the collection of functions guaranteed by Proposition~\ref{prop:sa1} for \( \{\Delta_i\}_{i\in I_d} \) and \( \{\Delta_{n,i}\}_{i\in I_d} \), respectively.
\end{proposition}

Since the concept of uniformly Szeg\H{o } measures is important to our analysis, let us provide a different characterization of this class. To this end, given an integrable function \( \theta \) on an interval \( \Delta \), we let
\[
(I_\gamma\theta)(x) :=  \frac1{\sqrt\pi}\int_\gamma^x \frac{\theta(t)dt}{\sqrt{|x-t|}}, \quad x\in\Delta,
\]
where \( \gamma\in\Delta \) is fixed, which is a version of the so-called Riemann-Liouville fractional integral (corresponding to the exponent \( 1/2 \)). As an integral transform, \( I_\gamma \) is a continuous operator from \( L^1(\Delta) \) into weak-\( L^2 (\Delta) \), see \cite[Lemma~2.13]{MR4191495}. Notice also that given two different $\gamma_1$ and $\gamma_2$ such that $\gamma_1,\gamma_2\in \Delta, \gamma_1<\gamma_2$, the difference
\[
(I_{\gamma_1}\theta)(x)-(I_{\gamma_2}\theta)(x)=\frac{1}{\sqrt \pi}\int_{\gamma_1}^{\gamma_2} \frac{\theta(t)dt}{\sqrt{|x-t|}}
\]
 is  continuous in $x$ on $\R\backslash [\gamma_1,\gamma_2]$.
\begin{proposition}
\label{prop:sa3}
Let \( [\alpha,\beta]=\Delta \subseteq \Delta(\mu) \). The following are equivalent
\begin{itemize}
\item[(i)] \( \mu\in\mathrm{USz}(\Delta) \);
\item[(ii)] \( (I_\gamma\log\mu^\prime)(x) \), \( x\in\Delta(\mu) \), is continuous at \( \beta \) and \( \alpha \), where \( \gamma \in(\alpha,\beta) \) is any;
\item[(iii)] for any \( \epsilon>0 \) there exists \( d_\epsilon>0 \) such that 
\[
\left| \int_{a_\alpha}^{b_\alpha}\frac{\log \mu^\prime(t) dt}{\sqrt{t-a_\alpha}} \right|,\left| \int_{a_\beta}^{b_\beta}\frac{\log \mu^\prime(t) dt}{\sqrt{b_\beta-t}} \right|<\epsilon
\]
when \( \dist(\alpha,[a_\alpha,b_\alpha]),\dist(\beta,[a_\beta,b_\beta])<d_\epsilon \).
\end{itemize}
\end{proposition}

If \( \theta\in L^p(\Delta) \) for some \( p>2 \), then \( I_\gamma\theta \)  is H\"older continuous on \( \Delta \) with exponent at least \( 2-1/p \), see \cite[Theorem~12]{MR1544927}. Since \( \mu^\prime \) is an integrable function, \( \log^+\mu^\prime \) is in \( L^p(\Delta) \) for any \( p>2 \). Hence, \( I_\gamma\log^+\mu^\prime \) is necessarily H\"older continuous and therefore Proposition~\ref{prop:sa3} could be equivalently stated with \( \log\mu^\prime \) replaced by either \( |\log\mu^\prime| \) or \( \log^-\mu^\prime \).
 
\subsection{Main Theorem} 

Recall definitions \eqref{ortho}--\eqref{Nc} as well as \eqref{vec_eq1}--\eqref{vec_eq2}. Below, we label simply by \( \n \) the quantities usually labeled by \( \vec  c \) when they are referred to with the value of the parameter \( \vec c \) being \( \n/|\n| \).

\begin{theorem}
\label{thm:sa4}
Let  \( \{P_\n(z)\}_\n \) be the table of multiple orthogonal polynomials with respect to an Angelesco system of measures \( \vec \mu \). Fix a non-marginal ray sequence in \( \mathcal N(\vec c) \). Assume that \( \mu_i \in \mathrm{USz}(\Delta_{\vec c,i}) \), \( i\in I_d \).  Then, it holds for each \( i\in I_d \) that
\[
P_{\n,i}(z) = (1+o(1)) \exp\left( |\n| \int\log(z-x)d\omega_{\n,i}(x)\right) \frac{S_{\vec c,i}(\infty)}{S_{\vec c,i}(z)} 
\]
locally uniformly in \( D_{\Delta_{\vec c,i}} \) for all \( |\n| \) large enough, where \( \{S_{\vec c,i}(z)\}_{i\in I_d} \) is the collection of functions guaranteed by Proposition~\ref{prop:sa1} for the collection of intervals \( \{ \Delta_{\vec c,i} \}_{i\in I_d} \).
\end{theorem}

Let us point out that if \( \mathcal N(\vec c) \) is such that \( \Delta_{\n,i} = \Delta_{\vec c,i} \), \( i\in I_d \), for all \( \n\in\mathcal N(\vec c) \) with \( |\n| \) large enough, then we can simply require that \( \mu_i \in \mathrm{Sz}(\Delta_{\vec c,i}) \). This holds whenever, e.g., \( \mathcal N(\vec c)=\{n\cdot\vec m:n\in\N\}\) with \( \vec c = \vec m/|\vec m| \) for some \( \vec m\in\N^d \), or when \( \vec c \) belongs to a relatively open subset of \( \{\vec c\in(0,1)^d:|\vec c|=1\} \) for which \( \Delta_{\vec c,i} = \Delta_i \), \(i\in I_d \). This set is known to be non-empty which follows from the paragraph after Proposition~\ref{prop:Ya}. 

When \( d=2 \), \( \vec c=(1/2,1/2) \), and \( \mathcal N(\vec c)=\{(n,n):n\in\Z_+\} \), the result of Theorem~\ref{thm:sa4} is contained in \cite{Ap88}. We also want to mention the work  \cite{ap4} where the strong asymptotics was obtained for two touching intervals and the Jacobi weights. Theorem~\ref{thm:sa4} also generalizes the results in \cite{Ya16}, where the measures \( \mu_i \) were  absolutely continuous and their Radon-Nikodym derivatives with respect to the Lebesgue measure  were assumed to be Fisher-Hartwig perturbations of functions non-vanishing and analytic around the corresponding intervals. We also mention our earlier work \cite{ApDenYa21}, where in the case of two absolutely continuous measures with analytic non-vanishing derivatives the strong asymptotics was derived along all ray sequences including the marginal cases \( \vec c=(0,1) \) and \( \vec c=(1,0) \). In the forthcoming work \cite{uYa1}, which is a continuation of \cite{ApDenYa21}, the error estimates are shown to be uniform in \( \n \) with explicit bounds on their rate of decay.

Our approach relies on two improvements of a theorem by Totik \cite[Theorem~14.4]{Totik} on the asymptotic behavior of orthogonal polynomials with varying weights. We discuss these generalizations in Section~\ref{sec:vw}.

In conclusion, we mention that finding strong asymptotics of MOPs is required in the theory of Jacobi matrices on trees  \cite{ser2,ser3,ApDenYa21}, the study of simultaneous Gaussian quadrature \cite{ser1}, random matrices \cite{ap12}, asymptotcs of special Hankel determinants, and other areas. We will address some of these applications in subsequent papers.

\subsection{Proofs of Propositions~\ref{prop:sa1} and~\ref{prop:sa2}}

In what follows, we switch to a vector notation. Given \( \vec \Delta = (\Delta_1,\Delta_2, \ldots ,\Delta_d ) \), a collection of closed pairwise disjoint intervals, we put
\[
L^1(\omega_{\vec \Delta}):= L^1(\omega_{\Delta_1})\times\cdots\times  L^1(\omega_{\Delta_d}),
\]
and define the spaces of continuous \( C(\vec\Delta) \) and continuously differentiable \( C^1(\vec\Delta) \) vector-functions similarly. We shall use vector notation \( \vec u \) to denote elements of these spaces and write either \( u_i \) or \( (\vec u)_i \), whichever is more convenient, for the \( i \)-th component of the vector~\( \vec u \).

Throughout the paper we write \( a_n\lesssim b_n \) or \( f_n(x) \lesssim g_n(x) \), \( x\in K \), if there exists a constant \( C>0 \), independent of \( n \) but possibly dependent on some other relevant parameters, such that \( a_n\leq C b_n \) or \( f_n(x) \leq C g_n(x) \), \( x\in K \). If we want to emphasize that  \( C \) does depend on some quantity \( Q \), we shall write \( \lesssim_Q \).

Given a (real-valued) function \( u \) in \( L^1(\omega_\Delta) \), we set
\begin{equation}
\label{Dirichlet}
(H_\Delta u)(z) := \log|\out_\Delta(e^u,z)|, \quad z\in D_\Delta = \overline\C\setminus\Delta,
\end{equation}
where \( \out_\Delta \) was introduced in \eqref{outer}. Then, \( H_\Delta u \) is a harmonic function in \( D_\Delta \) whose non-tangential boundary values from above and below \( \Delta \) exist almost everywhere and are equal to \( u \), see \eqref{outer-b}. That is, \( H_\Delta u \) is a solution of the Dirichlet problem in \( D_\Delta \) with boundary data~\( u \). When \( u\in C(\Delta) \), \( (H_\Delta u)(z) \) is in fact continuous in the entire extended complex plane, see \cite[Corollary~4.1.8 and Theorem~4.2.1]{Ransford}. Next, let
\[
H_{\Delta_j\to\Delta_i} : L^1(\omega_{\Delta_j}) \to L^1(\omega_{\Delta_i}), \quad u\mapsto (H_{\Delta_j}u)_{|\Delta_i}
\]
for \( i\neq j \), \( i,j\in I_d \). For convenience, define \( H_{\Delta_i\to\Delta_i} \) to be the operator that sends every function into the zero function. Define
\begin{equation}
\label{operH}
\mathcal H_{\vec \Delta} := -\frac12\begin{pmatrix} H_{\Delta_j\to\Delta_i} \end{pmatrix}_{i,j=1}^d: L^1(\omega_{\vec\Delta}) \to L^1(\omega_{\vec\Delta}),
\end{equation}
where \( i \) is the row index and \( j \) is the column one giving the matrix form
\[
\mathcal H_{\vec \Delta}=-\frac 12\left[
\begin{array}{cccc}
0& H_{\Delta_2\to \Delta_1}&\cdots&H_{\Delta_n\to\Delta_1}\\
H_{\Delta_1\to\Delta_2}&0&\cdots&H_{\Delta_n\to\Delta_2}\\
\vdots&\vdots&\ddots&\vdots\\
H_{\Delta_1\to\Delta_n}&H_{\Delta_2\to\Delta_n}&\cdots&0
\end{array}
\right]\,.
\]
Since harmonic functions are infinitely smooth, it in fact holds that \( \mathcal H_{\vec \Delta}(L^1(\omega_{\vec\Delta})) \subset C^1(\vec\Delta) \). 

We also need to introduce certain modifications of the operators \( \mathcal H_{\vec \Delta}\). These modifications are not important for the proofs of Propositions~\ref{prop:sa1} and \ref{prop:sa2}, but will be indispensable in the proof of Theorem~\ref{thm:sa4} when we work with ray sequences exhibiting the pushing effect. To this end, let
\[
\vec\Delta^* \subseteq \vec\Delta \quad (\Delta_i^*\subseteq\Delta_i, ~~ \forall i\in I_d), \quad \vec\Delta^* = (\Delta_1^*,\Delta_2^*,\ldots,\Delta_d^*),
\]
be a vector of closed intervals. For each pair \( \Delta^*\subseteq\Delta \), denote by \( R_{\Delta\to\Delta^*} \) the restriction operator of a function on \( \Delta \) to a function on \( \Delta^* \). Put
\begin{equation}
\label{operHrest}
\mathcal H_{\vec \Delta^*,\vec \Delta} := -\frac12\begin{pmatrix} H_{\Delta_j^*\to\Delta_i} \circ R_{\Delta_j\to\Delta_j^*}\end{pmatrix}_{i,j=1}^d: C(\vec\Delta) \to C(\vec\Delta).
\end{equation}
That is, \( \mathcal H_{\vec \Delta^*,\vec \Delta}\vec u \) is obtained by first restricting \( \vec u \) to \( \vec\Delta^* \), then applying the same harmonic extension operators as in the case of the operator \( \mathcal H_{\vec\Delta^*} \), and finally restricting these extensions to \( \vec\Delta \) (and not \( \vec\Delta^* \) as in the case of  \( \mathcal H_{\vec\Delta^*} \)). We consider \(  \mathcal H_{\vec \Delta^*,\vec \Delta} \) only as an operator on continuous functions because a restriction of \( \vec u \in L^1(\omega_{\vec\Delta}) \) to \( \vec\Delta^* \) might not lie in \( L^1(\omega_{\vec\Delta^*}) \).

\begin{lemma}
\label{lem:sa5}
If \( \mathcal H_{\vec \Delta}\vec u = \vec u \), \( \vec u \in L^1(\omega_{\vec\Delta}) \), then \( \vec u = \vec 0 \), where \( \vec 0 \) is the zero vector-function. The same conclusion holds in the case of \( \mathcal H_{\vec \Delta^*,\vec \Delta} \) on \( C(\vec\Delta) \).
\end{lemma}
\begin{proof}
Set \( h_i(z) := (H_{\Delta_i}u_i)(z) \), \( i\in I_d \). Since \( \vec u \) lies in the image of \( \mathcal H_{\vec \Delta} \), each \( u_i \) is necessarily a continuous function on \( \Delta_i \) and therefore \( h_i(z) \) is not only harmonic in \( D_{\Delta_i} \), but is also continuous in the whole extended complex plane (and is equal to \( u_i \) on \( \Delta_i \)). Thus, \( \mathcal H_{\vec\Delta}\vec u = \vec u \) can be rewritten as
\[
h_i(x) + \frac12 \sum_{j\in I_d,j\neq i} h_j(x) = 0, \quad x\in\Delta_i, \quad i\in I_d.
\]
This equality, conjugate-symmetry of each \( h_l(z) \), \( l\in I_d \), and the Schwarz reflection principle allow us to conclude that for each \( i\in I_d \) the function
\begin{equation}
\label{apsad1}
\mathcal{R}_i(z) := \begin{cases}
h_i(z) + \frac12 \sum_{j\in I_d,j\neq i} h_j(z), & z\in \mathbb{C}^+:=\{z:\im(z)>0\},  \medskip \\
-h_i(z) - \frac12 \sum_{j\in I_d,j\neq i} h_j(z), & z\in \mathbb{C}^- :=\{z:\im(z)<0\},
\end{cases}
\end{equation}
is not only harmonic in the upper and lower half-planes but is also harmonic across \( \Delta_i \) when extended to \( \Delta_i \) by zero.  Since the sum $\sum_{j\in I_d,j\neq i} h_j(z)$ is harmonic in  \(\C\setminus(\cup_{j\neq i} \Delta_j) \) and
\[
h_i(z) = \mathcal R_i(z) - \frac12 \sum_{j\in I_d,j\neq i} h_j(z), \quad z\in\C^+,
\]
this formula provides a harmonic continuation of \( h_i(z) \) from \( \C^+ \) to \( \C^- \) across \( \Delta_i \). It readily follows from the second equation in \eqref{apsad1} that this continuations coincides in \( \C^- \) with 
\[
-\sum_{j\in I_d} h_j(z) =: h_0(z), \quad z\in \C\setminus (\cup_{i\in I_d} \Delta_i).
\]
Notice that $h_0(z)$ does not depend on $i$ in our argument. In a similar way, one can show that $h_i(z)$ can be harmonically extended from $\C^-$ to $\C^+$ across $\Delta_i$ resulting in the same $h_0(z)$.

Let \( \RS \) be the Riemann surface introduced before Proposition~\ref{prop:Ya}. Set \( h \) to be a function on \( \RS \) that is equal to \( h_k(z) \) on \( \RS_k \), \( k\in\{0,1,\ldots,d\} \). It follows from the arguments above that \( h \) is a global harmonic function on \( \RS \) and therefore must be a constant, see \cite[Corollary~19.7]{Forster}. From the definition of \( h_0(z) \) it easily follows that this constant is zero.

Now, if \( \mathcal H_{\vec \Delta^*,\vec \Delta}\vec u = \vec u \) for some \( \vec u\in C(\vec\Delta) \), then \( \mathcal H_{\vec \Delta^*}\vec u^* = \vec u^* \), where \( \vec u^* \) is the restriction of \( \vec u \) to \( \Delta^* \). Hence, \( \vec u^* \) is the zero vector. Since \( \mathcal H_{\vec \Delta^*,\vec \Delta} \) depends only on \( \vec u^* \), \( \mathcal H_{\vec \Delta^*,\vec \Delta}\vec u \) is the zero vector as well, which finishes the proof of the lemma.
\end{proof}

In what follows, we set \( \mathcal I \) to be the identity operator on any considered space. In the next lemma  \( \mathfrak B \)  stands for either \( L^1(\omega_{\vec \Delta}) \), \(  C(\vec\Delta) \), or \( C^1(\vec\Delta) \). 

\begin{lemma}
\label{lem:sa6}
\( \mathcal H_{\vec \Delta} \) is a linear bounded compact operator on \( \mathfrak B \) while \( \mathcal I-\mathcal H_{\vec \Delta} \) is invertible on \( \mathfrak B \). In particular, for any \( \vec a\in \mathfrak B \), \( (\mathcal I-\mathcal H_{\vec \Delta})^{-1} \vec a \) is the unique solution of \( \vec u = \mathcal H_{\vec \Delta}\vec u + \vec a \). Similarly, \( \mathcal H_{\vec \Delta^*,\vec \Delta} \) is a linear bounded compact operator on \( \mathfrak B\) while \( \mathcal I-\mathcal H_{\vec \Delta^*,\vec \Delta} \) is invertible on \( \mathfrak B\) (here, we only speak of continuous spaces).
\end{lemma}
\begin{proof}
Let \( u \in L^1(\omega_\Delta) \). Explicit expression \eqref{outer} shows that for any closed subset \( K\subset D_\Delta \) there exists a constant \( C_K \) such that
\[
|(H_\Delta u)(z)| \leq C_K \|u\|, \quad z\in K,
\]
where we take the \( L^1 \)-norm of \( u \). Since \( L^1 \)-norm is dominated by the uniform norm, which is dominated by the \( C^1 \)-norm, the above inequality remains valid if \( u \) belongs to \( C(\Delta) \) or \( C^1(\Delta) \) with the norm coming from the corresponding space. Let \( U_i \), \( i\in I_d \), be pairwise disjoint open sets such that \( \Delta_i \subset U_i \). If \( \vec u=(u_1,u_2,\ldots,u_d)\in\mathfrak B \), then it follows from the first observation of the lemma that
\begin{equation}
\label{Bnorm_est}
\left|\sum_{j\in I_d,j\neq i}(H_{\Delta_j}u_j)(z)\right| \leq C_{U_i} \big\|\vec u \big\|_\mathfrak{B}, \quad z\in U_i,
\end{equation}
for some constant \( C_{U_i} \) and each \( i\in I_d \). Notice that when restricted to \( \Delta_i \) the expression inside the absolute value above is equal to the \( i \)-th component of \( -2(\mathcal H_{\vec \Delta}\vec u) \). This expression is also a harmonic function on \( U_i \). According to \cite[Theorem~1.3.10]{Ransford}, any  sequence of harmonic functions, which is uniformly bounded above and below, has a locally uniformly convergent subsequence. On \( \Delta_i \) this convergence takes place in particular in \( C^1(\Delta_i) \). Hence, \( \mathcal H_{\vec \Delta} \) is a compact linear operator from  \( \mathfrak B \) into itself. Let now \( \mathfrak{B} \) be either \( C(\vec\Delta) \) or \( C^1(\vec\Delta) \). Notice that \eqref{Bnorm_est} remains true if we replace operators \( H_{\Delta_j} \) by \( H_{\Delta_j^*}\circ R_{\Delta_j\to\Delta_j^*} \). Hence, we can similarly conclude that \( \mathcal H_{\vec \Delta^*,\vec \Delta} \)  is a compact linear operator from  \( \mathfrak B \) into itself.

According to the Fredholm theory of compact operators, see \cite[Section~7.11]{MR1009162}, \( \mathcal I-\mathcal H_{\vec \Delta} \) (resp. \( \mathcal I-\mathcal H_{\vec\Delta^*,\vec \Delta} \)) is invertible if and only if \( 1 \) is not an eigenvalue of \( \mathcal H_{\vec \Delta} \) (resp. \( \mathcal H_{\vec\Delta^*,\vec \Delta} \)), which was proven in Lemma~\ref{lem:sa5}. This finishes the proof of the lemma.
\end{proof}

\begin{lemma}
\label{lem:sa7}
Proposition~\ref{prop:sa1} takes place. Moreover,
\begin{equation}
\label{apsad2}
S_i(z) = \out_{\Delta_i}(e^{s_i},z), \quad i\in I_d,
\end{equation}
where \( \vec s = (s_1,s_2,\ldots,s_d) = (\mathcal I -\mathcal H_{\vec \Delta})^{-1}\vec a \) for
\[
\vec a = \frac12\big( \log v_1, \log v_2, \ldots, \log v_d \big) \in L^1(\omega_{\vec \Delta}),
\]
and we set, for simplicity, \( v_i(x) := v_{\Delta_i}(x) \), \( i\in I_d \).
\end{lemma}
\begin{proof}
We get from Lemma~\ref{lem:sa6} that \( \vec s\in L^1(\omega_{\vec\Delta}) \). Define functions $S_i(z)$ by \eqref{apsad2}. First, we show that $\{S_i\}_{i\in I_d}$ satisfy conditions of Proposition~\ref{prop:sa1}. From our discussion of formulae \eqref{outer} and \eqref{outer-b}, one concludes that \( S_i(z)\) is necessarily outer, conjugate-symmetric, $S_i(\infty)>0$, and it satisfies \( \log|S_{i\pm}(x)| = s_i(x) \) for almost every \( x\in\Delta_i \), \( i\in I_d \). To show that it belongs to the Hardy space, we use
\begin{equation}
\label{sapsad20}
e^{2s_i(x)} = v_i(x)e^{2(\mathcal H_{\vec\Delta}\vec s)_i(x)} \in L^1(\omega_i), \quad i\in I_d,
\end{equation}
where the inclusion holds because the image of \( \mathcal H_{\vec\Delta} \) lies in \( C(\vec\Delta) \). Therefore, $S_i\in H^2(D_{\Delta_i})$.

Clearly,  \( S_i(x)>0 \) for \( x\in\R\setminus\Delta_i \). We also see from \eqref{Dirichlet} that 
\begin{equation}\label{sapsad15}
 \log |S_i(z)| = (H_{\Delta_i} s_i)(z)\,.
 \end{equation} Hence, the equation \( \vec s = \mathcal H_{\vec \Delta}\vec s + \vec a \) can now be rewritten as 
\[
\log|S_{i\pm}(x)| = \frac12\left( \log v_i(x) - \sum_{j\in I_d,j\neq i} \log S_j(x)\right)
\]
for almost every \( x\in \Delta_i \), \( i\in I_d \). Exponentiation then readily yields \eqref{SzegoonRiemann}.

Conversely, let \( Q_i(z) \), \( i\in I_d \), be outer conjugate-symmetric functions with \( Q_i(\infty)>0 \) that satisfy \eqref{SzegoonRiemann} with \( S_i \) replaced by \( Q_i \). Then, by taking logarithms we get that
\[
\frac12\left(\log v_i(x)-\sum_{j\in I_d,j\neq i} \log |Q_j(x)| \right) =  \log |Q_{i\pm}(x)|
\]
for almost every \( x\in\Delta_i \), \( i\in I_d \). Moreover, we readily get from \eqref{outer}, \eqref{outer-b}, and \eqref{Dirichlet} that \( \log |Q_i(z)|  = (H_{\Delta_i} \log|Q_{i+}|)(z) \) and therefore the equation \( \vec u = \mathcal H_{\vec \Delta}\vec u + \vec a \) is solved by the vector-function \( (\log|Q_{1+}|,\log |Q_{2+}|,\ldots,\log|Q_{d+}|) \). Hence, this vector-function must be equal to \( \vec s \) by Lemma~\ref{lem:sa6}, which means that \( Q_i(z) =S_i(z) \), \( i\in I_d \), as desired.
\end{proof}

\begin{lemma}
\label{lem:sa8}
Let \( \{\Delta_n\} \) be a sequence of closed intervals converging to a non-degenerate interval \( \Delta \). Further, let \( \mu\in\mathrm{USz}(\Delta) \) and \( \{h_n\} \) be a  sequence of continuous functions that converges uniformly on some closed interval that contains \( \Delta \) in its interior. Then
\[
\big(H_{\Delta_n}(\log v_n+h_{n|\Delta_n})\big)(z) \to \big(H_\Delta (\log v+h_{|\Delta})\big)(z)
\]
locally uniformly in \( D_\Delta \) as \( n\to\infty \), where \( v_n \) and \( v \) are the Radon-Nikodym derivatives of \( \mu_{|\Delta_n} \) and \( \mu_{|\Delta} \) with respect to \( \omega_{\Delta_n} \) and \( \omega_\Delta \), respectively, and \( h \) is the limit of \( h_n \).
\end{lemma}
\begin{proof}
Let \( l_n(x) = l_{\Delta\to\Delta_n}(x) \) be as in \eqref{UnSzego}. Observe that \( l_n(x) \) converge to \( x \) uniformly on \( \Delta \). Set \( u_n(x) := h_{n|\Delta_n}(x)+\log v_n(x) \) and \( u(x):=h_{|\Delta}(x)+\log v(x) \). Then
\[
(H_{\Delta_n}u_n)(z) - (H_\Delta u)(z) = \int_\Delta \re\left(w_{\Delta_n}(z)  \frac{u_n(l_n(x))}{z-l_n(x)}- w_\Delta(z)\frac{u(x)}{z-x}\right) d\omega_\Delta(x),
\]
see  \eqref{outer} and \eqref{Dirichlet}. The function in parenthesis above can be rewritten as
\begin{multline*}
(w_{\Delta_n}(z)-w_\Delta(z))\frac{u(x)}{z-x} +w_{\Delta_n}(z)\frac{u_n(l_n(x))-u(x)}{z-l_n(x)} + \\   w_{\Delta_n}(z)\frac{(l_n(x)-x)u(x)}{(z-l_n(x))(z-x)} =: (J_{n,1}+J_{n,2}+J_{n,3})(x,z).
\end{multline*}
Observe that the functions \( w_{\Delta_n}(z) - w_\Delta(z) \) converge to zero uniformly in the whole extended complex plane. Hence,
\begin{equation}
\label{J123}
\int_\Delta J_{n,1}(x,z)d\omega_\Delta(x) \to 0  
\end{equation}
locally uniformly in \( D_\Delta \) as \( n\to\infty \). Furthermore, since
\[
v_n(x) = \pi\mu^\prime(x)\sqrt{(x-\alpha_n)(\beta_n-x)}, \quad x\in\Delta_n=[\alpha_n,\beta_n],
\]
it can be readily checked that \( u_n(l_n(x))-u(x) \) is equal to
\[
\big[h_n( l_n(x)) - h(x) \big] + \frac12\log \frac{\beta_n-\alpha_n}{\beta-\alpha} +  \big[\log\mu^\prime(l_n(x)) - \log\mu^\prime(x) \big]
\]
on \( \Delta \), where \( \Delta=[\alpha,\beta] \). Due to uniform convergence of \( h_n \) to \( h \), uniform continuity of \( h \) on \( \Delta \), and \eqref{UnSzego}, the functions \( u_n\circ l_n - u \) converge to zero in \( L^1(\omega_\Delta) \). As functions \( |w_{\Delta_n}(z)/(z-l_n(x))| \) are uniformly bounded for \( x\in[-1,1] \) and \( z \) on closed subsets of \( D_\Delta \), this necessarily yields that \eqref{J123} holds with  \( J_{n,1}(x,z) \) is replaced by \( J_{n,2}(x,z) \). Uniform convergence to zero of \( l_n(x)-x \)  on \( \Delta \) now guarantees that \eqref{J123} remains valid if \( J_{n,1}(x,z) \) is replaced by \( J_{n,3}(x,z) \), which finishes the proof of the lemma.
\end{proof}

\begin{lemma}
\label{lem:sa9}
Let \( \vec\Delta_n=(\Delta_{n,1},\Delta_{n,2},\ldots,\Delta_{n,d}) \) be as in Proposition~\ref{prop:sa2} and \( \mathfrak B_n \) be either \( C(\vec \Delta_n) \) or \( C^1(\vec \Delta_n) \). Then, 
\[
\big\|(\mathcal I-\mathcal H_{\vec \Delta_n})^{-1} \big\|_{\mathfrak B_n} \leq C
\]
for some constant \( C \) independent of \( n \). Further, let \( \vec\Delta^\prime =(\Delta_1^\prime,\Delta_2^\prime,\ldots,\Delta_d^\prime) \) be such that \( \vec\Delta_n\subseteq \vec\Delta^\prime \) and the intervals \( \Delta_i^\prime \), \( i\in I_d \), are pairwise disjoint. Let \( \mathfrak B^\prime \) be either \( C(\vec \Delta^\prime) \) or \( C^1(\vec \Delta^\prime) \). Then
\[
\big\|(\mathcal I-\mathcal H_{\vec \Delta_n,\vec\Delta^\prime})^{-1} \big\|_{\mathfrak B^\prime} \leq C^\prime
\]
for some constant \( C^\prime \) independent of \( n \).
\end{lemma}
\begin{proof}
We start by proving the first claim.  Using the notation of Proposition~\ref{prop:sa2}, let \( l_{n,i}(z) := l_{\Delta_{n,i}\to\Delta_i}(z) \) be the linear function that maps \( \Delta_{n,i} \) onto \( \Delta_i \) that has positive leading coefficient, \( i\in I_d\), and let
\[
\mathcal L_n : \mathfrak B \to \mathfrak B_n, \quad \vec u \mapsto ( u_1\circ l_{n,1}, u_2\circ l_{n,2},\ldots, u_d\circ l_{n,d} ),
\]
where \( \vec u =(u_1,u_2,\ldots,u_d) \) and \( \mathfrak B \) is either \( C(\vec\Delta) \) or \( C^1(\vec\Delta) \). Then, \( \mathcal L_n \) is an isometry of \( C(\vec \Delta_n) \) or an operator on \( C^1(\vec\Delta) \) with the norm of size \( 1+o(1) \) as \( n\to\infty \). Define
\[
\widetilde {\mathcal H}_n := \mathcal L_n^{-1} \circ \mathcal H_{\vec\Delta_n} \circ \mathcal L_n.
\]
It is sufficient to prove the statement of the lemma with \( \mathcal H_{\vec\Delta_n} \) replaced by \(  \widetilde {\mathcal H}_n \)  because
\[ 
\|(\mathcal I-\mathcal H_{\vec\Delta_n})^{-1}\|_{\mathfrak B_n} = (1+o(1)) \|(\mathcal I-\widetilde{\mathcal H}_n)^{-1}\|_{\mathfrak B},
\]
where we have actual equality of norms if \( \mathfrak B_n = C(\vec\Delta_n) \). It trivially holds that
\[
\mathcal R_n (\mathcal I-\widetilde{\mathcal H}_n)^{-1} = (\mathcal I-\mathcal H_{\vec \Delta})^{-1}, \quad \mathcal R_n:=  \mathcal I - (\mathcal I-\mathcal H_{\vec \Delta})^{-1}(\widetilde {\mathcal H}_n-\mathcal H_{\vec \Delta}).
\]
Thus, it is enough to show that the operators \( \mathcal R_n \) are invertible and the norms of their inverses are uniformly bounded. To this end, it is sufficient to show that
\[ 
\|  \widetilde {\mathcal H}_n - \mathcal H_{\vec \Delta} \|_{\mathfrak B} \to 0 \qasq  n\to\infty.
\]
First, let $\mathfrak B=C(\vec\Delta)$. The specific form \eqref{operH} of these operators yields that the above claim will follow if we prove that
\begin{equation}
\label{upperB1}
\big\|(H_{\Delta_{n,i}}(u\circ l_{n,i})) \circ l_{n,j}^{-1} - H_{\Delta_i}u \big\|_{C(\Delta_j)} \leq \epsilon_n \|u\|_{C(\Delta_i)}
\end{equation}
for each  \( i\neq j \), \( i,j\in I_d \), with \( \epsilon_n\to 0 \) as \( n\to\infty \). It follows from \eqref{outer} and \eqref{Dirichlet} that we need to estimate the supremum norm of the following function in $y$
\begin{multline*}
\int_{\Delta_{n,i}} w_{\Delta_{n,i}}(l_{n,j}^{-1}(y)) \frac{u(l_{n,i}(x))}{l_{n,j}^{-1}(y)-x} d\omega_{\Delta_{n,i}}(x) - \int_{\Delta_i} w_{\Delta_i}(y)\frac{u(x)}{y-x} d\omega_{\Delta_i}(x) \\ = \int_{\Delta_i} \left(\frac{w_{\Delta_{n,i}}(l_{n,j}^{-1}(y))}{l_{n,j}^{-1}(y)-l_{n,i}^{-1}(x)} - \frac{w_{\Delta_i}(y)}{y-x}\right) u(x)d\omega_{\Delta_i}(x)
\end{multline*}
on \( \Delta_{n,j} \), \( j\neq i \) (we have removed the reference to the real part as the integrals are real for the considered values of \( y \)). As \( \omega_{\Delta_i} \) is a probability measure, we  put
\begin{equation}
\label{eps_n}
\epsilon_n := \max_{i\neq j} \max_{y\in \Delta_j} \max _{x\in \Delta_i} \left| \frac{w_{\Delta_{n,i}}(l_{n,j}^{-1}(y))}{l_{n,j}^{-1}(y)-l_{n,i}^{-1}(x)} - \frac{w_{\Delta_i}(y)}{y-x} \right|
\end{equation}
to get the  desired bound. 
Since $\{\Delta_l\}_{l\in I_d}$ are pairwise disjoint and both \( l_{n,i}(z)\) and \(l_{n,j}(z) \) converge to \( z \) locally uniformly in \( \C \), it  holds that \( \epsilon_n \to 0 \) as \( n\to\infty \), as claimed. 

When $\mathfrak B=C^1(\vec\Delta)$, it is sufficient to show that
\begin{equation}
\label{upperB2}
\big\|(H_{\Delta_{n,i}}(u\circ l_{n,i})) \circ l_{n,j}^{-1} - H_{\Delta_i}u \big\|_{C^1(\Delta_j)} \leq \widehat\epsilon_n \|u\|_{C(\Delta_i)} 
\end{equation}
for each  \( i\neq j \), \( i,j\in I_d \), with \( \widehat\epsilon_n\to 0 \) as \( n\to\infty \). Since we can differentiate under the integral sign, the proof of this claim is no different from the already considered case.

The proof of the second claim is essentially the same. We have that
\[
\mathcal R_n^\prime (\mathcal I-\mathcal H_{\vec\Delta_n,\vec\Delta^\prime})^{-1} = (\mathcal I-\mathcal H_{\vec\Delta,\vec\Delta^\prime})^{-1}, \quad \mathcal R_n^\prime :=  \mathcal I - (\mathcal I-\mathcal H_{\vec \Delta,\vec\Delta^\prime})^{-1}(\mathcal H_{\vec \Delta_n,\vec\Delta^\prime}-\mathcal H_{\vec \Delta,\vec\Delta^\prime}).
\]
Clearly, \( C(\Delta_i) \) can be replaced by \( C(\Delta_i^\prime) \) in \eqref{upperB1} and \eqref{upperB2} simply because \( \Delta_i\subseteq \Delta_i^\prime \). Moreover, the spaces \( C(\Delta_j) \) and \( C^1(\Delta_j) \) in \eqref{upperB1} and \eqref{upperB2}, respectively, can be replaced by \( C(\Delta_j^\prime) \) and \( C^1(\Delta_j^\prime) \) because \eqref{eps_n}  only uses the fact that all the intervals \( \Delta_k \) are separated from each other, which is also true about the intervals \( \Delta_k^\prime \). Hence,
\[ 
\|  \mathcal H_{\vec \Delta_n,\vec\Delta^\prime}-\mathcal H_{\vec \Delta,\vec\Delta^\prime} \|_{\mathfrak B^\prime} \to 0 \qasq  n\to\infty,
\]
which finishes the proof of the lemma.
\end{proof}

\begin{lemma}
\label{lem:sa10}
Let \( \vec\Delta_n \) and \( \vec\Delta \)  be as in Proposition~\ref{prop:sa2}. Further, for each $i\in I_d$ let \( b_{n,i}(x) \), \( n\in \N \), be continuous functions defined on some interval containing  \( \cup_{m\in \mathbb{N}}\Delta_{m,i} \)  such that the sequence $\{b_{n,i}\}$ converges uniformly on this interval to a continuous function \( b_i(x) \) as $n\to\infty$. Let \( \vec y_n,\vec y \) be the unique solutions of
\[
\vec y_n = \mathcal H_{\vec \Delta_n}\vec y_n + \vec b_n \qandq \vec y = \mathcal H_{\vec \Delta}\vec y + \vec b,
\]
where \( (\vec b_n)_i(x) := b_{n,i}(x) \), \( x\in\Delta_{n,i} \), and \( (\vec b)_i(x) := b_i(x) \), \( x\in\Delta_i \), \( i\in I_d \). Then, for each \( i\in I_d \), the convergence
\[
H_{\Delta_{n,i}}(\vec y_n)_i(z) \to H_{\Delta_i}(\vec y)_i(z) \qasq n\to\infty
\]
holds uniformly in $z$ in the extended complex plane.
\end{lemma}
\begin{proof}
Since \( \vec b_n \in C(\vec \Delta_n) \) and \( \vec b \in C(\vec \Delta) \), we have that \( \vec y_n \in C(\vec \Delta_n) \) and \( \vec y \in C(\vec \Delta) \) by Lemma~\ref{lem:sa6}. Thus, we get from Lemma~\ref{lem:sa9} and the conditions placed on the functions \( b_{n,i}(x) \), \( i\in I_d \), that
\begin{equation}\label{sapsad4}
\|\vec y_n\|_{C(\vec \Delta_n)} = \|(\mathcal I-\mathcal H_{\vec\Delta_n})^{-1}\vec b_n\|_{C(\vec \Delta_n)} \lesssim \|\vec b_n\|_{C(\vec \Delta_n)} \lesssim 1.
\end{equation}
Put \(  h_{n,i}(z) := H_{\Delta_{n,i}}(\vec y_n)_i(z) \), which is a continuous function in \( \overline\C \) that is harmonic in \( D_{\Delta_{n,i}} \) and is equal to \( (\vec y_n)_i \) on \( \Delta_{n,i} \). Set \( \vec y_n^*  := \mathcal L_n^{-1} \vec y_n \), where \( \mathcal L_n \) was defined in the previous lemma. Then, it holds that 
\[
\vec y_n^* = \mathcal H_{\vec\Delta} \vec y_n^* + \vec b_n^*,
\]
where \( \vec b_n^* \) is the vector-function with coordinates
\[
b_{n,i}^*(x) := b_{n,i}(l_{n,i}^{-1}(x)) + \frac12\sum_{j\neq i,j\in I_d} \left(h_{n,j}(l_{n,j}^{-1}(x)) - h_{n,j}(l_{n,i}^{-1}(x))\right), \quad x\in \Delta_i,
\]
and \( l_{n,k}(z) \) is the linear function, see \eqref{UnSzego}, that maps \( \Delta_{n,k} \) onto \( \Delta_k \), \( k\in I_d \).  Explicit integral representation \eqref{Dirichlet}, \eqref{outer} and the bound \eqref{sapsad4} yield that
\[
|h_{n,j}(t_1) - h_{n,j}(t_2)| \lesssim \max_{x\in\Delta_{n,j}}\left|\frac{w_{\Delta_{n,j}}(t_1)}{t_1-x}-\frac{w_{\Delta_{n,j}}(t_2)}{t_2-x}\right|
\]
for any two points \( t_1,t_2\not\in \Delta_{n,j} \). Since the functions \( l_{n,k}(z) \) converge to \( z \) and the functions \( |w_{\Delta_{n,k}}(z)| \) converge to \( |w_{\Delta_k}(z)| \) locally uniformly in \( \C \), \( k\in I_d \), we get that
\[
\|\vec y_n^*-\vec y\|_{C(\vec\Delta)} \leq \| (\mathcal I-\mathcal H_{\vec\Delta})^{-1}\| \|\vec b_n^*-\vec b\|_{C(\vec\Delta)} \to 0
\]
as \( n \to \infty \) by Lemma~\ref{lem:sa6}. Put \( h_i(z) := H_{\Delta_i}(\vec y)_i(z) \), which is a continuous function in \( \overline\C \) that is harmonic in \( D_{\Delta_i} \) and equal to \( (\vec y)_i \) on \( \Delta_i \). It follows from the maximum principle for harmonic functions that
\[
\big|h_i(l_{n,i}(z)) - h_{n,i}(z) \big| \leq \|h_i\circ l_{n,i}-h_{n,i}\|_{C(\Delta_{n,i})} = \|(\vec y)_i-(\vec y_n^*)_i\|_{C(\Delta_i)} \to 0
\]
as \( n\to\infty \) for each \( z\in \overline \C \) and \( i\in I_d \). It only remains to observe that the differences \( h_i\circ l_{n,i} - h_i \) are uniformly converging to zero in the extended complex plane by the maximum modulus principle for harmonic functions as \( h_i \) is uniformly continuous on any compact set containing \( \Delta_i \) and \( l_{n,i}(z) \) converge to \( z \) uniformly on any such set.
\end{proof}

\begin{lemma}
\label{lem:sa11}
Proposition~\ref{prop:sa2} takes place.
\end{lemma}
\begin{proof}
Let \( v_{n,i} \) be the Radon-Nikodym derivative of \( \mu_{i|\Delta_{n,i}} \) with respect to \( \omega_{\Delta_{n,i}} \) and \( v_i \) be the Radon-Nikodym derivative of \( \mu_{i|\Delta_i} \) with respect to \( \omega_{\Delta_i} \), \( i\in I_d \). Set
\[
\vec a_n := \frac12\big( \log v_{n,1}, \ldots , \log v_{n,d}\big) \qandq \vec a := \frac12\big( \log v_1, \ldots , \log v_d\big).
\]
Further, let \( \vec s_n \) and \( \vec s \) be the unique solutions of
\begin{equation}\label{apsad3}
\vec s_n = \mathcal H_{\vec\Delta_n} \vec s_n + \vec a_n \qandq \vec s = \mathcal H_{\vec\Delta} \vec s + \vec a,
\end{equation}
 respectively.  Set \( \vec y_n:=\mathcal H_{\vec\Delta_n}\vec s_n \), \( \vec b_n:=\mathcal H_{\vec\Delta_n}\vec a_n \), \( \vec y:=\mathcal H_{\vec\Delta}\vec s \), and \( \vec b:= \mathcal H_{\vec\Delta}\vec a \). Applying the operators $\mathcal H_{\vec\Delta_n}$ and $\mathcal H_{\vec\Delta}$ to equations in \eqref{apsad3}, respectively, we get that
\[
\vec y_n = \mathcal H_{\vec\Delta_n} \vec y_n + \vec b_n \qandq \vec y = \mathcal H_{\vec\Delta}\vec y + \vec b.
\]
It follows from Lemma~~\ref{lem:sa8}, applied with \( h_n\equiv0 \), that the vector-functions \( \vec b_n \) and \( \vec b \) satisfy the conditions of Lemma~\ref{lem:sa10}. Hence, we have that
\[
h_{n,i}(z) := H_{\Delta_{n,i}}(\vec y_n)_i(z) \to h_i(z) := H_{\Delta_i}(\vec y)_i(z) \qasq n\to\infty
\]
uniformly in \( \overline\C \). Recall that
\[
\begin{cases}
h_{n,i}(x) & = (\vec y_n)_i(x) =  \big(\mathcal H_{\vec\Delta_n}\vec s_n\big)_i(x), \quad x\in\Delta_{n,i}, \medskip \\
h_i(x) & = (\vec y)_i(x) =  \big(\mathcal H_{\vec\Delta}\vec s\big)_i(x), \quad x\in\Delta_i,
\end{cases}
\]
for each \( i\in I_d \), by the very definition of \( \vec y_n \) and \( \vec y \) as well as the properties of harmonic extensions of continuous functions. Thus, we get from the very definition of \( \vec s_n \) and \( \vec s \) that
\[
\begin{cases}
(\vec s_n)_i(x) &= h_{n,i}(x) + \frac12\log v_{n,i}(x), \quad x\in\Delta_{n,i}, \medskip \\
(\vec s)_i(x) &= h_i(x) + \frac12\log v_i(x), \quad x\in\Delta_i,
\end{cases}
\] 
for each \( i\in I_d \). Since functions \( h_{n,i} \) converge uniformly to \( h_i \), we get from Lemma~\ref{lem:sa8} that
\begin{equation}
\label{limit_HnH}
H_{\Delta_{n,i}}(\vec s_n)_i(z) \to H_{\Delta_i}(\vec s)_i(z) \qasq n\to\infty
\end{equation}
locally uniformly in \( D_{\Delta_i} \) for each \( i\in I_d \). 

It only remains to show that \eqref{limit_HnH} implies \eqref{SzegoCont}. Since \( S_i(z) \) is a fixed non-vanishing function, it is enough to show that the ratios \( S_{n,i}(z)/S_i(z) \) converge to \( 1 \) locally uniformly in \( D_{\Delta_i} \) for each \( i\in I_d \). It can be readily seen from \eqref{outer}, \eqref{Dirichlet}, \eqref{apsad2} and \eqref{limit_HnH} that this claim is true for \( |S_{n,i}(z)/S_i(z)| \). Thus, the functions \( S_{n,i}(z)/S_i(z) \) form a normal family in \( D_{\Delta_i} \) whose only limit points are unimodular constants. As these functions are positive at infinity, the desired claim follows.
\end{proof}

\subsection{Proof of Proposition~\ref{prop:sa3} and an example.}

Set \( \theta(t) := |\log\mu^\prime(t)| \). As mentioned right after the statement of Proposition~\ref{prop:sa3}, we can consider \( (I_\gamma\theta)(x) \) instead of \( (I_\gamma\log\mu^\prime)(x) \).

\begin{lemma}
\label{lem:sa12}
Proposition~\ref{prop:sa3}(i) implies Proposition~\ref{prop:sa3}(ii).
\end{lemma}
\begin{proof}
We shall prove continuity at \( \beta \) understanding that continuity at \( \alpha \) can be proven analogously. We need to show that
\[
\lim_{n\to\infty} (I_\gamma\theta)(\beta_n) = (I_\gamma\theta)(\beta) 
\]
for any sequence \( \{\beta_n\}\subset\Delta(\mu) \) such that \( \beta_n\to\beta \) as \( n\to\infty \). Clearly, in the limit above we can replace \( \gamma \) by \( \alpha \). Let \( l_n(t) = \alpha + \frac{\beta_n-\alpha}{\beta-\alpha}(t-\alpha) \). Observe also that
\[
 (I_\alpha\theta)(\beta_n) = \sqrt{\frac{\beta_n-\alpha}{\beta-\alpha}} \int_\alpha^\beta \frac{\theta(l_n(t))}{\sqrt{\beta-t}}dt.
\]
The claim of the lemma now follows from \eqref{UnSzego} and the estimate
\begin{align*}
\left| \int_\alpha^\beta \frac{\theta(t)}{\sqrt{\beta-t}}dt - \int_\alpha^\beta \frac{\theta(l_n(t))}{\sqrt{\beta-t}}dt \right| & \leq \pi\sqrt{\beta-\alpha}\int_\alpha^\beta |\theta(t) - \theta(l_n(t))|d\omega_\Delta(t) \\
& \leq \pi\sqrt{\beta-\alpha}\int_\alpha^\beta |\log\mu^\prime(t) - \log\mu^\prime(l_n(t))|d\omega_\Delta(t). \qedhere
\end{align*}
\end{proof}

\begin{lemma}
\label{lem:sa13}
Proposition~\ref{prop:sa3}(iii) implies Proposition~\ref{prop:sa3}(i).
\end{lemma}
\begin{proof}
Let \( \{ \Delta_n \}\) be a sequence of closed subintervals of \( \Delta(\mu) \) that converges to \( \Delta \). Pick \( \epsilon> 0 \) and let \( \delta\)  be some positive number we will specify later. Then,
\begin{eqnarray}\nonumber
\int_\alpha^\beta  |\log\mu^\prime(t) - \log\mu^\prime(l_n(t))|d\omega_\Delta(t) \leq  \int_{\alpha+\delta}^{\beta-\delta}  |\log\mu^\prime(t) - \log\mu^\prime(l_n(t))|d\omega_\Delta(t) + \\ \int_{\beta-\delta}^\beta \big( \theta(t) + \theta(l_n(t)) \big) d\omega_\Delta(t) + \int_{\alpha}^{\alpha+\delta} \big( \theta(t) + \theta(l_n(t)) \big) d\omega_\Delta(t),\label{sapsad6}
\end{eqnarray}
where \( l_n(t) \) is the linear transformation with the positive leading coefficient that takes \( \Delta \) onto \( \Delta_n \). 
If $\alpha_n$ and $\beta_n$ are the endpoints of $\Delta_n$, i.e.,  \( \Delta_n = [\alpha_n,\beta_n] \), then
\[
\int_{\beta-\delta}^\beta \theta(l_n(t)) d\omega_\Delta(t) = \int_{\beta_n-\delta_n}^{\beta_n} \theta(t) d\omega_{\Delta_n}(t),
\] 
where \( \delta_n = \frac{\beta_n-\alpha_n}{\beta-\alpha}\delta \), and a similar equality holds for the integral of \( \theta(l_n(t)) \) on \( [\alpha,\alpha+\delta] \). Notice that $\lim_{n\to\infty}\delta_n=\delta$.
Now, it becomes clear that the assumption (iii) of the proposition  implies that exists \( \delta>0 \) and $N\in \mathbb{N}$ such that
\begin{eqnarray*}
\int_{\beta-\delta}^\beta \theta(t) d\omega_\Delta(t)< \frac\epsilon5,\,\int_{\alpha}^{\alpha+\delta}\theta(t) d\omega_\Delta(t) < \frac\epsilon5,\,\hspace{3cm}\\
\int_{\beta-\delta}^\beta \theta(l_n(t)) d\omega_\Delta(t)< \frac\epsilon5,\,\int_{\alpha}^{\alpha+\delta}\theta(l_n(t)) d\omega_\Delta(t) < \frac\epsilon5
\end{eqnarray*}
for all $n\ge N$.
To see that the first integral in \eqref{sapsad6} also can be made smaller than \( \epsilon/5 \) for all  large enough \( n \), observe that \( d\omega_\Delta(t) \leq (\pi\delta)^{-1}dt \) on the interval of integration and that
\[
\lim_{n\to\infty}
\int_{\alpha+\delta}^{\beta-\delta}  |\log\mu^\prime(t) - \log\mu^\prime(l_n(t))|dt=0\,,
\]
which can be shown, for instance, by approximating \( \log\mu^\prime(t) \) with continuous functions in \( L^1 \) norm (the desired estimate for continuous functions follows trivially from uniform continuity).
\end{proof}

\begin{lemma}
\label{lem:sa14}
Proposition~\ref{prop:sa3}(ii) implies Proposition~\ref{prop:sa3}(iii).
\end{lemma}
\begin{proof}
Since \( (I_\gamma\theta)(x) \) is continuous at \( \beta \), there exists an interval, say \( [a,b] \), that contains \( \beta \) in its interior (unless \( \beta \) is the right endpoint of \( \Delta(\mu)\), in which case \( b=\beta \)), on which \( (I_\gamma\theta)(x) \) is bounded.  It is known \cite[Theorem~2.1]{SamkoKilbasMarichev} that
\begin{equation}
\label{I-0}
\theta(x) = \frac{d}{dx} \left(\frac1{\sqrt\pi} \int_\gamma^x \frac{(I_\gamma\theta)(t)}{\sqrt{x-t}}dt\right) = \frac{d}{dx} (I_\gamma(I_\gamma\theta))(x), \quad x\in(\gamma,b).
\end{equation}
Let us write \( (I_\gamma^2\theta)(x) \) for \( (I_\gamma(I_\gamma\theta))(x) \), which is an absolutely continuous function on \( [\gamma,b] \) that vanishes at \( \gamma \), see again \cite[Theorem~2.1]{SamkoKilbasMarichev}. Fix some \( \delta \in (0,(b-a)/2) \). Notice that
\begin{multline*}
\frac{d}{dx}\left( \int_\gamma^{x-\delta}\frac{(I_\gamma^2\theta)(t)}{\sqrt{x-t}}dt \right) = \frac{d}{dx}\left( \int^{x-\gamma}_\delta \frac{(I_\gamma^2\theta)(x-s)}{\sqrt s}d s \right) = \\
 \lim_{h\to 0}\int^{x-\gamma}_\delta  \frac{(I_\gamma^2\theta)(x+h-s)-(I_\gamma^2\theta)(x-s)}{h}\frac{ds}{\sqrt s} +\lim_{h\to 0} \frac1h\int_\gamma^{\gamma+h}\frac{(I_\gamma^2\theta)(t)}{\sqrt{x+h-t}}dt.
 \end{multline*}
The second limit is equal to zero due to continuity of the integrand and vanishing of \( (I_\gamma^2\theta)(t) \) at \( \gamma \). It is known,  see \cite[Theorem~6.9]{RoydenFitzpatrick}, that absolute continuity of a function is equivalent to uniform integrability of its divided differences. As  \( (I_\gamma^2\theta)(x) \) is absolutely continuous and \( 1/\sqrt{s} \) is continuous on \( [\delta,b-\gamma] \), we get from Vitali's convergence theorem and \eqref{I-0} that the first limit is equal to 
\[
\int_\delta^{x-\gamma} \frac{\theta(x-s)}{\sqrt s}ds 
\]
and hence
\begin{equation}
\label{I-1}
\frac{d}{dx}\left( \int_\gamma^{x-\delta}\frac{(I_\gamma^2\theta)(t)}{\sqrt{x-t}}dt \right) =\int_\delta^{x-\gamma} \frac{\theta(x-s)}{\sqrt s}ds = \int_\gamma^{x-\delta} \frac{\theta(t)}{\sqrt{x-t}}dt.
\end{equation}
Writing the outer \( I_\gamma \) transform explicitly and  changing the order of integration gives us
\begin{multline}
\label{I-2}
\int_\gamma^{x-\delta}\frac{(I_\gamma^2\theta)(t)}{\sqrt{x-t}}dt = \frac1{\sqrt\pi}\int_\gamma^{x-\delta} \left( \int_s^{x-\delta} \frac{dt}{\sqrt{(t-s)(x-t)}} \right) (I_\gamma\theta)(s)ds  = \\ \int_\gamma^{x-\delta}  F\left(1-\frac\delta{x-s} \right) (I_\gamma\theta)(s)ds, \quad F(s) := \frac1{\sqrt \pi} \int_0^s \frac{dt}{\sqrt{t(1-t)}}.
\end{multline}
 Again, we need to justify changing the order of differentiation and integration. To this end, assume now that \( x\in(a+2\delta,b) \). By the mean-value theorem and its very definition, the derivative of the last integral in \eqref{I-2} is equal to the limit as \( h\to0 \) of the following sum of three terms
 \begin{multline}
 \label{I-3}
 \delta\int_\gamma^a F^\prime\left(1-\frac\delta{x+\xi_h-s} \right) \frac{(I_\gamma\theta)(s)}{(x-s)^2}ds + \\
 \frac1h\int_a^{x-\delta} \left( F\left(1-\frac\delta{x+h-s}\right) - F\left(1-\frac\delta{x-s}\right) \right)  (I_\gamma\theta)(s)ds + \\ 
\frac1h\int_{x-\delta}^{x+h-\delta}F\left(1-\frac\delta{x+h-s}\right)  (I_\gamma\theta)(s)ds,
 \end{multline}
where \( \xi_h = \xi_h(x,s) \) is such that \( |\xi_h|\leq |h| \). The first term in the sum above converges to 
\[
 \sqrt{\frac\delta\pi} \int_\gamma^a \frac{(I_\gamma\theta)(s)ds}{(x-s)\sqrt{x-\delta-s}},
\]
when \( h\to0 \).  This follows from  the dominated convergence theorem because \( (I_\gamma\theta)(s) \) is a fixed integrable function and the other factor in the integrand is a function continuous in  $s$ that converges uniformly when $h\to 0$. The second term in \eqref{I-3} has the following limit 
\[
\delta \int_a^{x-\delta} F^\prime\left(1-\frac\delta{x-s} \right) \frac{(I_\gamma\theta)(s)}{(x-s)^2}ds = \sqrt{\frac\delta\pi} \int_a^{x-\delta}\frac{(I_\gamma\theta)(s)ds}{(x-s)\sqrt{x-\delta-s}}
\]
as \( h\to0 \) according to Vitali's convergence theorem. Indeed,  \( (I_\gamma\theta)(s) \) is bounded on the interval of integration by assumptions of Proposition~\ref{prop:sa3}(ii) and the divided differences of $F(1-\delta(x-s)^{-1})$ are uniformly integrable. The last term in \eqref{I-3} can be rewritten as
\[
\frac\delta h \int_0^{h/(h+\delta)}  (I_\gamma\theta)\left(x+h-\frac\delta{1-t} \right)\frac{F(t)}{(1-t)^2}dt.
\]
Its limit as \( h\to0 \) is equal to \( 0 \) due to the boundedness of \( (I_\gamma\theta)(x) \) on \( (a,b) \) as well as the continuity of the function \( F(t)/(1-t)^2 \) around the origin and its vanishing  at $t=0$.  Altogether, we get from \eqref{I-1}, \eqref{I-2}, and the reasoning above that
\[
\int_\gamma^{x-\delta} \frac{\theta(t)dt}{\sqrt{x-t}} = \sqrt{\frac\delta\pi} \int_\gamma^{x-\delta} \frac{(I_\gamma\theta)(s)ds}{(x-s)\sqrt{x-\delta-s}} = \frac2{\sqrt\pi} \int_0^{ L_{x,\delta}} (I_\gamma\theta)\big(x-\delta-\delta t^2\big)\frac{dt}{1+t^2},
\]
where \( L_{x,\delta} = \sqrt{(x-\delta-\gamma)/\delta} \). We can use the identity $\int_0^\infty (t^2+1)^{-1}dt=\pi/2$ to rewrite 
\begin{multline*}
\frac1{\sqrt\pi}\int_{x-\delta}^x \frac{\theta(t)dt}{\sqrt{x-t}} = (I_\gamma\theta)(x) -\frac1{\sqrt\pi} \int_\gamma^{x-\delta} \frac{\theta(t)dt}{\sqrt{x-t}} = \frac2\pi \left( \int_{ \ell_\delta/\sqrt\delta}^\infty \frac{(I_\gamma\theta)(x)dt}{1+t^2} \right. \\   \left. - \int_{\ell_\delta/\sqrt\delta}^{ L_{x,\delta}} \frac{(I_\gamma\theta)(x-\delta-\delta t^2)dt}{1+t^2} + \int_0^{\ell_\delta/\sqrt\delta} \frac{(I_\gamma\theta)(x) - (I_\gamma\theta)(x-\delta-\delta t^2)}{1+t^2}dt\right)
\end{multline*}
for any positive \( \ell_\delta \). If we choose $\ell_\delta=\delta^{\frac 18}$, we get that
\[
\left| \int_{\ell_\delta/\sqrt\delta}^\infty \frac{(I_\gamma\theta)(x)dt}{1+t^2} \right| \leq |(I_\gamma\theta)(x)|\frac{\sqrt\delta}{\ell_\delta}=|(I_\gamma\theta)(x)|\delta^{\frac 38}\,.
\]
Similarly,
\begin{multline*}
\left| \int_{\ell_\delta/\sqrt\delta}^{ L_{x,\delta}} \frac{(I_\gamma\theta)(x-\delta-\delta t^2)dt}{1+t^2} \right| \leq \frac{\sqrt\delta}{\ell_\delta^2}\int_\gamma^{x-\delta-\ell_\delta^2}\frac{|(I_\gamma\theta)(s)|ds}{\sqrt{x-\delta-s}} \\
 \leq \frac{\sqrt\delta}{\ell_\delta^3} \int_\gamma^b |(I_\gamma\theta)(s)|ds=\delta^{\frac 18} \int_\gamma^b |(I_\gamma\theta)(s)|ds\lesssim\delta^{\frac 18}  \,.
\end{multline*}
Next, we have that
\[
\left|\int_0^{\ell_\delta/\sqrt{\delta}} \frac{(I_\gamma\theta)(x) - (I_\gamma\theta)(x-\delta-\delta t^2)}{1+t^2}dt \right| \leq \frac\pi2 \max_{s\in[x-\delta-\delta^{\frac 14},x-\delta]} \big|(I_\gamma\theta)(x) - (I_\gamma\theta)(s) \big|.
\]
Now, using continuity of \( (I_\gamma\theta)(x) \) at \( \beta \), the above estimates show that given \( \epsilon> 0\), we can always find \( d_\epsilon \) so that \( |x-\beta|<d_\epsilon\) and \(|x-\delta-\beta|<d_\epsilon \) imply
\[
\int_{x-\delta}^x \frac{\theta(t)dt}{\sqrt{x-t}} < \epsilon.
\]
This is precisely the statement of the second part in Proposition~\ref{prop:sa3}(iii). The first one is proved similarly.
\end{proof}

\noindent{\bf An example.} The uniform Szeg\H{o} condition is subtle and it depends on the direction in which the point is approached as shown by the following example. Given \( \epsilon\in[0,1] \), let \( \mu_\epsilon \) be an absolutely continuous measure on \( [-1,1] \) such that \( \log\mu^\prime_\epsilon(x) = - \theta_\epsilon(x) \), where
\[
\theta_\epsilon(x) := \begin{cases}
0, & x\in[-1,0], \\
x^{-1/2}(1 - \log x)^{-\epsilon}, & x\in(0,1].
\end{cases}
\]
When \( \epsilon=0 \), it holds that \( \mu_0\in\mathrm{Sz}([-1,a]) \) for any \( a\in(-1,1] \). Indeed, the claim is obvious when \( a\leq 0 \). When \( a>0 \), it holds that
\[
\frac1\pi\int_{-1}^a \frac{\theta_0(x)dx}{\sqrt{(x+1)(a-x)}} = \frac1\pi \int_0^a \frac{dx}{\sqrt{(x+1)x(a-x)}} \leq \frac1\pi \int_0^a \frac{dx}{\sqrt{x(a-x)}} = 1.
\]
This computation also shows that
\[
\frac1\pi \int_{-1}^\delta \frac{\theta_0(x)dx}{\sqrt{\delta-x}} = \begin{cases} 0, & \delta\leq0, \\ 1, & \delta>0, \end{cases}
\]
and therefore \( \mu\not\in\mathrm{USz}([-1,0]) \) as follows from Proposition~\ref{prop:sa3}(ii). On the other hand, when \( \epsilon>0 \), it holds that
\[
0\leq \frac1\pi \int_{-1}^\delta \frac{\theta_\epsilon(x)dx}{\sqrt{\delta-x}} \leq \frac1{(1-\log\delta)^\epsilon} \to 0 \qasq \delta\to0^+,
\]
and therefore \( \mu\in\mathrm{USz}([-1,0]) \) again by Proposition~\ref{prop:sa3}(ii) (in fact, \( \mu\in\mathrm{USz}([-1,a]) \)  for any \( a\in(-1,1] \) in this cases). However, when \( \epsilon\leq 1 \), we have that \( \mu_\epsilon\not\in\mathrm{Sz}([0,1]) \) since
\[
\frac1\pi \int_0^1  \frac{\theta_\epsilon(x)dx}{\sqrt{x(1-x)}} \geq \frac1\pi \int_0^1 \frac{dx}{x(1-\log x)^\epsilon} = \frac1\pi\int_1^\infty \frac{du}{u^\epsilon} = \infty.
\]

\subsection{Proof of Theorem~\ref{thm:sa4}}

For convenience, we keep \( \vec c \) as a symbol standing for an arbitrary vector in \( (0,1)^d \) that satisfies $|\vec c|=1$ while fixing \( \vec \varsigma\in(0,1)^d \), \( |\vec \varsigma|=1  \), and a ray sequence of multi-indices \( \mathcal N(\vec\varsigma) \) satisfying
\[
\n / |\n| \to \vec\varsigma \qasq |\n|\to\infty, \;\; \n\in\mathcal N(\vec\varsigma).
\]
As agreed earlier, we use subindex \( \vec  c\) to indicate that a quantity depends on the parameter \( \vec c \) while for rational values of \( \vec c \), i.e., when \( \vec c \) is equal to  \( \n/|\n| \) for some \( \n\in\N^d \), we replace the subindex \( \vec c \) by \( \n \). For the purposes of this subsection, we also let  \( \Delta_i := \Delta(\mu_i) \) and assume \( \mu_i \in \mathrm{USz}(\Delta_{\vec \varsigma,i}) \), \( i\in I_d \).

Theorem~\ref{thm:vw2} further below is central to our approach. It has three conditions appearing in it. The first one will be trivial to check. The other two require  more work and we do it in two separate lemmas. Our goal is to show that these conditions are satisfied with \( \mu=\mu_i \) and \( \Delta_n = \Delta_{\n,i} \), \( \n\in\mathcal N(\vec\varsigma) \), for each \( i\in I_d \). Notice that the intervals \( \Delta_{\n,i} \) converge to \( \Delta_{\vec \varsigma,i} \) for each \( i\in I_d \) as \( |\n|\to\infty \), \( \n\in\mathcal N(\vec \varsigma) \), as shown in \cite[Proposition~2.1]{Ya16}.

\begin{lemma}
\label{lem:sa15}
Let \( N(\vec\varsigma) \) be as above and \((\omega_{\n,1},\omega_{\n,2},\ldots,\omega_{\n,d}) \), \( \n\in\mathcal N(\vec\varsigma) \), be the vector-equilibrium measures \eqref{vec_eq1}--\eqref{vec_eq2}. For each \( i\in I_d \), the sequence \( \{ ( |\n|/n_i) \omega_{\n,i} \}_{\n\in\mathcal N(\vec\varsigma)} \) satisfies condition (3) of Theorem~\ref{thm:vw2} with \( \mu=\mu_i \) and \( \Delta_n = \Delta_{\n,i} \).
\end{lemma}
\begin{proof}
From now on \( l\in I_d \) is fixed. By its very definition, \( ( |\n|/n_l) \omega_{\n,l} \) is a probability measure supported on \( \Delta_{\n,l} \) for each \( \n\in\mathcal N(\vec\varsigma) \). It follows from \eqref{eq_meas} that each of these measures is absolutely continuous with respect to the Lebesgue measure. Since \( |\n|/n_l \to 1/(\vec\varsigma)_l\neq 0 \) as \( |\n|\to\infty \), \( \n\in\mathcal N(\vec\varsigma) \), we shall omit the factors \( |\n|/n_l \) in the forthcoming analysis.

Recall functions \( \h_{\vec c,k}(z) \), \( k\in\{0,1,\ldots,d\} \), defined in \eqref{CauchyInt}. Put
\begin{equation}
\label{Wc-1}
W_{\vec c,l}(z) := w_{\Delta_{\vec c,l}}(z) \big( \h_{\vec c,0}(z) - \h_{\vec c,l}(z) \big),
\end{equation}
which is holomorphic in \( \C\setminus\cup_{i\in I_d}\Delta_{\vec c,i} \). Since \( \h_{\vec c,0}(z) \) and \( \h_{\vec c,l}(z) \) are branches of a rational function on a Riemann surface, it holds that \(  \h_{\vec c,0\pm}(x) = \h_{\vec c,l\mp}(x) \) for \( x\in\Delta_{\vec c,l} \). One can also readily see that \( w_{\Delta_{\vec c,l}+}(x) = - w_{\Delta_{\vec c,l}-}(x) \) for \( x\in\Delta_{\vec c,l} \) and so $W_{{\vec c,l}+}(x)=W_{{\vec c,l}-}(x)$ for $x\in \Delta_{\vec c,l}$. Hence,  \( W_{\vec c,l}(z) \) is in fact analytic across \( (\alpha_{\vec c,l},\beta_{\vec c,l}) \). Since
\begin{equation}
\label{hhw}
\h_{\vec c,0+}(x) - \h_{\vec c,l+}(x) = \h_{\vec c,l-}(x) - \h_{\vec c,l+}(x) = -2\pi\ic\omega_{\vec c,l}^\prime(x), \quad x\in\Delta_{\vec c,l},
\end{equation}
by \eqref{eq_meas}, we get from the second claim of Proposition~\ref{prop:gaps} that the difference \( \h_{\vec c,0}(z) - \h_{\vec c,l}(z) \) can have at most square root singularities at \( \alpha_{\vec c,l},\beta_{\vec c,l} \), which means that \( W_{\vec c,l}(z) \) is in fact analytic in some neighborhood of \( \Delta_{\vec c,l} \). 

Let now \( \vec\varsigma \) and \( \mathcal N(\vec\varsigma) \) be as in the statement of the lemma. It follows from the last claim of Proposition~\ref{prop:Ya} and \eqref{CauchyInt} that there exists an open neighborhood of \( \Delta_{\vec\varsigma,l} \) such that 
\begin{equation}
\label{Wc-2}
W_{\n,l}(z) \to W_{\vec\varsigma,l}(z) \qasq |\n|\to\infty, \quad \n\in\mathcal N(\vec\varsigma),
\end{equation}
uniformly on its closure. 

Assume that \( \omega_{\vec \varsigma,l}(x) \) blows up like a square root at both \( \alpha_{\vec \varsigma,l} \) and \( \beta_{\vec \varsigma,l} \). Then, we get from \eqref{hhw} and Proposition~\ref{prop:gaps} that \( W_{\vec\varsigma,l}(x) \neq 0 \) for \( x\in\Delta_{\vec \varsigma,l} \). Thus, the neighborhood in \eqref{Wc-2} can be chosen so that all the functions are non-vanishing on its closure and respectively the moduli \( |W_{\n,l}(z)| \) are uniformly bounded above and away from zero there for all \( |\n| \) large enough, \( \n\in\mathcal N(\vec\varsigma) \). Then \eqref{Wc-1} and \eqref{hhw} yield that
\[
|w_{\Delta_{\n,l}}(x)|^{-1} \lesssim \omega_{\n,l}^\prime(x) \lesssim |w_{\Delta_{\n,l}}(x)|^{-1}, \quad x\in\Delta_{\n,l},
\]
for all \( |\n| \) large enough, \( \n\in\mathcal N(\vec\varsigma) \), where the constants of proportionality are independent of \( \n \). Thus, \eqref{w_cond_1} holds with \( \varkappa_L=\varkappa_U=-1 \).

Next, assume that \( \omega_{\vec \varsigma,l}(x) \) blows up like a square root at \( \beta_{\vec \varsigma,l} \) and vanishes like a square root at \( \alpha_{\vec \varsigma,l} \). Then, we get from \eqref{hhw} and Proposition~\ref{prop:gaps} that \( W_{\vec\varsigma,l}(x) \neq 0 \) for \( x\in(\alpha_{\vec \varsigma,l},\beta_{\vec \varsigma,l}] \) and \( W_{\vec\varsigma,l}(z) \) has a simple zero at \( \alpha_{\vec \varsigma,l} \). Therefore, the neighborhood in \eqref{Wc-2} can be chosen so that all the functions have exactly one zero, necessarily simple, in its closure.  Each \( W_{\n,l}(z) \)  is conjugate-symmetric and does not vanish on \( (\alpha_{\n,l},\beta_{\n,l}) \) by \eqref{hhw} and Proposition~\ref{prop:gaps}. Hence, for each \( \n \) there exists \(\gamma_{\n,l}$ such that \(\gamma_{\n,l}\leq\alpha_{\n,l} \), \( W_{\n,l}(\gamma_{\n,l})=0 \), and \( \gamma_{\n,l}\to \alpha_{\vec\varsigma,l} \) as \( |\n| \to\infty \), \( \n\in\mathcal N(\vec\varsigma) \). Then, \eqref{Wc-2} yields that
\begin{equation}
\label{Wc-3}
\frac{W_{\n,l}(z)}{z-\gamma_{\n,l}} \to \frac{W_{\vec\varsigma,l}(z)}{z-\alpha_{\vec\varsigma,l}} \qasq |\n|\to\infty, \quad \n\in\mathcal N(\vec\varsigma),
\end{equation}
uniformly on the closure of some neighborhood of \( \Delta_{\vec\varsigma,l} \), and all the functions are non-vanishing on this closure. As in the first  case, we can conclude that
\begin{equation}
\label{second_est}
\frac{x-\gamma_{\n,l}}{|w_{\Delta_{\n,l}}(x)|} \lesssim \omega_{\n,l}^\prime(x) \lesssim \frac{x-\gamma_{\n,l}}{|w_{\Delta_{\n,l}}(x)|}, \quad x\in\Delta_{\n,l},
\end{equation}
for all \( |\n| \) large enough, \( \n\in\mathcal N(\vec\varsigma) \). Trivially, \( x-\alpha_{\n,l}\leq x-\gamma_{\n,l}<2(\beta_{\vec\varsigma,l}-\alpha_{\vec\varsigma,l}) \), \( x\in\Delta_{\n,l} \), where the upper bound holds for all \( |\n| \) large enough. Necessarily,
\[
|w_{\Delta_{\n,l}}(x)| \lesssim \omega_{\n,l}^\prime(x) \lesssim |w_{\Delta_{\n,l}}(x)|^{-1}, \quad x\in\Delta_{\n,l},
\]
for all \( |\n| \) large enough, \( \n\in\mathcal N(\vec\varsigma) \).  Thus, \eqref{w_cond_1} holds with \( -\varkappa_L=\varkappa_U=-1 \).

The cases where \( \omega_{\vec \varsigma,l}(x) \) blows up like a square root at \( \alpha_{\vec \varsigma,l} \) and vanishes like a square root at \( \beta_{\vec \varsigma,l} \) and where \( \omega_{\vec \varsigma,l}(x) \) vanishes like a square root at both \( \alpha_{\vec \varsigma,l} \) and \( \beta_{\vec \varsigma,l} \) can be examined similarly. This finishes the proof of \eqref{w_cond_1}.

Let us now verify \eqref{w_cond_2}. If \( \alpha_{\n,l}>\alpha(\mu_l) \), then we get from \eqref{vec_eq2} and Proposition~\ref{prop:Ya} that \( \boldsymbol z_{\n,l-1} = \boldsymbol \alpha_{\n,l} \). This, according to Proposition~\ref{prop:gaps}, implies that \( \omega_{\n,l}(x) \) vanishes like a square root at \( \alpha_{\n,l} \). That is, \( \gamma_{\n,l} = \alpha_{\n,l} \) in \eqref{second_est} by \eqref{Wc-1} and \eqref{hhw}. Hence, the upper bound in \eqref{second_est} can be rewritten as
\[
\omega_{\n,l}^\prime(x) \lesssim \sqrt{\frac{x-\alpha_{\n,l}}{\beta_{\n,l}-x}}, \quad x\in\Delta_{\n,l},
\]
which yields \eqref{w_cond_2} with \( \widehat\varkappa_U = 1/2 \). When \( \beta_{\n,l}<\beta(\mu_l) \), the analysis around \( \beta_{\n,l} \)  is identical.

Finally, since \eqref{Wc-2} is a claim about uniform convergence of analytic functions, we easily get from Cauchy integral formula as well as \eqref{Wc-1} and \eqref{hhw} that \( \{\omega_{\n,l}^{\prime\prime}(x) \}\) converges uniformly to \( \omega_{\vec\varsigma,l}^{\prime\prime}(x) \) on compact subsets of \((\alpha_{\vec\varsigma,l},\beta_{\vec\varsigma,l}) \). This, in turn, yields  equicontinuity of the densities \( \omega_{\n,l}^\prime(x) \) on compact subsets of \( (\alpha_{\vec\varsigma,l},\beta_{\vec\varsigma,l}) \).
\end{proof}

Functions \( \kappa_n(x) \) that appear in Theorem~\ref{thm:vw2} will be drawn from the following family: for each \( \vec c\in(0,1)^d, |\vec c|=1 \) and \( i\in I_d \), set
\begin{equation}
\label{kappani}
\kappa_{\vec c,i}(x) := \frac1{c_i}\left( -V^{\omega_{\vec c,i}}(x) + \frac{\ell_{\vec c,i}}2 - \frac12\sum_{j\in I_d,j\neq i} V^{\omega_{\vec c,j}}(x)\right), \quad x\in\Delta_i.
\end{equation}

\begin{lemma}
\label{lem:sa16}
 For each \( i\in I_d \), the sequence \( \{ \kappa_{\n,i} \}_{\n\in\mathcal N(\vec\varsigma)} \) satisfies condition (1) of Theorem~\ref{thm:vw2} with \( \mu=\mu_i \) and \( \Delta_n = \Delta_{\n,i} \), \( \n\in\mathcal N(\vec\varsigma) \).
\end{lemma}
\begin{proof}
Fix \( l\in I_d \). It follows from \eqref{vec_eq2}, that each \( \kappa_{\vec c,l}(x) =0 \) on \( \Delta_{\vec c,l} \) and \( \kappa_{\vec c,l}(x) < 0 \) on \( \Delta_l\setminus\Delta_{\vec c,l} \). Hence, the functions \( \kappa_{\n,l}(x) \) satisfy the first part of assumption (1) of Theorem~\ref{thm:vw2}. To verify the rest, let us concentrate on the estimates around \( \alpha_{\n,l} \) as estimates around \( \beta_{\n,l} \) are similar. Naturally, we only have something to prove when \( \alpha_{\n,l}> \alpha(\mu_l) \). It follows from \eqref{future_kappas} that
\[
-2(|\n|/n_l)\kappa_{\n,l}(x) = \int_{\alpha_{\n,l}}^x \big( \boldsymbol h_{\n,l}(t) - \boldsymbol h_{\n,0}(t) \big) dt, \quad x\in[\alpha(\mu_l),\alpha_{\n,l}].
\]
As we have explained in the next to last paragraph of the previous lemma, \( \alpha_{\n,l}> \alpha(\mu_l) \) implies that \( \gamma_{\n,l} = \alpha_{\n,l} \) in \eqref{Wc-3}. In particular, we have that 
\[
(\alpha_{\n,l}-x)^{1/2} \lesssim \boldsymbol h_{\n,0}(x) - \boldsymbol h_{\n,l}(x) \lesssim (\alpha_{\n,l}-x)^{1/2}, \quad  x\in[\alpha_{\vec\varsigma,l}-\epsilon,\alpha_{\n,l}]
\]
for some \( \epsilon>0 \) by \eqref{Wc-1}.  Since \( n_i/|\n| \to (\vec\varsigma)_i \) as \( |\n|\to\infty \), \( \n\in\mathcal N(\vec\varsigma) \), we then get that
\[
(\alpha_{\n,l}-x)^{3/2} \lesssim - \kappa_{\n,l}(x) \lesssim (\alpha_{\n,l}-x)^{3/2}, \quad x\in[\alpha_{\vec\varsigma,l}-\epsilon,\alpha_{\n,l}].
\]
If \( \alpha(\mu_l) < \alpha_{\vec\varsigma,l}-\epsilon \), the above estimate can be readily extended to \( [\alpha(\mu_l),\alpha_{\vec\varsigma,l}-\epsilon] \) by noticing that the functions \( \kappa_{\n,l}(x) \) converge uniformly there to \( \kappa_{\vec\varsigma,l}(x) \) as \( |\n|\to\infty \), \( \n\in\mathcal N(\vec\varsigma) \) (one can use \eqref{kappani} and the weak$^*$ convergence of measures to see this). This finishes the proof of the lemma.
\end{proof}

Recall definition \eqref{operH} of the operators \( \mathcal H_{\vec\Delta_{\vec c}} \) from the previous subsection. Let
\begin{equation}
\label{vec-sc}
\vec s_{\vec c} := (\mathcal I-\mathcal H_{\vec\Delta_{\vec c}})^{-1}\vec a_{\vec c}, \quad \vec a_{\vec c} := \frac12\big( \log v_{\vec c,1}, \log v_{\vec c,2}, \ldots, \log v_{\vec c,d} \big), 
\end{equation}
where \( v_{\vec c,i}(x) \) is the Radon-Nikodym derivative of \( \mu_{i|\Delta_{\vec c,i}} \) with respect to the arcsine distribution of \( \Delta_{\vec c,i} \), \( i\in I_d \), see \eqref{szego}. As follows from Lemma~\ref{lem:sa7} and \eqref{Dirichlet}, we have 
\[
\log \big|S_{\vec c,i}(z) \big| = H_{\Delta_{\vec c,i}}(s_{\vec c,i},z),
\]
where we write \( \vec s_{\vec c} = (s_{\vec c,1},s_{\vec c,2},\ldots,s_{\vec c,d})  \). Since the intervals \( \Delta_{\n,i} \) converge to \( \Delta_{\vec \varsigma,i} \) for each \( i\in I_d \) as \( |\n|\to\infty \), \( \n\in\mathcal N(\vec \varsigma) \), we get from Proposition~\ref{prop:sa2} that
\begin{equation}
\label{Conv-H}
H_{\Delta_{\n,i}}(s_{\n,i},z) \to H_{\Delta_{\vec\varsigma,i}}(s_{\vec \varsigma,i},z)
\end{equation}
locally uniformly in \( D_{\Delta_{\vec\varsigma,i}} \) as \( |\n|\to\infty \), \( \n\in\mathcal N(\vec \varsigma) \), for each \( i\in I_d \).

The following lemma is an application of Theorem~\ref{thm:vw2}. In what follows, we denote by \( T_n(f\mu) \) the \( n \)-th monic orthogonal polynomial with respect to the measure \( f\mu \), where \( f \) is a non-vanishing continuous function and \( \mu \) is a Borel measure.

\begin{lemma}
\label{lem:sa17}
Let \( U \) be a  compact set in $C(\vec\Delta)$. Assume that \( \mu_i\in\mathrm{USz}(\Delta_{\vec\varsigma,i}) \) for every $i\in I_d$. For each \( \n\in \mathcal N(\vec\varsigma) \), 
choose \( \vec u_\n \in U \)  and write \( \vec u_\n = (u_{\n,1},u_{\n,2},\ldots,u_{\n,d}) \). Then, it holds for every \( i\in I_d \) that
\begin{multline*}
\log\left|T_{n_i} \left( e^{2 u_{\n,i} - \sum_{j\neq i}(H_{\Delta_{\n,j}}s_{\n,j} + |\n|V^{\omega_{\n,j}})}\mu_i \right)(z) \right| = o_U(1) - |\n| V^{\omega_{\n,i}}(z)
\\  + H_{\Delta_{\n,i}}(s_{\n,i}+u_{\n,i})(\infty) - H_{\Delta_{\n,i}}(s_{\n,i}+u_{\n,i})(z),
\end{multline*}
where \( o_U(1) \) is such that for every \( \epsilon>0 \) and every closed \( K_i \subset D_{\Delta_{\vec\varsigma,i}} \), there exists \( N_\epsilon \) such that \( \| o_U(1) \|_{C^1(K_i)} \leq \epsilon \), \( |\n|\geq N_\epsilon \), \( \n\in\mathcal N(\vec \varsigma) \), independently of the choice of \( \{\vec u_\n \}\subseteq U \). 
\end{lemma}
\begin{proof}
For each $i\in I_d$, we apply Theorem~\ref{thm:vw2} with
\[
\Delta(\mu_i), \quad \Delta_n = \Delta_{\n,i}, \quad \Delta=\Delta_{\vec\varsigma,i}, \quad \omega_n = (|\n_i|/n_i) \omega_{\n,i}, \quad \kappa_n(x) = \kappa_{\n,i}(x), \quad \n\in\mathcal N(\vec\varsigma),
\]
where \( \kappa_{\vec c,i}(x) \) were defined in \eqref{kappani}, and
\begin{equation}
\label{sapsad11}
h_{\n,i}(x) := 2u_{\n,i}(x) -  \sum_{j\in I_d,j\neq i}(H_{\Delta_{\n,j}}s_{\n,j})(x), \quad x\in\Delta_i, \quad \n\in\mathcal N(\vec\varsigma). 
\end{equation}
The sequence $\{\kappa_{\n,i}\}$ satisfies assumption (1) of Theorem~\ref{thm:vw2} by Lemma~\ref{lem:sa16}. It readily follows from \eqref{Conv-H} that
\[
\left\{ 2u_i - \sum_{j\in I_d,j\neq i} H_{\Delta_{\n,j}} s_{\n,j} :\vec u\in U \qandq \n\in\mathcal N(\vec\varsigma)\right\} 
\]
is a precompact subset of \( C(\Delta_i) \). Clearly, its closure, say $\mathcal K_i$, is compact and \( h_{\n,i} \in \mathcal K_i \) for each \( \n\in\mathcal N(\vec\varsigma) \). That is, the sequence \( \{ h_{\n,i} \} \) satisfies assumption (2) of Theorem~\ref{thm:vw2}. The measures \( \{(|\n|/n_i)\omega_{\n,i} \}\) satisfy  assumption (3) of Theorem~\ref{thm:vw2} by Lemma~\ref{lem:sa15}. All assumptions of Theorem~\ref{thm:vw2} have now been checked. Set
\[
\theta_{\n,i}(x) := 2n_i\big( (|\n|/n_i) V^{\omega_{\n,i}}(x) + \kappa_{\n,i}(x) \big) + h_{\n,i}(x), \quad x\in\Delta_i.
\]
It readily follows from \eqref{kappani}, \eqref{sapsad11}, and the definition above that
\[
T_{n_i} \left( e^{2 u_{\n,i} - \sum_{j\neq i}(H_{\Delta_{\n,j}}s_{\n,j} + |\n|V^{\omega_{\n,j}})}\mu_i \right)(z) = T_{n_i}(e^{\theta_{\n,i}-|\vec n|\ell_{\n,i}}\mu_i)(z) = T_{n_i}(e^{\theta_{\n,i}}\mu_i)(z),
\]
where the last equality holds because monic orthogonal polynomials do not depend on the normalization of the measure of orthogonality. Then, an application of Theorem~\ref{thm:vw2} gives
\begin{multline*}
\log|T_{n_i}(e^{\theta_{\n,i}}\mu_i)(z)| =  o_U(1) -|\n| V^{\omega_{\n,i}}(z) + \\ \log\left|\out_{\Delta_{\n,i}}\left(e^{h_{\n,i}/2}\sqrt{v_{\n,i}},\infty\right)\right| -  \log\left|\out_{\Delta_{\n,i}}\left(e^{h_{\n,i}/2}\sqrt{v_{\n,i}},z\right)\right|,
\end{multline*}
where we used \eqref{SzegoFun}. Here, \( o_U(1) \) is such that for every \( \epsilon>0 \) and every closed \( K_i \subset D_{\Delta_{\vec\varsigma,i}} \), there exists \( N_\epsilon \) such that \( \| o_U(1) \|_{C(K_i)} \leq \epsilon \), \( |\n|\geq N_\epsilon \), \( \n\in\mathcal N(\vec \varsigma) \), independently of the choice of \( \{\vec u_\n \}\subseteq U \). Now, notice that the error terms \( o_U(1) \) represent functions harmonic on $K_i$.  Hence, it follows from properties of harmonic functions, e.g., their integral representations, that for every $\epsilon>0$, there is $\widehat N_\epsilon$ such that \( \| o_U(1) \|_{C^1(K_i)} \leq \epsilon \), \( |\n|\geq \widehat N_\epsilon \), \( \n\in\mathcal N(\vec \varsigma) \), independently of the choice of \( \{\vec u_\n \}\subseteq U \).

 Finally, we apply  \eqref{vec-sc}  to the vector \( \vec s_\n \) to write
\begin{equation}
\label{simplify}
\sqrt{v_{\n,i}}(x)e^{(\mathcal H_{\vec\Delta_\n}\vec s_\n)_i(x)} = \sqrt{v_{\n,i}}(x)e^{s_{\n,i}(x)-(\vec a_\n)_i(x)} = e^{s_{\n,i}(x)}
\end{equation}
for almost every \( x\in \Delta_{\n,i} \), \( i\in I_d \). Notice that \(  h_{\n,i}/2 = u_{\n,i} + (\mathcal H_\n\vec s_\n)_i \) on \( \Delta_{\n,i} \) by the very definition of \( \mathcal H_{\vec\Delta_\n} \) in \eqref{operH}. Hence, \eqref{Dirichlet} yields that
\[
\log\left|\out_{\Delta_{\n,i}} \big(e^{h_{\n,i}/2} \sqrt{v_{\n,i}},z\big) \right| = H_{\Delta_{\n,i}}(s_{\n,i}+u_{\n,i})(z),
\]
which finishes the proof of the lemma. 
\end{proof}

Operators \( \mathcal H_{\vec\Delta_{\vec c}} \) were designed to construct vectors of Szeg\H{o} functions based on the boundary value problem satisfied by them. As the statement of Theorem~\ref{thm:sa4} implies, (normalized) Szeg\H{o} functions also appear as the deviation of the multiple orthogonal polynomials from their expected geometric behavior. This motivates introduction of the following nonlinear operators. For \( j\neq i \), \( j,i\in I_d \), define 
\begin{align*}
D_{\n,\Delta_j\to\Delta_i} : & ~C(\Delta_j) \to C(\Delta_i), \\ & u \mapsto \left( \log\left| T_{n_j}\left(e^{2u-|\n|\sum_{k\neq j}V^{\omega_{\n,k}}}\mu_j\right) \right| + |\n|V^{\omega_{\n,j}} \right)_{|\Delta_i}.
\end{align*}
It readily follows from \eqref{vec_eq2}  that in the case when \( \Delta_{\n,j} = \Delta_j \) we can replace the sum \( \sum_{k\neq j}V^{\omega_{\n,k}} \) by \( -2V^{\omega_{\n,j}} \) in this definition as monic orthogonal polynomials do not depend on the normalization of the measure of orthogonality. Further, let \( D_{\n,\Delta_i\to\Delta_i} \) be the operator whose image is the zero function, \( i\in I_d \). Put
\[
\mathcal D_\n := \frac12\begin{pmatrix} D_{\n,\Delta_j\to\Delta_i} \end{pmatrix}_{i,j=1}^d: C(\vec \Delta) \to C(\vec\Delta),
\]
where \( i \) is the row index, \( j \) is the column one, and \( \vec\Delta=(\Delta_1,\Delta_2,\ldots,\Delta_d\)).  In the matrix form, this gives
\[
\mathcal D_\n =\frac 12\left[
\begin{array}{cccc}
0& D_{\n,\Delta_2\to\Delta_1}&\cdots&D_{\n,\Delta_d\to\Delta_1}\\
D_{\n,\Delta_1\to\Delta_2}&0&\cdots&D_{\n,\Delta_d\to\Delta_2}\\
\vdots&\vdots&\ddots&\vdots\\
D_{\n,\Delta_1\to\Delta_d}&D_{\n,\Delta_2\to\Delta_d}&\cdots&0
\end{array}
\right]\,.
\]
As each component of \( \mathcal D_\n\,\vec u  \) is a restriction of a harmonic function, \( \mathcal D_\n(C(\vec \Delta))\subset C^1(\vec \Delta) \).
\begin{lemma}
\label{lem:sa18}
Let \( P_{\n,i}(x) \), \( i\in I_d \), be as in \eqref{ortho} and \eqref{partition}.  Set 
\[
q_{\n,i}(x) := \frac12\sum_{j\in I_d,j\neq i } \big( \log | P_{\n,j}(x) | + |\n| V^{\omega_{\n,j}}(x) \big) , \quad x\in\Delta_i, \quad i\in I_d.
\]
Then, \( \vec q_\n = (q_{\n,1},q_{\n,2},\ldots,q_{\n,d}) \) is the unique solution of \( \vec u = \mathcal D_\n \, \vec u \) in \( C(\vec\Delta) \).
\end{lemma}
\begin{proof}
Since each \( q_{\n,i}(x) \) is a restriction of a harmonic function, \( \vec q_\n \in C(\vec\Delta_\n) \). Moreover, for each \( i\in I_d \), we have 
\[
T_{n_i}\left(e^{2q_{\n,i}-|\n|\sum_{j\neq i}V^{\omega_{\n,j}}}\mu_i\right)(z) = T_{n_i}\left(|P_\n/P_{\n,i}|\mu_i\right)(z) = P_{\n,i}(z),
\]
by the  definition of \( P_\n(x) \) because \( P_\n(x)/P_{\n,i}(x) \) does not change sign on \( \Delta_i \). Hence, 
\[
\big(\mathcal D_\n\,\vec q_\n\big)_i(x) = \frac12\sum_{j\in I_d,j\neq i} \big( \log |P_{\n,j}(x)| + |\n| V^{\omega_{\n,j}}(x) \big) = q_{\n,i}(x), \quad x\in\Delta_i,
\]
i.e., \( \vec q_\n(x) \) is a fixed point of \( \mathcal D_\n \). 

Conversely, let \( \mathcal D_\n \,\vec u= \vec u \) for some \( \vec u=(u_1,u_2,\ldots,u_d)\in C(\vec \Delta) \). Set 
\[
T_\n(x) = \prod_{i\in I_d} T_{\n,i}(x), \quad T_{\n,i} = T_{n_i}\left(e^{2u_i-|\n|\sum_{j\neq i}V^{\omega_{\n,j}}}\mu_i\right), \quad i\in I_d.
\]
Then, the equality \( \mathcal D_\n \, \vec u= \vec u \) yields that
\[
2u_i(x)  = \sum_{j\in I_d,j\neq i}\big( \log |T_{\n,j}(x)| + |\n| V^{\omega_{\n,j}}(x) \big), \quad x\in \Delta_i, \quad i\in I_d.
\]
The last two displayed formulae imply that \( T_{\n,i} = T_{n_i}(|T_\n/T_{\n,i}|\mu_i) \), \( i\in I_d \). That is, the polynomial \( T_\n \) satisfies orthogonality conditions \eqref{ortho} with respect to the measures \( \mu_i \), \( i\in I_d \). It follows from the uniqueness of the type II multiple orthogonal polynomials for Angelesco systems that \( T_\n = P_\n\) and the proof is completed.
\end{proof}

Let now \( \mathcal H_\n := \mathcal H_{\vec\Delta_\n,\vec\Delta} \), see \eqref{operHrest}, and recall that $\vec s_{\vec c}$ was defined in \eqref{vec-sc}. In what follows, we shall slightly abuse the notation and write \( \mathcal H_\n \vec s_\n \) without specifying the exact extension of \( \vec s_\n \) to a vector-function on \( \vec\Delta \) as the images under \( \mathcal H_\n \) of all such extensions are identical.  

For each \( \vec u=(u_1,u_2,\ldots,u_d) \) in \( L^1(\omega_{\vec \Delta_{\vec c}}) \), we introduce the following notation
\begin{equation}
\label{u_infty}
\vec u_{\vec c,\infty} := \big( (H_{\Delta_{\vec c,1}}u_1)(\infty),(H_{\Delta_{\vec c,2}}u_2)(\infty),\ldots,(H_{\Delta_{\vec c,d}}u_d)(\infty) \big).
\end{equation}
We also abbreviate \( (\vec s_\n)_{\n,\infty} \) into \( \vec s_{\n,\infty} \). Observe that
\begin{equation}
\label{H-consts}
2(\mathcal H_\n \vec k)_i(x) = -\sum_{j\in I_d,j\neq i} (H_{\Delta_\n,j} k_j)(x) = - \sum_{j\in I_d,j\neq i} k_j
\end{equation}
for any vector of constants \( \vec k=(k_1,k_2,\ldots,k_d) \).

To obtain the asymptotics of $P_{\n,i}$, we will first find the asymptotics of $q_{\n,i}$. 
The expected behavior of the polynomials \( P_{\n,i}(z) \) suggests that
\begin{align*}
e^{2q_{\n,i}(x)} & = \prod_{j\in I_d,j\neq i}\left| P_{\n,j}(x) \right| e^{|\n|V^{\omega_{\n,j}}(x)}   \\
& ``\sim" \prod_{j\in I_d,j\neq i} \left|\frac{S_{\n,j}(\infty)}{S_{\n,j}(x)}\right|  \stackrel{\eqref{sapsad15}+\eqref{operH}}{=} e^{2(\mathcal H_\n \vec s_\n)_i(x) - 2(\mathcal H_\n (\vec s_{\n,\infty})_i(x)}\,.
\end{align*}
 Hence,  we expect that
\begin{equation}
\label{yn}
\vec q_\n \, ``\sim" \,  \mathcal H_\n \big(\vec s_\n - \vec s_{\n,\infty}\big) =: \vec y_\n.  
\end{equation}
Since $\vec q_\n$ is the fixed point of \( \mathcal D_\n \),   these heuristics suggest to study \( \mathcal D_\n \) in the vicinity of the vector-function \( \vec y_\n \), which we do now in the sequence of lemmas.

\begin{lemma}
\label{lem:sa19}
For any \( \epsilon>0 \) and any compact \( U\subset C(\vec \Delta)\) there exists \( N_U(\epsilon) \) such that
\[
\big\| \Upsilon_\n \, \vec u  \big\|_{C^1(\vec\Delta)} \leq \epsilon
\]
for all \( \vec u \in U \) and  \( |\n| \geq N_U(\epsilon) \), where
\begin{align*}
\Upsilon_\n  :&~ C(\vec \Delta) \to C(\vec\Delta), \\
&~ \vec u \mapsto \mathcal D_\n(\vec y_\n + \vec u)  - \mathcal H_\n \big (\vec u -\vec u_{\n,\infty} \big) - \vec y_\n.
\end{align*}
\end{lemma}
\begin{proof}
Since monic orthogonal polynomials do not depend on the normalization of the measure of orthogonality, we get from Lemma~\ref{lem:sa17} and the definition of the operators \( \mathcal D_\n \)  that
\begin{multline*}
\big(\mathcal D_\n (\vec y_\n + \vec u)\big)_i(x) = \big(\mathcal D_\n\,(\mathcal H_\n\vec s_\n + \vec u)\big)_i(x) \\
= \frac12\sum_{j\in I_d,j\neq i} \left( \log\left| T_{n_j}\left(e^{2u_i-\sum_{k\neq j}(H_{\Delta_{\n,k}}s_{\n,k} + |\n|V^{\omega_{\n,k}})}\mu_j\right)(x) \right| + |\n|V^{\omega_{\n,j}}(x) \right)  \\ 
= o_U(1) + \frac12\sum_{j\in I_d,j\neq i} \left( H_{\Delta_{\n,i}}(s_{\n,i}+u_i)(\infty) - H_{\Delta_{\n,i}}(s_{\n,i}+u_i)(x)\right),
\end{multline*}
where, as usual, \( \vec u = (u_1,u_2,\ldots,u_d) \), and for any \( \epsilon>0 \) there exists \( N_U(\epsilon) \) such that \( \|o_U(1)\|_{C^1(\vec\Delta)} \leq \epsilon \) for all \( |\n|\geq N_U(\epsilon) \), \( \n\in\mathcal N(\vec \varsigma) \), regardless the choice of \( \vec u \in U \). Now, it only remains to observe that
\[
\mathcal D_\n (\vec y_\n + \vec u) = o_U(1) + \mathcal H_\n\big(\vec u - \vec u_{\n,\infty} \big) + \vec y_\n.
\]
by \eqref{operHrest}, \eqref{u_infty},  \eqref{yn} as well as identity \eqref{H-consts}.
\end{proof}

In the view of the definition of the operator \( \Upsilon_\n \), the fact that \( \vec q_\n \) is the fixed point of \( \mathcal D_\n \) can now be rewritten as
\begin{equation}
\label{restatement}
(\vec q_\n - \vec y_\n) - \mathcal K_\n (\vec q_\n - \vec y_\n) = (\mathcal I-\mathcal H_\n)^{-1}\Upsilon_\n (\vec q_\n - \vec y_\n),
\end{equation}
where \( \mathcal K_\n \) is a bounded linear operator given by
\begin{align*}
\mathcal K_\n  :&~ C(\vec \Delta) \to C(\vec\Delta), \\
&~ \vec u \mapsto - (\mathcal I-\mathcal H_\n)^{-1}\mathcal H_\n \, \vec u_{\n,\infty} = (\mathcal I - (\mathcal I-\mathcal H_\n)^{-1}) \,\vec u_{\n,\infty}.
\end{align*}
The formula \eqref{u_infty} shows that \( \mathcal K_\n \) has rank at most $d$ and hence it is a compact  operator.
\begin{lemma}
\label{lem:sa20}
Let \( \mathfrak B \) be either \( C(\vec \Delta) \) or \( C^1(\vec \Delta) \).  \( \mathcal I - \mathcal K_\n \) is an invertible linear operator from \( \mathfrak B \) into itself. Moreover, there exists a constant \( C \) such that 
\[
\|(\mathcal I - \mathcal K_\n)^{-1}\|_{\mathfrak B} \leq C, \quad \n\in\mathcal N(\vec\varsigma). 
\]
\end{lemma}
\begin{proof}
Since \( \mathcal H_\n \, \vec u_{\n,\infty} \) is a vector of constants and \( \vec k_{\n,\infty} = \vec k \) for every vector of constants \( \vec k \), we get from the very definition of \( \mathcal K_\n \) that
\begin{equation}\label{sapsad16}
(\mathcal I - \mathcal K_\n) (\vec u - \mathcal H_\n\, \vec u_{\n,\infty} ) = \vec u - \vec u_{\n,\infty} + (\mathcal I-\mathcal H_\n)^{-1} ( \vec u_{\n,\infty} - \mathcal H_\n \, \vec u_{\n,\infty})  = \vec u
\end{equation}
for every \( \vec u\in \mathfrak B \). It follows from the linearity of $\mathcal H_\n$ that
\[
\vec u - \mathcal H_\n\, \vec u_{\n,\infty} = \vec v - \mathcal H_\n\, \vec v_{\n,\infty} \quad \Leftrightarrow \quad (\vec u-\vec v) = \mathcal H_\n (\vec u-\vec v)_{\n,\infty}.
\]
Thus, for the above equalities to be true, \( \vec u-\vec v  \) must be a vector of constants and therefore is a fixed point of \( \mathcal H_\n \). Hence, \( \vec u=\vec v \) by Lemma~\ref{lem:sa5}. Therefore, \( \mathcal I - \mathcal K_\n \) is indeed an invertible linear operator. Acting on both sides of \eqref{sapsad16} with $(\mathcal I - \mathcal K_\n)^{-1}$ gives
\[
(\mathcal I - \mathcal K_\n)^{-1} \vec u = \vec u - \mathcal H_\n\, \vec u_{\n,\infty}.
\] 
It holds that \( \|\vec u_{\n,\infty}\|_{\mathfrak B} \leq \|\vec u\|_{\mathfrak B} \) by the maximum principle for harmonic functions, see \eqref{u_infty}. Hence, 
\[
\|(\mathcal I - \mathcal K_\n)^{-1}\|_{\mathfrak B} \leq 1 + \| \mathcal H_\n\|_{\mathfrak B}.
\]
The first estimate of Lemma~\ref{lem:sa6} can be easily quantified to show that \( \| \mathcal H_\n\|_{\mathfrak B} \) is  bounded above by a number that depends only on the size of the convex hull of   \( \cup_{i\in I_d}\Delta_i \) and the shortest distance between \(  \Delta_i \) and \( \Delta_j \), \( i\neq j \), \( i,j\in I_d\). This yields the desired claim.
\end{proof}

Lemmas~\ref{lem:sa18} and~\ref{lem:sa20} as well as \eqref{restatement} allow us to make the crucial observation that \( \vec q_\n - \vec y_\n \) is {\it the unique fixed point of the operator \( (\mathcal I-\mathcal K_\n)^{-1}(I-\mathcal H_\n)^{-1}\Upsilon_\n \).} We now use the  Schauder-Tychonoff fixed-point theorem and Lemma~\ref{lem:sa19} to show that these fixed points must lie  arbitrarily close to zero for all \( |\n| \) large enough.

\begin{lemma}
\label{lem:sa21}
Given \( \delta>0 \), there exists \( N_{\vec\varsigma}(\delta) \) such that
\[
 \vec q_\n - \vec y_\n \in B_\delta := \big\{\vec u:~\|\vec u\|_{C^1(\vec\Delta)} \leq \delta\big\}
 \]
 for all \( |\n| \geq N_{\vec\varsigma}(\delta) \), \( \n\in \mathcal N(\vec\varsigma) \).
\end{lemma}
\begin{proof}
Observe that \( B_\delta \) is a compact convex subset of a locally convex space \( C(\vec\Delta) \). It follows from Lemma~\ref{lem:sa9} and Lemma~\ref{lem:sa20} that
\[
\big\| (\mathcal I-\mathcal K_\n)^{-1}(I-\mathcal H_\n)^{-1}\Upsilon_\n \, \vec u \big\|_{C^1(\vec\Delta)} \leq C_{\vec\varsigma} \big\| \Upsilon_\n \, \vec u \big\|_{C^1(\vec\Delta)}
\]
for some constant \( C_{\vec\varsigma} \) independent of \( \n\in\mathcal N(\vec\varsigma) \). We further get from Lemma~\ref{lem:sa19} that there exists \( N_{\vec\varsigma}(\delta) \) such that
\[
C_{\vec\varsigma}\big\| \Upsilon_\n \, \vec u \big\|_{C^1(\vec\Delta)} \leq \delta
\]
for all \( |\n| \geq N_{\vec\varsigma}(\delta) \), \( n\in \mathcal N(\vec\varsigma) \). Therefore, it holds that
\[
\big( (\mathcal I-\mathcal K_\n)^{-1}(I-\mathcal H_\n)^{-1}\Upsilon_\n \big) (B_\delta) \subseteq B_\delta
\]
for all \( |\n| \geq N_{\vec\varsigma}(\delta) \), \( \n\in \mathcal N(\vec\varsigma) \). As \( (\mathcal I-\mathcal K_\n)^{-1}(I-\mathcal H_\n)^{-1}\Upsilon_\n \) is a continuous operator from \( C(\vec\Delta) \) into itself, it must have a fixed point in \( B_\delta \) according to Schauder-Tychonoff fixed-point theorem \cite[Theorem~V.10.5]{MR1009162}. Since \(  \vec q_\n - \vec y_\n \) is its unique fixed point, the desired claim follows.
\end{proof}

\begin{lemma}
\label{lem:sa22}
Under the conditions of Theorem~\ref{thm:sa4} (with \( \vec c \) replaced by \( \vec\varsigma \)) it holds that
\[
P_{\n,i}(z) = (1+o(1)) \exp\left( |\n| \int\log(z-x)d\omega_{\n,i}(x)\right) \frac{S_{\vec \varsigma,i}(\infty)}{S_{\vec \varsigma,i}(z)} 
\]
locally uniformly in \( D_{\Delta_{\vec \varsigma,i}} \) for all \( |\n| \) large enough, \( \n\in\mathcal N(\vec \varsigma) \), and each \( i\in I_d \).
\end{lemma}
\begin{proof}
Recall the definition of the functions \( \kappa_{\n,i}(x) \) in \eqref{kappani}. Since monic orthogonal polynomials do not depend on the normalization of the measure of orthogonality, it readily follows from the definition of the vector-functions \( \vec q_\n \) in Lemma~\ref{lem:sa18} that
\[
P_{\n,i}(z) = T_{n_i}\left(e^{2q_{\n,i}+2n_i((|\n|/n_i)V^{\omega_{\n,i}}+\kappa_{\n,i})}\mu_i\right)(z).
\]
It has been shown in Lemma~\ref{lem:sa21} that \( \vec q_\n = \vec y_\n + \vec u_\n \), where \( \|\vec u_\n\|_{C^1(\vec \Delta)} \to 0 \) as \( |\n|\to\infty \), \( \n\in\mathcal N(\vec \varsigma) \). It clearly follows from \eqref{Conv-H} and their definition in \eqref{yn} that the vector-functions \( y_\n \) form a uniformly convergent sequence whose limit, when restricted to  \( \vec\Delta_{\vec\varsigma} \), is equal to \( \mathcal H_{\vec\Delta_{\vec\varsigma}} \big( \vec s_{\vec\varsigma} - \vec s_{\vec\varsigma,\infty}\big) \).  As the sequence \( \{ \vec q_\n \} \) is convergent, it represents a  precompact set. Hence, we  conclude that Theorem~\ref{thm:vw2} is applicable with $h_n=2q_{\n,i}$ and we have
\begin{eqnarray*}
P_{\n,i}(z) = (1+o(1)) \exp\left( |\n| \int\log(z-x)d\omega_{\n,i}(x)\right) \times
\\  \frac{
G_{\exp(2\mathcal H_{\vec\Delta_{\vec\varsigma}} ( \vec s_{\vec\varsigma} - \vec s_{\vec\varsigma,\infty})_i)}(\infty)
}{
G_{\exp(2\mathcal H_{\vec\Delta_{\vec\varsigma}}( \vec s_{\vec\varsigma} - \vec s_{\vec\varsigma,\infty})_i)}(z)
}
\frac{
G_{{\mu_i}_{|\Delta_{\vec\varsigma,i}}}(\infty)
}{
G_{{\mu_i}_{|\Delta_{\vec\varsigma,i}}}(z)
}
\end{eqnarray*}
as \( |\n|\to\infty \), \( \n\in\mathcal N(\vec \varsigma) \). Every component of function $\mathcal H_{\vec\Delta_{\vec\varsigma}} \vec s_{\vec\varsigma,\infty} $ is a constant function and it will cancel out in the fraction.  We recall definitions \eqref{outer} and \eqref{SzegoFun} to conclude that
\begin{equation}
G_{\exp(2\mathcal H_{\vec\Delta_{\vec\varsigma}}( \vec s_{\vec\varsigma} )_i)}(z) G_{{\mu_i}_{|\Delta_{\vec\varsigma,i}}}(z) =
\Omega_{\Delta_{\vec\varsigma,i}} \big(\exp(\mathcal H_{\vec\Delta_{\vec\varsigma}}( \vec s_{\vec\varsigma} )_i,z \big)  \, \Omega_{\Delta_{\vec\varsigma,i}}(\sqrt{\nu_{\Delta_{\vec\varsigma,i}}},z).
\label{sapsad21}
\end{equation}
Formula \eqref{sapsad20} yields
\[
\exp(\vec s_{\vec\varsigma,i})=\sqrt{\nu_{\Delta_{\vec\varsigma,i}}}
\exp\Bigl((\mathcal{H}_{\Delta_{\vec\varsigma}}\vec s_{\vec\varsigma})_i\Bigr)
\]
and formula \eqref{apsad2} gives $S_{\vec\varsigma,i}(z)=\Omega_{\Delta_{\vec\varsigma,i}}(\exp(\vec s_{\vec\varsigma,i}),z)$. Substitution into \eqref{sapsad21} provides the required asymptotics.
\end{proof}

\section{Strong Asymptotics of OPs with Varying Weights}
\label{sec:vw}

In this section, we obtain two results that were crucial for proving Theorem \ref{thm:sa4}. They improve and generalize \cite[Theorem~14.4]{Totik} by Totik.

\subsection{Main Theorems}

We recall the following notation: given a non-negative Borel measure \( \mu \) on an interval \( \Delta \) and a continuous non-negative function \( f \) on \( \Delta \), we denote by \( T_n(f\mu) \) the \( n \)-th monic orthogonal polynomial with respect to the measure \( f\mu \), where, with a slight abuse of notation, we write \( f\mu \) for the measure \( fd\mu \). 

\begin{theorem}
\label{thm:vw1}
Let \( \Delta=[\alpha,\beta] \) and \( \{(\mu_n,h_n,\omega_n)\} \) be a sequence of triples, where \( \mu_n,\omega_n \) are measures on \( \Delta \) and \( h_n \) is a continuous function on \( \Delta \). Assume further that
\begin{itemize}
\item[$(A)$] there exists a finite measure $\mu$ on $\Delta$ such that for any non-negative function \( f\in C(\Delta) \) it holds that
\[
\limsup_{n\to\infty} \int fd\mu_n \leq \int fd\mu;
\]
\item[$(B)$] if $v_n$ and $v$ denote the Radon-Nikodym derivatives of \( \mu_n \) and \( \mu \) with respect to the arcsine distribution \( \omega_\Delta \), see \eqref{szego}, then  \( \|\log v_n- \log v\|_{L^1(\omega_\Delta)}\to 0 \) as \( n\to\infty \); \smallskip
\item[$(C)$] the functions \( h_n(x) \) belong to \( \mathcal K \), a fixed compact subset of \( C(\Delta) \); \smallskip
\item[$(D)$] \( d\omega_n(x)=\omega_n^\prime(x)dx \) are probability measures such that the functions \( \omega_n^\prime(x) \) form a uniformly equicontinuous family on each compact subset of \( (\alpha,\beta) \) and there is $\tau>0$ such that
\begin{align*}
\omega_n^\prime(x)  & \gtrsim |w_\Delta(x)|^{\varkappa_L}, \quad x\in [\alpha+n^{-\tau},\beta-n^{-\tau}], \\
\omega_n^\prime(x) &  \lesssim |w_\Delta(x)|^{\varkappa_U}, \quad  x\in (\alpha,\beta),
\end{align*}
for some \( \varkappa_L,\varkappa_U>-2 \), see \eqref{phi-w}.
\end{itemize}
Write \( \theta_n(x) = 2nV^{\omega_n}(x)+h_n(x) \). Then, it holds locally uniformly in \( D_\Delta = \overline\C\setminus \Delta \) that
\begin{equation}\label{sapsad32}
T_n\left(e^{\theta_n}\mu_n\right)(z) = \left(1+o_{\mathcal K}(1)\right) \exp\left( n \int\log(z-x)d\omega_n(x)\right) \frac{G(e^{h_n}\mu_n,\infty)}{G(e^{h_n}\mu_n,z)}.
\end{equation}
Moreover,
\[
\int_\Delta T_n^2\left(e^{\theta_n}\mu_n\right)(x)e^{\theta_n(x)}d\mu_n(x) = 2\big(1+o_{\mathcal K}(1)\big)G^2\big(e^{h_n}\mu_n,\infty\big).
\]
\end{theorem}

Theorem~\ref{thm:vw1} generalizes \cite[Theorem~14.4]{Totik} in the following ways: it replaces a single absolutely continuous measure \( \mu \) with a sequence of not necessarily absolutely continuous measures \( e^{h_n}\mu_n \). In \cite{Totik}, the analog of Theorem~\ref{thm:vw1} is deduced from \cite[Theorem~10.2]{Totik} combined with the work in \cite{CYLL98}. We follow the same approach but replace the relevant results in \cite{CYLL98} by Theorem~\ref{thm:pw8} proven further below in Section~\ref{sec:pw}. Besides Theorem~\ref{thm:vw1}, we also need the following generalization which is used when the ``pushing effect'' is present in the vector-potential problem \eqref{vec_eq1}--\eqref{vec_eq2} for Angelesco systems.

\begin{theorem}
\label{thm:vw2}
Let \( \mu \) be a compactly supported positive Borel measure and \( [\alpha_n,\beta_n] = \Delta_n \subseteq \Delta(\mu) = [\alpha(\mu),\beta(\mu)] \) be intervals that converge to some interval \( \Delta =[\alpha,\beta]\). Assume that \( \mu\in\mathrm{USz}(\Delta) \). Further, let \( \{(\kappa_n,h_n,\omega_n)\} \) be a sequence of triples, where \( \kappa_n,h_n \) are continuous functions on \( \Delta(\mu) \) and \( \omega_n \) are measures on \( \Delta(\mu) \) such that
\begin{itemize}
\item[$(1)$] the functions \( \kappa_n(x) \) are such that \( \kappa_n(x)\leq 0 \) on \( \Delta(\mu) \), \( \kappa_n(x)\equiv 0 \) on \( \Delta_n \), and it holds on \( \Delta(\mu)\setminus\Delta_n \) that \( |w_{\Delta_n}(x)|^\varkappa \lesssim |\kappa_n(x)| \lesssim |w_{\Delta_n}(x)|^\varkappa \) for some \( \varkappa>0 \); \smallskip
\item[$(2)$] the functions \( h_n(x) \) belong to \( \mathcal K \), a fixed compact subset of \( C(\Delta(\mu)) \); \smallskip
\item[$(3)$] \( d\omega_n(x)=\omega_n^\prime(x)dx \) are probability measures such that \( \supp\omega_n = \Delta_n \), the densities \( \{\omega_n^\prime(x)\} \)  form a uniformly equicontinuous family on any compact subset of $(\alpha,\beta)$, and 
\begin{equation}
\label{w_cond_1}
|w_{\Delta_n}(x)|^{\varkappa_L} \lesssim \omega_n^\prime(x) \lesssim|w_{\Delta_n}(x)|^{\varkappa_U}, \quad x\in(\alpha_n,\beta_n),
\end{equation}
for some \( \varkappa_L,\varkappa_U>-2 \); in addition, for all \( n \) such that \( \alpha_n>\alpha(\mu) \), we assume that there exists $\widehat\varkappa_U>0$ for which
\begin{equation}
\label{w_cond_2}
\omega_n^\prime(x)\lesssim |x-\alpha_n|^{\widehat\varkappa_U}, \quad x\in(\alpha_n,\alpha_n+\delta),
\end{equation}
for some $\delta>0$, and a similar assumption is made for all \( n \) for which  \( \beta_n<\beta(\mu) \).

\end{itemize}
Set \( \theta_n(x) := 2n(V^{\omega_n}(x) + \kappa_n(x))+h_n(x) \). Then, it holds locally uniformly in \( D_\Delta  \) that
\[
T_n\left(e^{\theta_n}\mu\right)(z) = \left(1+o_{\mathcal K}(1)\right) \exp\left( n \int\log(z-x)d\omega_n(x)\right) \frac{G(e^{h_n}\mu_{|\Delta},\infty)}{G(e^{h_n}\mu_{|\Delta},z)}.
\]
Moreover,
\[
\int_{\Delta(\mu)} T_n^2\left(e^{\theta_n}\mu\right)(x)e^{\theta_n(x)}d\mu(x) = 2\big(1+o_{\mathcal K}(1)\big)G^2\big(e^{h_n}\mu_{|\Delta},\infty\big).
\]
In the above two formulae, the functions \( G(e^{h_n}\mu_{|\Delta},z) \) can be replaced by \( G(e^{h_n}\mu_{|\Delta_n},z) \).
\end{theorem}

\subsection{Proof of Theorem~\ref{thm:vw1}}

We prove Theorem~\ref{thm:vw1} in three steps that we organize as separate lemmas.

\begin{lemma}
\label{lem:vw3}
It is enough to prove Theorem~\ref{thm:vw1} for \( \Delta=[-1,1] \) only.
\end{lemma}
\begin{proof}
Let \( l(z) = az+b \) be any linear transformation with \( a>0 \) and \( b \) real. Set \( \Delta^{(l)}:=l^{-1}(\Delta) \) and, given a measure \( \mu \) on \( \Delta \), let \( \mu^{(l)} \) denote a Borel measure on \( \Delta^{(l)} \) such that \( \mu^{(l)}(B) = \mu(l(B)) \) for any Borel set \( B\subseteq \Delta^{(l)} \). Notice that \( \omega_\Delta^{(l)} = \omega_{\Delta^{(l)}} \) and that the Radon-Nikodym derivative of \( \mu^{(l)} \) with respect to the Lebesgue measure (resp. \( \omega_{\Delta^{(l)}} \)) is equal to \( a\mu^\prime(l(x)) \) (resp. \( v(l(x) \)), where \( \mu^\prime(x) \) (resp. \( v(x) \)) is the Radon-Nikodym derivative of \( \mu \) with respect to the Lebesgue measure (resp. \( \omega_\Delta \)). Observe also that
\[
V^\omega(l(x)) = -\int\log|l(x)-l(y)|d\omega^{(l)}(y) = -\log a + V^{\omega^{(l)}}(x)
\]
for any Borel measure \( \omega \). Hence, it holds that
\[
T_n\left(e^{2nV^{\omega_n}+h_n}\mu_n\right)(l(z)) = a^n  T_n\left(e^{2nV^{\omega_n^{(l)}}+h_n^{(l)}}\mu_n^{(l)}\right)(z),
\]
where \( h_n^{(l)} = h\circ l \). Since \( w_\Delta(l(z)) = aw_{\Delta^{(l)}}(z) \), the above considerations also show that triples \( (\mu_n^{(l)},h_n^{(l)},\omega_n^{(l)}) \) satisfy conditions \( (A)-(D) \) on \( \Delta^{(l)} \) if the triples \( (\mu_n,h_n,\omega_n) \) satisfy  \( (A)-(D) \) on \( \Delta \). Finally, we get from \eqref{outer} and \eqref{SzegoFun} that
\[
e^{n \int\log(l(z)-x)d\omega_n(x)} \frac{G(e^{h_n}\mu_n,\infty)}{G(e^{h_n}\mu_n,l(z))} = a^ne^{n \int\log(z-x)d\omega_n^{(l)}(x)} \frac{G(e^{h_n^{(l)}}\mu_n^{(l)},\infty)}{G(e^{h_n^{(l)}}\mu_n^{(l)},z)},
\]
which finishes the proof of the lemma.
\end{proof}

Condition \( (D) \), placed on the measures \( \omega_n \) in Theorem~\ref{thm:vw1}, comes from \cite[Theorem~10.2]{Totik}. Under this assumption it was shown there that there exist polynomials \( H_n(x) \), \( \deg H_n \leq n \), that do not vanish on \( [-1,1] \) and such that the functions
\[
\iota_n(x) := e^{2n V^{\omega_n}(x)}|H_n(x)|^2, \quad x\in[-1,1],
\]
satisfy
\begin{equation}
\label{iota}
\begin{cases}
 0< \iota_n(x)\leq 1, \quad x\in [-1,1], \\
 \lim_{n\to\infty}\int\log\iota_n d\omega = 0, \quad \omega=\omega_{[-1,1]}.
\end{cases}
\end{equation}
We remark that \cite[Theorem~10.2]{Totik} was formulated on \( [0,1] \), but its results can be easily brought to \( [-1,1] \) by a linear transformation. In that theorem,  we put \( \gamma=1/2 \) and \( u\equiv 1 \),  and the degree satisfies \( \deg H_n = n-i_n \), where \( i_n\to\infty \). The non-vanishing of $H_n$ was claimed only on \( (-1,1) \), but it is clear from the construction, see \cite[pages~58 and~75]{Totik}, that these polynomials also do not vanish at the endpoints. Set
\begin{equation}\label{sapsad29}
\tau_n(z) := \frac{H_n(z)\overline{H_n(\bar z)}}{|H_n(0)|^2}.
\end{equation}
The polynomial \( \tau_n(z) \) has even degree,  \( \deg \tau_n\le 2n \), and $\tau_n$ satisfies  the following properties:
\begin{itemize}
\item[1.] \( \tau_n(z) \) has real coefficients,
\item[2.] $\tau_n(0)=1$,
\item[3.] $\tau_n$ does not vanish on $[-1,1]$,
\item[4.]  if we denote  the zeros of \( \tau_n(z) \) by \( \{a_{n,j}\} \), \( j\in\{1,2,\ldots,\deg \tau_n \} \), then at least a half of them, see \cite[page~94]{Totik}, are located in $\{|\re\, z|<0.9, |\im \,z|>L_n/n\}$ with $\lim_{n\to\infty}L_n=+\infty$. This guarantees that condition \( (D_{[-1,1]})\) of Theorem~\ref{thm:pw8} further below is satisfied by $\tau_n(z)$.
\end{itemize}

\begin{lemma}
\label{lem:vw4}
Under the conditions of Theorem~\ref{thm:vw1} with \( \Delta=[-1,1] \), it holds that
\begin{multline}
\label{sapsad31}
T_n^2\left(e^{\theta_n}\mu_n\right)(z) = \big(1+o_{\mathcal K}(1)\big) \frac{G^2(e^{h_n}\iota_n\mu_n,\infty)}{G^2(e^{h_n}\iota_n \mu_n,z)} \times \\  \frac1{2^{2n}} \prod_{j=1}^{\deg \tau_n}(2a_{n,j}\phi(a_{n,j})) \frac{\tau_n(z)}{\phi^{2n-\deg \tau_n}(z)}  \prod_{j=1}^{\deg\tau_n}\frac{1-\overline{\phi(a_{n,j})}\phi(z)}{\phi(z)-\phi(a_{n,j})} 
\end{multline}
locally uniformly in \( \overline\C\setminus[-1,1] \), where \( \phi(z) := \phi_{[-1,1]}(z) \), see \eqref{phi-w}. Moreover,
\[
\int_{-1}^1 T_n^2\left(e^{\theta_n}\mu_n\right)(x)e^{\theta_n(x)}d\mu_n(x) = \big(1+o_{\mathcal K}(1)\big) \frac{G^2(e^{h_n}\iota_n\mu_n,\infty)}{2^{2n-1}}\prod_{j=1}^{\deg \tau_n}(2\phi(a_{n,j})).
\]
\end{lemma}
\begin{proof}
Since monic orthogonal polynomials do not depend on the normalization of the measure of orthogonality, we have that
\[
T_n\left(e^{\theta_n}\mu_n\right)(z) = T_n\left(e^{h_n}\tau_n^{-1}\tilde\mu_n\right)(z), \quad \tilde\mu_n := \iota_n \mu_n. 
\]
Let us show that conditions \( (A_{[-1,1]}) - (D_{[-1,1]}) \) of Theorem~\ref{thm:pw8} are satisfied by the triples \( (\tilde\mu_n,h_n,\tau_n) \). We have already mentioned that the polynomials \( \tau_n(z) \) fulfill \( (D_{[-1,1]}) \). Moreover, condition \( (C) \) of Theorem~\ref{thm:vw1} is identical to condition \( (C_{[-1,1]}) \) of Theorem~\ref{thm:pw8}.  Since the Radon-Nikodym derivative of \( \tilde\mu_n \) with respect to \( \omega=\omega_{[-1,1]} \) is \( (\iota_n v_n)(x) \), the functions \( \iota_n(x) \) obey the first line of \eqref{iota}, and
\[
\big|\log(\iota_nv_n)(x) - \log v(x) \big| \leq \big|\log v_n(x) - \log v(x) | - \log\iota_n(x),
\]
condition \( (B_{[-1,1]}) \) of Theorem~\ref{thm:pw8} follows from condition \( (B) \) of Theorem~\ref{thm:vw1} and the second line of \eqref{iota}. Finally, condition \( (A) \) of Theorem~\ref{thm:vw1} implies condition \( (A_{[-1,1]}) \) of Theorem~\ref{thm:pw8} for the same measure \( \mu \) due to the upper bound in the first line of \eqref{iota}.  If $p_n(\sigma)(z)=\gamma_nz^n+\cdots$ denotes the $n$-th orthonormal polynomial with respect to the measure $\sigma$, then we can write
\[
T_n(\sigma)(z) = \gamma_n^{-1}p_n(\sigma)(z) \qandq \gamma_n^{-2}=\int T_n^2(\sigma)d\sigma.
\]
So, the first claim of the lemma is deduced from Theorem~\ref{thm:pw8}. To get the second one, we first observe that \eqref{sapsad29} implies \( |H_n(0)|^2 = \prod_j a_{n,j} \). Then, 
\[
\int_{-1}^1 T_n^2\left(e^{\theta_n}\mu_n\right)(x)e^{\theta_n(x)}d\mu_n(x)  = \frac1{|H_n(0)|^2} \int_{-1}^1 T_n^2\left(e^{h_n}\tau_n^{-1}\tilde\mu_n\right)(x)e^{h_n(x)}\frac{d\tilde\mu_n(x)}{\tau_n(x)}
\]
and we only need to apply \eqref{A8-2} to the last integral.
\end{proof}

\begin{lemma}
\label{lem:vw5}
Theorem~\ref{thm:vw1} holds on \( \Delta=[-1,1] \). 
\end{lemma}
\begin{proof}
We will show that the right-hand side of \eqref{sapsad31} can be written in a form consistent with \eqref{sapsad32}. We readily get from \eqref{SzegoFun} and \eqref{outer} that
\[
G^2(e^{h_n}\iota_n\mu_n,z) / G^2(e^{h_n}\mu_n,z) = \Omega(\iota_n,z) = \Omega\big(e^{2nV^{\omega_n}},z\big) \Omega\big(|H_n|^2,z\big),
\]
where  \( \Omega \) is used  as a shorthand for \( \Omega_{[-1,1]} \). Let us show that
\[
\begin{cases}
\Omega\big(|H_n|^2,z\big) & \displaystyle = \tau_n(z)\phi^{\deg \tau_n}(z) \left( \prod_{j=1}^{\deg \tau_n} a_{n,j}\right)\prod_{j=1}^{\deg \tau_n}\frac{1-\overline{\phi(a_{n,j})}\phi(z)}{\phi(z)-\phi(a_{n,j})}, \medskip \\
\Omega\big(e^{V^{\omega_n}},z\big) & \displaystyle = \phi^{-1}(z)  \exp\left\{ \int \log(z-x)d\omega_n(x)\right\}.
\end{cases}
\]
Both equalities follow from the same general principle: if \( f(x) \) is a continuous function on \( [-1,1] \) and \( \Omega(z) \) is a holomorphic non-vanishing function in \( \overline\C\setminus[-1,1] \) such that \( |\Omega(z)| \) is continuous in the entire extended complex plane and \( |\Omega(x)| = f(x) \) on \( [-1,1] \), then \( \Omega(z) = \Omega(f,z) \). Continuity of \( |H_n(x)|^2 \) is obvious while continuity of \( V^{\omega_n}(x) \) follows from condition \( (D) \) and properties of logarithmic potentials. Recall that \( \deg \tau_n \) is an even integer and that \( 2z\phi(z) \to 1 \) as \( z\to\infty \). Now,  to prove the lemma it only remains to notice that the explicit representations given above yield
\[
\Omega\big(|H_n|^2,\infty\big) = \prod_{j=1}^{\deg \tau_n}(2\phi(a_{n,j}))^{-1} \qandq \Omega\big(e^{V^{\omega_n}},\infty\big) = 2. \qedhere
\]

\end{proof}

\subsection{Proof of Theorem~\ref{thm:vw2}}

Similarly to Theorem~\ref{thm:vw1}, we prove Theorem~\ref{thm:vw2} in four steps organized as separate lemmas.

Write \( \Delta(\mu) = [\alpha(\mu),\beta(\mu)] \), \( \Delta = [\alpha,\beta] \), and \( \Delta_n = [\alpha_n,\beta_n] \).  Recall condition (1) of Theorem~\ref{thm:vw2}. Let \( \{\delta_n\} \) be a sequence of positive numbers such that
\begin{equation}
\label{delta_n}
\delta_n = (\xi_n/n)^{2/(\varkappa+1)},
\end{equation}
where \( \xi_n\to 0 \) as \( n\to\infty \) and it will be specified  later in the proof of Lemma~\ref{lem:vw8}, see \eqref{ser20}. Set \( \Delta_n^* := [\alpha_n^*,\beta_n^*] \), where \( \alpha_n^* := \max\big\{\alpha_n-\delta_n,\alpha(\mu) \big\} \) and \( \beta_n^* := \min\big\{\beta_n+\delta_n,\beta(\mu) \big\} \). Our strategy will consist in applying Theorem~\ref{thm:vw1}  to obtain the asymptotics of orthogonal polynomials $P_n^*$ for measures reduced to  $\Delta_n^*$. Then, we show that asymptotics of the polynomials $T_n$ from Theorem~\ref{thm:vw2} coincides with that of $P_n^*$.

\begin{lemma}
\label{lem:vw6}
Recall that \( \theta_n(x) = 2n(V^{\omega_n}(x) + \kappa_n(x))+h_n(x) \), \( x\in\Delta(\mu) \). It holds that
\[
P_n^*(z) := T_n\left(e^{\theta_n}\mu_{|\Delta_n^*}\right)(z) = (1+o_{\mathcal K}(1)) F_n^*(z)
\]
locally uniformly in \( D_\Delta  \), where
\begin{equation}\label{ser21}
F_n^*(z) := \exp\left( n \int\log(z-x)d\omega_n(x)\right) \frac{G(e^{h_n}\mu_{|\Delta_n^*},\infty)}{G(e^{h_n}\mu_{|\Delta_n^*},z)}.
\end{equation}
\end{lemma}
\begin{proof}
To apply  Theorem~\ref{thm:vw1}, we need to rescale the intervals \( \Delta_n^* \) to their limit \( \Delta \). To this end, let \( l_n(x) = a_nx+b_n \), \( a_n>0 \), be the linear function that maps \( \Delta \) onto \( \Delta_n^* \). Clearly, \( a_n\to1 \) and \( b_n\to0 \) as \( n\to\infty \). In the notation of Lemma~\ref{lem:vw3}, set \( \tilde \omega_n := \omega_n^{(l_n)} \). Then, it holds that
\[
\supp\tilde\omega_n = l_n^{-1}(\Delta_n)\subseteq \Delta \qandq d\tilde\omega_n(x)  = \tilde\omega_n^\prime(x)dx = a_n\omega_n^\prime(l_n(x))dx.
\]
Write \( l_n^{-1}(\Delta_n)=[\alpha_n^{(l)},\beta_n^{(l)}]\). Since \( a_n\to 1 \) as \( n\to\infty \), we get from the upper bound in assumption (3) of Theorem~\ref{thm:vw2} that
\[
\tilde\omega_n^\prime(x) \lesssim \big(x-\alpha_n^{(l)}\big)^{\varkappa_U/2}\big(\beta_n^{(l)}-x\big)^{\varkappa_U/2},  \quad x\in \big( \alpha_n^{(l)},\beta_n^{(l)} \big).
\]
If \(\alpha_n^{(l)}>\alpha \) for some index \( n \), then \( \alpha_n>\alpha_n^*\geq \alpha(\mu) \) and \( (x-\alpha_n^{(l)})^{\varkappa_U/2} \) can be replaced by \( (x-\alpha_n^{(l)})^{\widehat\varkappa_U} \) in the above estimate as required by assumption (3) of Theorem~\ref{thm:vw2}, where \( \widehat\varkappa_U>0 \). Similarly, if \( \beta_n^{(l)}<\beta \) for some \( n \), then \( (\beta_n^{(l)}-x)^{\varkappa_U/2} \) can again be replaced by \( (\beta_n^{(l)}-x)^{\widehat\varkappa_U} \). Either way, the upper bound in assumption \( (D) \) of Theorem~\ref{thm:vw1} is fulfilled. Similarly to the upper bound, we have that
\[
\tilde\omega_n^\prime(x) \gtrsim |w_{l_n^{-1}(\Delta_n)}(x)|^{\varkappa_L},  \quad x\in \big( \alpha_n^{(l)},\beta_n^{(l)} \big),
\]
by the lower bound in assumption (3) of Theorem~\ref{thm:vw2}. If \( l_n^{-1}(\Delta_n) =\Delta \) for all \( n \), the above inequality gives the desired lower bound in assumption \( (D) \) of Theorem~\ref{thm:vw1}. If at least one of the intervals \( l_n^{-1}(\Delta_n) \) is a proper subinterval of \( \Delta \), then the corresponding upper bound requires that \( \varkappa_L>0 \). It can be readily checked that
\[
|w_{l_n^{-1}(\Delta_n)}(x)| \gtrsim  |w_\Delta(x)|, \quad x\in \big(\alpha+2\big(\alpha_n^{(l)}-\alpha\big), \beta-2(\beta-\beta_n^{(l)}\big)\big).
\]
Notice that \( a_n( \alpha_n^{(l)}-\alpha) = \alpha_n-\alpha_n^* \leq \delta_n \) and similarly that \( a_n(\beta-\beta_n^{(l)}) \leq \delta_n \). It now  follows from \eqref{delta_n}  that there exists \( \tau>0 \) such that
\[
\tilde\omega_n^\prime(x) \gtrsim  |w_\Delta(x)|^{\varkappa_L}, \quad x\in(\alpha+n^{-\tau},\beta-n^{-\tau}).
\] 
Finally, equicontinuity of the functions \( \tilde\omega_n^\prime(x) \) on compact subsets of \( (\alpha,\beta) \) follows from the analogous conditions placed on the densities \( \omega_n^\prime(x) \). Altogether, all the requirements of assumption \( (D) \) of Theorem~\ref{thm:vw1} are satisfied.

Next, let \( \tilde h_n = h_n\circ l_n \). It follows from the locally uniform convergence to the identity of the functions \( l_n \) that  assumption $(C)$ of Theorem~\ref{thm:vw1} follows from  assumption (2) of Theorem~\ref{thm:vw2} with the compact subset of \( C(\Delta) \) being the closure of \( \cup_n\{h\circ l_n : h\in\mathcal K\} \). Finally, let
\[
d\tilde \mu_n(x) = e^{2n\tilde\kappa_n(x)}d(\mu_{|\Delta_n^*})^{(l_n)}(x), \quad \tilde \kappa_n = \kappa_n\circ l_n.
\]
Given a continuous function \( f \) on \( \Delta \), it holds that
\[
\int_\Delta fd(\mu^{l_n}-\mu) = \int_{\Delta\cap \Delta_n^*}( f\circ l_n^{-1}-f)d\mu + \int_{\Delta_n^*\setminus \Delta}  f\circ l_n^{-1} d\mu - \int_{\Delta\setminus\Delta_n^*} fd\mu.
\]
The first integral on the right-hand side above converges to zero due to uniform continuity of \( f \). The second one converges to zero because \( f \) is bounded and \( \cap_{n}\Delta_n^*\setminus \Delta=\emptyset\) so  \(\mu(\Delta_n^*\setminus \Delta)\to 0\) as \( n\to\infty \). The third integral is always non-negative. Hence, as \( \tilde\kappa_n(x) \leq 0 \), assumption $(A)$ of Theorem~\ref{thm:vw1} is clearly fulfilled. To verify assumption $(B)$, observe that
\[
\log\tilde v_n(x) = 2n\tilde \kappa_n(x) + \log v_{\Delta_n^*}(l_n(x)),
\]
where \( \tilde v_n \) is the Radon-Nikodym derivative of \( \tilde\mu_n \) with respect to \( \omega_\Delta \), see \eqref{szego}. We readily have that
\[
|\log v_\Delta(x) - \log v_{\Delta_n^*}(l_n(x))| = |\log\mu^\prime(x) - \log\mu^\prime(l_n(x)) - \log a_n|.
\]
Since \( a_n\to 1\), it follows from \eqref{UnSzego} that \( \| \log v_\Delta - \log v_{\Delta_n^*}(l_n)\|_{L^1(\omega_\Delta)} \to 0 \) as \( n\to\infty \). Furthermore, we have that
\[
\int_\Delta |\tilde\kappa_n(x)|d\omega_\Delta(x) = \int_{\Delta_n^*} |\kappa_n(x)| d\omega_{\Delta_n^*}(x)  \lesssim \int_{\Delta_n^*\setminus\Delta_n} |w_{\Delta_n}(x)|^\varkappa d\omega_{\Delta_n^*}(x),
\]
where we used assumption (1) of Theorem~\ref{thm:vw2}. Hence, it readily follows from the definition of \( \Delta_n^* \) that
\begin{equation}
\label{L1kappa}
2n\int_\Delta |\tilde\kappa_n(x)|d\omega_\Delta(x) \lesssim_\Delta n \int_0^{\delta_n} x^{(\varkappa+1)/2} d\omega_{[0,\delta_n]}(x) \lesssim_\Delta n \delta_n^{(\varkappa+1)/2} = \xi_n.
\end{equation}
Since $\lim_{n\to\infty}\xi_n=0$ by our assumptions, we get \( \| 2n\tilde \kappa_n \|_{L^1(\omega_\Delta)} \to 0 \) as \( n\to\infty \), which shows that condition $(B)$ of Theorem~\ref{thm:vw1} is also satisfied. Altogether, we have that the triples \( (\tilde\mu_n,\tilde h_n,\tilde\omega_n) \) satisfy all the conditions of Theorem~\ref{thm:vw1} and therefore
\[
T_n\left(e^{2nV^{\tilde\omega_n}+\tilde h_n}\tilde\mu_n\right)(z) = \left(1+o_{\mathcal K}(1)\right) \exp\left( n \int\log(z-x)d\tilde \omega_n(x)\right) \frac{G(e^{\tilde h_n}\tilde \mu_n,\infty)}{G(e^{\tilde h_n}\tilde\mu_n,z)}
\]
holds locally uniformly in \( D_\Delta  \). Computations in Lemma~\ref{lem:vw3} now show that the above formula is equivalent to the statement of the lemma.
\end{proof}

The next lemma provides a simple uniform estimate on the boundary behavior for the sequence of outer functions. Recall the definition of function $G$ given in \eqref{SzegoFun} and \eqref{outer}.
\begin{lemma}
\label{lem:vw7}
There exists a non-decreasing function \( \epsilon_\mu(t) \) such that \( \lim_{t\downarrow 0}\epsilon_\mu(t)=0 \) and
\[
\log G^{-2}\big(e^{h_n}\mu_{|\Delta_n},\beta_n+t\big) \leq \frac{\epsilon_\mu(t)}{\sqrt t}
\] 
for every \( t\in[0,\beta(\mu)-\beta_n] \) and every \( n \). Moreover, an analogous estimate holds with \( \beta_n+t \) replaced by \( \alpha_n-t \) for \( t\in[0,\alpha_n-\alpha(\mu)] \).
\end{lemma}
\begin{proof}
By \eqref{outer} and \eqref{SzegoFun} that
\[
\log G^{-2}\big(e^{h_n}\mu_{|\Delta_n},\beta_n+t\big) \leq \frac{w_{\Delta_n}(\beta_n+t)}{2} \int_{\Delta_n} \frac{|h_n(x)|+|\log v_{\Delta_n}(x)|}{\beta_n+t-x}d\omega_{\Delta_n}(x).
\]
It follows from the Cauchy integral formula that
\[
\frac1{w_{\Delta_n}(z)} = \frac1{\pi\ic}\int_{\Delta_n} \frac1{x-z}\frac{dx}{w_{\Delta_n+}(x)} = \int_{\Delta_n} \frac{d\omega_{\Delta_n}(x)}{z-x},
\]
where we used the fact that \( w_{\Delta_n+}(x) = \ic |w_{\Delta_n}(x)| \) for \( x\in\Delta_n \). Since \( \mathcal K \) is a compact,  the functions \( |h_n(x)| \) are uniformly bounded. Therefore, we get that
\[
w_{\Delta_n}(\beta_n+t) \int_{\Delta_n} \frac{|h_n(x)|}{\beta_n+t-x}d\omega_{\Delta_n}(x) \lesssim_{\mathcal K} 1.
\]
Recall further that \(  v_{\Delta_n}(x) = \pi\mu^\prime(x)|w_{\Delta_n}(x)|\). We have that
\begin{align*}
w_{\Delta_n}(\beta_n+t) \int_{\Delta_n} \frac{|\log|w_{\Delta_n}(x)||}{\beta_n+t-x}d\omega_{\Delta_n}(x) & \lesssim_\Delta  \sqrt t + \sqrt t \int_{\beta_n-1}^{\beta_n} \frac{-\log(\beta_n-x)}{\beta_n+t-x} \frac{dx}{\sqrt{\beta_n-x}} \\
& \lesssim_\Delta \sqrt t + |\log t|.
\end{align*}
Moreover,
\begin{align*}
\int_{\Delta_n} \frac{|\log\mu^\prime(x)|}{\beta_n+t-x}d\omega_{\Delta_n}(x) 
\lesssim_{\Delta}
\frac{1}{\sqrt t} \int_{\alpha_n}^{\beta_n-\sqrt t} \frac{|\log\mu^\prime(x)|}{\sqrt{\beta_n-x}}dx+\frac1t \int_{\beta_n-\sqrt t}^{\beta_n} \frac{|\log\mu^\prime(x)|}{\sqrt{\beta_n-x}}dx\\
\lesssim 
\frac{\|\log\mu^\prime\|_{L^1(\omega_\Delta)}}{\sqrt t} + \frac1t \int_{\beta_n-\sqrt t}^{\beta_n} \frac{|\log\mu^\prime(x)|}{\sqrt{\beta_n-x}}dx,
 \end{align*}
where we used \eqref{UnSzego} for the last estimate. Set
\[
\epsilon(t) := \sup_{n>N_0} \int_{\beta_n-\sqrt t}^{\beta_n} \frac{|\log\mu^\prime(x)|}{\sqrt{\beta_n-x}}dx,
\]
where $N_0$ is sufficiently large to make sure that integrals converge. Clearly, this is a non-decreasing function of \( t \). We claim that \( \lim_{t\downarrow 0} \epsilon(t) =0 \). Indeed, assume to the contrary that there exist \( \epsilon_0>0 \), a sequence \( \{t_m\} \) decreasing to \( 0 \), and a set \( \{m_n\} \) such that
\begin{equation}
\label{contrapos}
\int_{\beta_{n_m}-\sqrt t_m}^{\beta_{n_m}} \frac{|\log\mu^\prime(x)|}{\sqrt{\beta_{n_m}-x}}dx \geq \epsilon_0.
\end{equation}
 If $\{m_n\}$ is bounded, it contains a constant sequence. However, \eqref{contrapos} cannot hold along this sequence as the integrals of a fixed integrable function over sets of decreasing measure must vanish. On the other hand, if \( \{n_m\} \) contains a strictly increasing sequence, \eqref{contrapos} contradicts Proposition~\ref{prop:sa3}(iii) since \( \mu\in \textrm{USz}(\Delta) \). Hence, altogether, 
\[
w_{\Delta_n}(\beta_n+t) \int_{\Delta_n}\frac{|\log\mu^\prime(x)|}{\beta_n+t-x}d\omega_{\Delta_n}(x) \lesssim_{\mu,\Delta} 1 + \frac{\epsilon(t)}{\sqrt t}.
\]
Collecting all the previous estimates gives us
\[
\log G^{-2}\big(e^{h_n}\mu_{|\Delta_n},\beta_n+t\big) \lesssim_{\Delta,\mu,\mathcal K} \frac{t+\sqrt t + \sqrt t|\log t| + \epsilon(t)}{\sqrt t}.
\]
Since the numerator above can be easily estimated by a non-decreasing function that has zero limit at zero, the claim of the lemma follows.
\end{proof}

Before stating the next lemma, we deduce the following estimate. If \( f \) is a function in the Hardy space \( H^2(\D) \), then  the Cauchy integral formula and Cauchy-Schwarz inequality  give
\[
|f(z)| \leq \int_\T \frac{|f(\eta)|}{|z-\eta|}\frac{|d\eta|}{2\pi} \leq \frac{ \|f\|_{L^2(\T)} }{\sqrt{1-|z|}}.
\]
Thus, if \( F \) is a function in the Hardy space \( H^2(\overline\C\setminus\Delta) \), then \( F\circ \phi_\Delta \) is a function in \( H^2(\D) \), see \eqref{phi-w}, and therefore
\begin{equation}
\label{HardyEst}
|F(\beta+\delta)| \leq \frac{\|F\|_{L^2(\omega_\Delta)}}{\sqrt{1-\phi_\Delta(\beta+\delta)}} \lesssim_\Delta \frac{\|F\|_{L^2(\omega_\Delta)}}{\sqrt[4]\delta}.
\end{equation}

\begin{lemma}
\label{lem:vw8}
The sequence \( \{\xi_n\} \) in the definition of \( \delta_n \) in \eqref{delta_n} can be chosen so that
\[
\int_{\Delta(\mu)\setminus\Delta_n^*} P_n^*(x)^2e^{\theta_n(x)}d\mu(x) \to 0 \qasq n\to\infty.
\]
\end{lemma}
\begin{proof}
Let  \( l_n(x)=a_nx+b_n \), \( a_n>0 \), be the linear function that maps of $\Delta$ onto $\Delta_n^*$, as in the proof of Lemma~\ref{lem:vw6}. Due to the compactness of \( \mathcal K \), it holds that the sequence \( h_n-h_n\circ l_n \) converges to zero uniformly on \( \Delta \). The uniform Szeg\H{o} condition implies 
\begin{align}
\label{limitGs}
G\big(e^{h_n}\mu_{|\Delta_n^*},z\big) & = \big(1+o_{\mathcal K}(1)\big)G\big(e^{h_n}\mu_{|\Delta},z\big), \\ G\big(e^{h_n}\mu_{|\Delta_n},z\big) & = \big(1+o_{\mathcal K}(1)\big)G\big(e^{h_n}\mu_{|\Delta},z\big), \nonumber
\end{align}
locally uniformly in \( D_\Delta \).  Hence, the last asymptotic formula of Theorem~\ref{thm:vw1} pulled back to the intervals \( \Delta_n^* \) as in Lemma~\ref{lem:vw6} shows that
\begin{equation}
\label{kvhlk}
G^{-2}(e^{h_n}\mu_{|\Delta_n},\infty)\int_{\Delta_n^*} P_n^*(x)^2e^{\theta_n(x)}d\mu(x) = 2+o_{\mathcal K}(1).
\end{equation}
Using the fact that \( \kappa_n(x)\equiv 0 \) on \( \Delta_n \), we get that
\[
|F_{n\pm}(x)|^{-2} = G^{-2}(e^{h_n}\mu_{|\Delta_n},\infty)e^{\theta_n(x)}v_{\Delta_n}(x)
\]
almost everywhere on \( \Delta_n \) by \eqref{outer-b}, \eqref{szego} and \eqref{SzegoFun}, where
\[
F_n(z) = \exp\left( n \int\log(z-x)d\omega_n(x)\right) \frac{G(e^{h_n}\mu_{|\Delta_n},\infty)}{G(e^{h_n}\mu_{|\Delta_n},z)}
\]
(here we again restricted \( \mu \) to \( \Delta_n \) and not \( \Delta_n^* \) as in Lemma~\ref{lem:vw6}). Therefore, it holds that
\[
2+o_{\mathcal K}(1) \geq \int_{\Delta_n} \big|P_n^*(x)/F_{n\pm}(x)\big|^2 d\omega_{\Delta_n}(x),
\]
where we first reduce the interval of integration from \( \Delta_n^* \) to \( \Delta_n \) in \eqref{kvhlk} and then drop the singular part of \( \mu \). Since the potentials \( V^{\omega_n} \) are continuous in \( \C \),  each function \( P_n^*/F_n \) is a product of a Szeg\H{o} function and a bounded analytic function. As such, it belongs to \( H^2(D_{\Delta_n}) \). Hence, it follows from \eqref{HardyEst} that
\[
\big|(P_n^*/F_n)(x)\big|^2 \lesssim_{\Delta,\mathcal K} 1/\sqrt{x-\beta_n}, \quad x\in[\beta_n,\beta(\mu)].
\]
The compactness of \( \mathcal K \) as well as \eqref{limitGs} imply that \( G(e^{h_n}\mu_{|\Delta_n},\infty) \lesssim_{\mu,\mathcal K} 1 \).

Now, we are ready to estimate the integrals in the statement of the lemma. We only carry out the estimates on \( [\beta_n^*,\beta(\mu)] \) as the estimate on \( [\alpha(\mu),\alpha_n^*] \) can be done analogously. Assume that \( \beta_n^*<\beta(\mu) \) (otherwise we have nothing to prove). In this case \( \beta_n^*=\beta_n+\delta_n \). Using the bounds we just obtained, Lemma~\ref{lem:vw7} and condition (1) in Theorem~\ref{thm:vw2} give
\begin{multline*}
\int_{\beta_n^*}^{\beta(\mu)} P_n^*(x)^2e^{\theta_n(x)}d\mu(x) = \int_{\beta_n^*}^{\beta(\mu)} \left|\frac{P_n^*(x)}{F_n(x)}\right|^2 \frac{G^2(e^{h_n}\mu_{|\Delta_n},\infty)}{G^2(e^{h_n}\mu_{|\Delta_n},x)} e^{2n\kappa_n(x)+h_n(x)} d\mu(x) \\ \lesssim_{\Delta,\mu,\mathcal K} \int_{\beta_n^*}^{\beta(\mu)} \exp\left( -C_\Delta n(x-\beta_n)^{\varkappa/2} + \frac{\epsilon_\mu(x-\beta_n)}{\sqrt{x-\beta_n}} \right) \frac{d\mu(x)}{\sqrt{x-\beta_n}}.
\end{multline*}
We further get that 
\begin{align}\nonumber
\int_{\beta_n^*}^{\beta(\mu)} P_n^*(x)^2e^{\theta_n(x)}d\mu(x) &\lesssim_{\Delta,\mu,\mathcal K} \int_{\delta_n}^{\beta(\mu)-\beta_n} \exp\left( C \left( -n t^{\varkappa/2} + \frac{\epsilon_\mu(t)}{\sqrt{t}} \right) \right) d\mu(t+\beta_n) \\\label{ser11}
& \lesssim_{\Delta,\mu,\mathcal K} \max_{\delta_n\leq t \leq T}\exp\left( C \left( -n t^{\varkappa/2} + \frac{\epsilon_\mu(t)}{\sqrt{t}} \right) \right),
\end{align}
where \( T := \sup_n (\beta(\mu)-\beta_n) \) and $T>0$. Let \begin{equation} T_n := (2\epsilon_\mu(T)/n )^{2/(\varkappa+1)}.\label{ser2a} \end{equation}
When 
\begin{equation}\label{ser2}
 T_n\leq t \leq T 
 \end{equation}
we get \( \epsilon_\mu(t) \leq \epsilon_\mu(T)\)
since  \( \epsilon_\mu(t) \) is non-decreasing. Also, for such $t$, we have \( \epsilon_\mu(T) \leq (n/2)t^{(\varkappa+1)/2} \) by the choice of $T_n$ and combining these bounds gives
\begin{equation}\label{ser1}
\epsilon_\mu(t)\le (n/2)t^{(\varkappa+1)/2}\,.
\end{equation}
 Respectively, we get for these \( t \) that
\[
- n t^{\varkappa/2} + \epsilon_\mu(t)/\sqrt{t} \stackrel{\eqref{ser1}}{\leq} - (n/2) t^{\varkappa/2} \stackrel{\eqref{ser2}}{\leq} - (n/2) T_n^{\varkappa/2} \stackrel{\eqref{ser2a}}{=} - \big(n\epsilon_\mu^\varkappa(T)/2\big)^{1/(\varkappa+1)},
\]
and the right-hand side converges to \( - \infty \) as $n\to\infty$. Since \( \epsilon_\mu(T_n)\to 0 \), we can choose \( \xi_n \) in \eqref{delta_n} so that 
\begin{equation}\label{ser3}
2\epsilon_\mu(T_n)\leq \xi_n.
\end{equation}
Respectively, for
\begin{equation} 
\label{ser7}
\delta_n\leq t \leq T_n 
\end{equation}
 we have that \( \epsilon_\mu(t) \leq \epsilon_\mu(T_n)\) by monotonicity and  \(\epsilon_\mu(T_n) \leq (n/2)t^{(\varkappa+1)/2} \) by $t>\delta_n$, \eqref{ser3} and formula \eqref{delta_n}. Thus, combining these bounds gives 
\begin{equation}
\epsilon_\mu(t)\le (n/2)t^{(\varkappa+1)/2}\,.\label{ser6}
\end{equation}
Hence, we get for such \( t \) that
\[
-n t^{\varkappa/2} + \epsilon_\mu(t)/\sqrt{t} \stackrel{\eqref{ser6}}{\leq} - (n/2) t^{\varkappa/2} \stackrel{\eqref{ser7}}{\leq}  - (n/2) \delta_n^{\varkappa/2} \stackrel{\eqref{delta_n}}{=} -(1/2)(n\xi_n^\varkappa)^{1/(\varkappa+1)}. 
\]
Now, we choose positive sequence $\{\xi_n\}$ so that
\begin{equation}
\lim_{n\to\infty}\xi_n=0, \quad \lim_{n\to\infty}n\xi_n^\varkappa=+ \infty, \qandq 2\epsilon_\mu(T_n)\stackrel{\eqref{ser3}}{\leq} \xi_n.
\label{ser20}
\end{equation}
This finishes the proof of the lemma because the right-hand side of \eqref{ser11} converges to zero when $n\to\infty$.
\end{proof}

\begin{lemma}
\label{lem:vw9}
Theorem~\ref{thm:vw2} takes place.
\end{lemma}
\begin{proof}
Let \( F_n^*(z) \) be as in \eqref{ser21} and
\[
\Omega_{\kappa_n}(z) := \Omega_{\Delta_n^*}\big(e^{-n\kappa_n};z),
\]
which is an outer function in \( H^2(D_{\Delta_n^*}) \) whose traces satisfy \( |\Omega_{\kappa_n\pm}(x)|^2 = e^{-2n\kappa_n(x)} \) almost everywhere on \( \Delta_n^* \), see \eqref{outer-b}. Thus, it holds that
\[
|(\Omega_{\kappa_n}F_n^*)_\pm(x)|^{-2} = G_n^{-2}e^{\theta_n(x)}v_{\Delta_n^*}(x)
\]
almost everywhere on \( \Delta_n^* \) due to \eqref{outer-b}, \eqref{szego} and \eqref{SzegoFun}, where \( G_n := G(e^{h_n}\mu_{|\Delta_n^*},\infty) \). For brevity, put \( P_n = T_n(e^{\theta_n}\mu) \). Then, it follows from the Cauchy integral formula that
\begin{align*}
\frac{P_n(z)-P_n^*(z)}{ (\Omega_{\kappa_n}F_n^*)(z)w_{\Delta_n^*}(z)} &= \frac1{2\pi\ic}\oint_{\Gamma_n} \frac{P_n(s)-P_n^*(s)}{z-s} \frac{1}{(\Omega_{\kappa_n}F_n^*)(s)} \frac{ds}{w_{\Delta_n^*}(s)} \\
& = \int_{\Delta_n^*} \frac{P_n(x)-P_n^*(x)}{z-x} \re\left( \frac1{(\Omega_{\kappa_n}F_n^*)_+(x)} \right) d\omega_{\Delta_n^*}(x),
\end{align*}
where \( \Gamma_n \) is any counter-clockwise oriented Jordan curve that separates \( \Delta_n^* \) and \( z \) and we used the following facts
\[
F_{n-}^*(x) = \overline{F_{n+}^*(x)} \qandq w_{\Delta_n^*+}(x) = - w_{\Delta_n^*-}(x) = \ic|w_{\Delta_n^*}(x)|, \quad x\in\Delta_n^*,
\]
as well as \eqref{arcsine}. Since the absolute value of the real part does not exceed the absolute value of a complex number, we get from the Cauchy-Schwarz inequality that
\begin{align}\nonumber
\left|\frac{P_n(z)-P_n^*(z)}{ (\Omega_{\kappa_n}F_n^*)(z)}\right|^2 & \leq G_n^{-2}\frac{|w_{\Delta_n^*}(z)|^2}{\dist(z,\Delta_n^*)^2} \int_{\Delta_n^*} \big(P_n(x)-P_n^*(x) \big)^2e^{\theta_n(x)} v_{\Delta_n^*}(x)d\omega_{\Delta_n^*}(x) \\\label{ser22}
& \leq G_n^{-2}\frac{|w_{\Delta_n^*}(z)|^2}{\dist(z,\Delta_n^*)^2}  \int_{\Delta_n^*} \big(P_n(x)-P_n^*(x)\big)^2 e^{\theta_n(x)}d\mu(x).
\end{align}
It follows from the compactness of \( \mathcal K \) and \eqref{limitGs} that the constants \( G_n^{-2} \) are uniformly bounded above. Because \( P_n^* \) is the \( n \)-th monic orthogonal polynomials with respect to \( e^{\theta_n}\mu_{|\Delta_n^*} \) and \( P_n \) is a monic polynomials of degree \( n \), we have that
\[
\int_{\Delta_n^*} P_n(x)P_n^*(x)e^{\theta_n(x)}d\mu(x) = \int_{\Delta_n^*} P_n^*(x)^2 e^{\theta_n(x)}d\mu(x),
\]
from which we easily deduce that
\begin{equation}
\label{ser23a}
\int_{\Delta_n^*} \big(P_n(x)-P_n^*(x)\big)^2 e^{\theta_n(x)}d\mu(x) =\int_{\Delta_n^*} P_n(x)^2e^{\theta_n(x)}d\mu(x) - \int_{\Delta_n^*} P_n^*(x)^2e^{\theta_n(x)}d\mu(x).
\end{equation}
Recall that for the \( n\)-th monic orthogonal polynomial $P_n$ it holds that
\begin{equation}
\label{ser23}
\int_{\Delta(\mu)} P_n^2(x) e^{\theta_n(x)}d\mu(x) \leq \int_{\Delta(\mu)} P^2(x)e^{\theta_n(x)}d\mu(x)
\end{equation}
for any monic polynomial \( P \) of degree \( n \), and $P_n^*$ in particular. Therefore, we obtain from \eqref{ser23a}, simple majorization, and \eqref{ser23} that
\begin{align*}
\int_{\Delta_n^*} \big(P_n(x)-P_n^*(x)\big)^2 e^{\theta_n(x)}d\mu(x) & \leq \int_{\Delta(\mu)} P_n(x)^2e^{\theta_n(x)}d\mu(x) - \int_{\Delta_n^*} P_n^*(x)^2e^{\theta_n(x)}d\mu(x) \\
& \leq  \int_{\Delta(\mu)} P_n^*(x)^2e^{\theta_n(x)}d\mu(x) - \int_{\Delta_n^*} P_n^*(x)^2e^{\theta_n(x)}d\mu(x) \\
& = \int_{\Delta(\mu)\setminus\Delta_n^*} P_n^*(x)^2e^{\theta_n(x)}d\mu(x) \to  0
\end{align*}
as \( n\to\infty \), where the last conclusion was shown in Lemma~\ref{lem:vw8}. Since the intervals \( \Delta_n^* \) converge to \( \Delta \), we get from \eqref{ser22}, the above estimates, and Lemma~\ref{lem:vw6} that
\[
P_n(z) = P_n^*(z) + o_{\mathcal K}(1) (\Omega_{\kappa_n} F_n^*)(z) = \big( 1+ o_{\mathcal K}(1) \big) F_n^*(z)
\]
locally uniformly in \( D_\Delta \), where the last equality also used the fact that \( \Omega_{\kappa_n}(z) = 1+o(1) \) locally uniformly in the complement of \( \Delta \) as is clear from \eqref{outer}, \eqref{delta_n}, and \eqref{L1kappa}. The first claim of Theorem~\ref{thm:vw2} now follows from \eqref{limitGs}. The second claim of the theorem is a consequence of \eqref{limitGs} for \( \Delta_n \) and of \eqref{kvhlk} if we observe that
\begin{multline*}
\int_{\Delta_n^*}P_n^*(x)^2e^{\theta_n(x)}d\mu(x) \leq \int_{\Delta_n^*} P_n(x)^2e^{\theta_n(x)}d\mu(x) \leq \int_{\Delta(\mu)} P_n(x)^2e^{\theta_n(x)}d\mu(x) \\
 \leq \int_{\Delta(\mu)}P_n^*(x)^2e^{\theta_n(x)}d\mu(x) = \int_{\Delta_n^*}P_n^*(x)^2e^{\theta_n(x)}d\mu(x) + o(1),
\end{multline*}
where we used  Lemma~\ref{lem:vw8}  for the last equality. As we just mentioned, the last claim of the theorem follows from the second line of \eqref{limitGs}.
\end{proof}

\section{Strong Asymptotics of OPs with Reciprocal Polynomial Weights}
\label{sec:pw}

In this section, we describe strong asymptotics of polynomials orthogonal with respect to a sequence of measures  that are ratios of certain bounded perturbations of a Szeg\H{o} measure and polynomials of growing degrees. Below, we mostly follow \cite{St00} and generalize some results obtained in the pioneering work \cite{CYLL98}. Theorem~\ref{thm:pw8} further below was used in a crucial way in Lemma~\ref{lem:vw4} on the way to proving Theorem~\ref{thm:vw1}. Theorem~\ref{thm:pw8} itself is a straightforward consequence of Theorem~\ref{thm:pw1} that we state in the next subsection.

\subsection{Orthogonality on \( \T \)}

Let \( \{ (\sigma_n,g_n,W_n) \} \) be triples of a finite positive Borel measure on \( \T \), a continuous real-valued function on \( \T \), and a monic polynomial of degree \( n \) with all its zeros inside the unit disk. For each \( n \) we then define an inner product on the unit circle by
\[
\langle f,k\rangle_{\T,n} = \int f(\xi)\overline{k(\xi)} \, \frac{e^{g_n(\xi)}d\sigma_n(\xi)}{|W_n(\xi)|^2}.
\]
We are interested in orthonormal polynomials \( \phi_n(z) \), \( \deg\phi_n=n \), satisfying
\begin{equation}
\label{ortho-circle}
\begin{cases}
\langle \phi_n,\xi^m\rangle_{\T,n}=0, &  m\in\{0,1,\ldots,n-1\}, \smallskip \\
\langle \phi_n,\phi_n\rangle_{\T,n} =1,
\end{cases}
\end{equation}
and normalized to have a positive leading coefficient, i.e., \( \phi_n(z) = \alpha_n z^n + \ldots \), \( \alpha_n>0 \) (when \( g_n=0 \) and \( \sigma_n=\sigma \), this is exactly the setting considered in \cite{St00}). 

In this section, we let the symbol \( * \) denote the transformation
\[
p^{\ast}(z)=z^n\overline{p(1/{\bar{z}})}
\]
defined on the set of polynomials of degree at most \( n \). In particular, if we denote the zeros of \( W_n(z) \) by \( b_{n,j} \), \( j\in\{1,2,\ldots,n\} \), which all belong to the unit disk, then
\[
W_n(z)=\prod_{j=1}^n (z-b_{n,j}) \qandq W_n^*(z)=\prod_{j=1}^n (1-\overline{b_{n,j}}z).
\]
Write $d\sigma_n=\upsilon_ndm+d\sigma_n^s$, where the measures \( d\sigma_n^s \) are singular to $dm(\xi)=(2\pi)^{-1}|d\xi|$, the normalized Lebesgue measure on \( \T \). In this section, we shall assume the following:
\emph{
\begin{itemize}
\item[$(A_\T)$] there exists a finite measure $d\sigma=\upsilon dm+d\sigma^s$, where \( d\sigma^s \) is singular to $dm$, such that for any non-negative continuous function \( f \) on \( \T \) it holds that
\[
\limsup_{n\to\infty} \int fd\sigma_n \leq \int fd\sigma;
\]
\item[$(B_\T)$] the densities \( \upsilon_n \) and \( \upsilon \) have  logarithms integrable with respect to \( dm \) and  \( \|\log \upsilon_n- \log \upsilon\|_{L^1(\T)}\to 0 \) as \( n\to\infty \); \smallskip
\item[$(C_\T)$] the functions \( g_n \) belong to \( \mathcal E \), a fixed compact subset of \( C(\T) \); \smallskip
\item[$(D_\T)$] the zeros \( \{b_{n,j}\} \) satisfy \( \sum_{j=1}^n (1-|b_{n,j}|)\to \infty \) as \( n\to\infty \).
\end{itemize}
}
It is known, see \cite[Section~II.2]{Ga}, that the condition $(D_\T)$ is equivalent to
\begin{equation}
\label{Blaschke}
\frac{W_n(z)}{W_n^*(z)} = o(1)
\end{equation}
locally uniformly in \( \D \).
To describe the results, let us introduce the Szeg\H{o} function of a measure \( d\sigma=\upsilon dm+d\sigma^s \), which is given by
\[
D(\sigma,z):=\exp\left(\frac 12\int\frac{\xi+z}{\xi-z}\log\upsilon(\xi)dm(\xi)\right)
\]
and is independent of the singular part \( \sigma^s \). When \( \log\upsilon \) is integrable with respect to \( dm \), the function \( D(\sigma,z) \) is analytic in \( \overline\C\setminus\T \), in fact, it is an outer function in both \( \D \) and \( \overline\C\setminus\overline \D \), its values inside and outside of the unit disk are related via the identity
\[
D^{-1}(\sigma,z) = \overline{D(\sigma,1/\bar z)}, \quad z\not\in \T,
\]
and it has a non-tangential limit on \( \T \) (taken within $\D$) that satisfies  $|D(\sigma,\xi)|^2=\upsilon(\xi)$ for almost every $\xi$ on $\T$.

\begin{theorem}
\label{thm:pw1}
With the above definitions, assume that the orthogonality measures in \eqref{ortho-circle} satisfy conditions \( (A_\T)-(D_\T) \). Then,
\begin{equation}
\label{A1-1}
 \frac{\phi_n^*(z)}{W_n^*(z)}D_n(z)  = 1 + o_\mathcal{E}(1) \qandq \frac{\phi_n(z)}{\phi_n^*(z)} = o_\mathcal{E}(1)
 \end{equation}
locally uniformly in \( \D \) as \( n\to\infty \), where \( D_n(z) = D(e^{g_n}\sigma_n,z) \). In particular, we have that
\begin{equation}
\label{A1-2}
\alpha_n D_n(0) = 1 + o_\mathcal{E}(1) \qasq n\to\infty.
\end{equation}
\end{theorem}

\subsection{Proof of Theorem~\ref{thm:pw1}}

As in the previous sections, we shall prove Theorem~\ref{thm:pw1} in a sequence of steps that we organize as separate lemmas. For the first step we define the Caratheodory function of a measure \( \sigma \) by
\[
F(\sigma,z):=\int \frac{\xi+z}{\xi-z}d\sigma(\xi).
\]
Further, for each orthonormal polynomial \( \phi_n(z) \) we define its companion polynomial  as a polynomial \( \psi_n(z) \) of degree at most \( n \) such that \( \psi_n^*(z) \) interpolates  $\phi_n^*(z) F(e^{g_n}\sigma_n,z)$ at the zeros of $zW_n(z)$, see \cite[Equation~(3.6)]{St00}.


\begin{lemma}
\label{lem:pw2}
Under the conditions of Theorem~\ref{thm:pw1}, we have
\[
F(e^{g_n}\sigma_n,z)- \frac{\psi^*_n(z)}{\phi^*_n(z)} = o_\mathcal{E}(1) \qasq n\to\infty
\]
locally uniformly in \( \D \). Moreover, 
\begin{equation}\label{zar1}
F(e^{g_n}\sigma_n,0)=\frac{\psi^*_n(0)}{\phi^*_n(0)}\,.
\end{equation}
\end{lemma}
\begin{proof}
Combining formulas (3.23), (3.24), and (3.26) from \cite{St00}, we get
\[
\left|F(e^{g_n}\sigma_n,z)- \frac{\psi^*_n(z)}{\phi^*_n(z)} \right| \leq 2\sqrt 2\left|z\frac{W_n(z)}{W_n^*(z)}\right| \frac{|e^{g_n}\sigma_n|}{(1-|z|)^{3/2}}
\]
for \( z\in \D \), where \( |\mu| \) denotes the total mass of the measure \( \mu \). Taking $f=1$ in $(A_\T)$, we get
\begin{equation}
\label{zar2}
\limsup_{n\to \infty}|\sigma_n|\le|\sigma| \qandq \limsup_{n\to \infty}|e^{g_n}\sigma_n|\lesssim_{\mathcal E} |\sigma|,
\end{equation}
where the last bound follows from $(C_\T)$. Applying \eqref{Blaschke} finishes the proof of the lemma.
\end{proof}

Notice that Jensen's inequality and $(C_\T)$ yield that
\begin{equation}
\label{zar3}
\liminf_{n\to\infty} |\sigma_n|\ge \liminf_{n\to\infty} \exp\left(\int \log\upsilon_n dm\right)=\exp\left(\int \log\upsilon dm\right)>-\infty
\end{equation}
and $(B_\T$) gives a lower bound
\begin{equation}
\label{zar4}
\liminf_{n\to\infty} |e^{g_n}\sigma_n|\gtrsim_{\mathcal{E}}\exp\left(\int \log\upsilon dm\right).
\end{equation}

Put
\begin{equation}
\label{lam1}
\lambda_n(z) = \re\left( \frac{\psi^*_n(z)}{\phi^*_n(z)} \right), \quad |z|\leq1,
\end{equation}
which is a harmonic function in some neighborhood of \( \overline \D \). It has been shown in \cite[Equation~(3.9)]{St00} that
\begin{equation}
\label{lam2}
\frac{\psi^*_n(z)}{\phi^*_n(z)} = \int \frac{\xi+z}{\xi-z} \left| \frac{W_n(\xi)}{\phi_n(\xi)} \right|^2 dm(\xi).
\end{equation}
Hence, \( \lambda_n(z) \) is a real part of the Caratheodory function of an absolutely continuous measure with strictly positive density and therefore is a strictly positive harmonic function in \( \overline\D \).  It readily follows from the second claim of Lemma~\ref{lem:pw2} that
\begin{equation}
\label{ls-masses}
|\lambda_nm| = \lambda_n(0) = F(e^{g_n}\sigma_n,0) = |e^{g_n}\sigma_n|.
\end{equation}

\begin{lemma}
\label{lem:pw3} 
Let $\mathcal Z$ be a compact set in $C(\T)$. Under the conditions of Theorem~\ref{thm:pw1}, it holds that for any \( \epsilon>0 \) there exists \( N_{\mathcal E,\mathcal Z}(\epsilon) \) such that
\[
\left|\int he^{g_n}d\sigma_n-\int h\lambda_ndm\right| \leq \epsilon
\]
for all \( n \geq N_{\mathcal E,\mathcal Z}(\epsilon) \) and each \( h\in\mathcal Z \).
\end{lemma}
\begin{proof}
We use a standard approximation argument. Comparing the Taylor coefficients  of functions $F(e^{g_n}\sigma_n,z)$ and $\psi^*_n(z)/\phi^*_n(z)$ at $z=0$ and applying Lemma~\ref{lem:pw2} gives
\[
\int h(\xi) \big(e^{g_n}d\sigma_n-\lambda_ndm\big)  = o_{\mathcal E,h}(1)
\]
as \( n\to\infty \), where \( h(\xi) = \xi^j \), $j\in \Z$. Hence, this claim remains true for each trigonometric polynomial $h$. Now, given $\epsilon>0$, we can use compactness and denseness of trigonometric polynomials in $C(\T)$ to find a finite collection of trigonometric polynomials \( \{h_1,h_2,\ldots,h_{K(\epsilon)}\}\) such that for each $h\in \mathcal{Z}$ there is $k\in\{1,2,\ldots,K(\epsilon)\}$ for which\
\[
\|h-h_k\|_\infty<\frac{\epsilon}{3M_{\mathcal E}},
\]
where \( \|e^{g_n}\sigma_n\| \leq M_{\mathcal E} \). Then, 
\[
 \left|\int (h-h_k)e^{g_n}d\sigma_n\right| \leq \frac\epsilon3 \qandq \left|\int (h-h_k)\lambda_ndm\right|\leq \frac\epsilon3
\]
by \eqref{ls-masses}. As just observed, for each \( h_k \) one can find a natural number $N_{\mathcal E,h_k}(\epsilon)$ such that 
\[
\left| \int  h_k\big(e^{g_n}d\sigma_n-\lambda_ndm\big) \right|\le \frac\epsilon3
\]
for $n\ge N_{\mathcal E,h_k}(\epsilon)$. Taking $N_{\mathcal E,\mathcal Z}(\epsilon)=\max_{1\leq k\leq K(\epsilon)} N_{\mathcal E,h_k}(\epsilon)$ yields the desired claim.
\end{proof}

For the next step, we shall need the mutual entropy of two measures. Let $\mu$ and $\nu$ be two measures on $\T$ such that $\mu$ is absolutely continuous with respect to $\nu$. The entropy $S(\mu|\nu)$ is defined as 
\[
S(\mu|\nu)=-\int \log\left(\frac{d\mu}{d\nu}\right)d\mu.
\]
It is known that $S(\mu|\nu)\le \log |\nu|$  and the following representation holds, see \cite[Lemma~2.3.3,  p.137]{Simon}:
\begin{equation}
\label{aper1}
S(\mu|\nu) = \inf_f\left(  \int fd\nu-\int (1+\log f)d\mu \right),
\end{equation}
where the infimum is taken over all  positive continuous functions on \( \T \). Moreover, see \cite[Example~2.3.2]{Simon}, if \( \mu=m \) and \( d\nu= \nu^\prime dm + d\nu^s \), then
\begin{equation}
\label{aper2}
S(m|\nu) = - \int\log\left(\frac{dm}{d\nu}\right)dm = \int \log\nu^\prime dm.
\end{equation}

\begin{lemma}
\label{lem:pw4}
Under the conditions of Theorem~\ref{thm:pw1}, it holds that for any \( \epsilon>0 \) there exists \( N_{\mathcal E}(\epsilon) \),
that depends on \( \mathcal E \) but not a particular choice of \( \{g_n \}\subset \mathcal E \), such that 
\[
\int \log\lambda_n dm \le \int \big( g_n+\log\upsilon_n\big) dm + \epsilon, \quad n\geq N_{\mathcal E}(\epsilon).
\]
\end{lemma}
\begin{proof}
It follows from \eqref{aper1} and \eqref{aper2}, applied with \( \mu=m \) and \( \nu=\sigma \), as well as the condition \( (B_\T) \) that there exists a positive continuous function \( f_\epsilon \) and a natural number \( N_1(\epsilon) \) such that
\[
\int f_\epsilon d\sigma - \int (1+\log f_\epsilon)dm \leq \int\log\upsilon dm +  \frac\epsilon4 \leq \int\log\upsilon_n dm +  \frac\epsilon2
\]
for all \( n\geq N_1(\epsilon) \). Hence, we get again from \eqref{aper1} and \eqref{aper2}, applied this time with \( \mu=m \) and \( d\nu=\lambda_ndm \), that
\begin{align}
\int \log\lambda_n dm &\leq \int (e^{-g_n}f_\epsilon) \lambda_n dm - \int \big(1+\log(e^{-g_n}f_\epsilon) \big) dm \nonumber \\
&\leq \int\big( g_n+\log\upsilon_n\big) dm + \int (e^{-g_n}f_\epsilon)\lambda_n dm - \int f_\epsilon d\sigma +  \frac\epsilon2
\label{A4-1}
\end{align}
for all \( n\geq N_1(\epsilon) \). Furthermore, it follows from the condition \( (A_\T) \) that there exists a natural number \( N_2(\epsilon) \) such that
\[
\int f_\epsilon d\sigma_n \leq \int f_\epsilon d\sigma +  \frac\epsilon4
\]
for all \( n\geq N_2(\epsilon) \). Thus, we get from Lemma~\ref{lem:pw3}, applied with \( \mathcal Z=\{e^{-g}f_\epsilon:g\in\mathcal E\} \), that there exists \( N_{\mathcal E}(\epsilon) \geq \max\{ N_1(\epsilon),N_2(\epsilon) \} \)  such that
\begin{equation}
\label{A4-2}
\int (e^{-g_n}f_\epsilon)\lambda_n dm \leq \int (e^{-g_n}f_\epsilon) e^{g_n}d\sigma_n +  \frac\epsilon4 \leq \int f_\epsilon d\sigma +  \frac\epsilon2
\end{equation}
for all \( n\geq N_{\mathcal E}(\epsilon) \). Clearly, inequalities \eqref{A4-1} and \eqref{A4-2} yield the desired claim.
\end{proof}

\begin{lemma}
\label{lem:pw5}
Under the conditions of Theorem~\ref{thm:pw1}, \eqref{A1-2} takes place.
\end{lemma}
\begin{proof}
 It follows directly from \eqref{lam1} and \eqref{lam2}, see also \cite[Equation~(3.8)]{St00}, that
\[
\left|\frac{W_n^*(\xi)}{\phi_n^*(\xi)}\right|^2 = \left|\frac{W_n(\xi)}{\phi_n(\xi)}\right|^2 = \lambda_n(\xi), \quad |\xi|=1.
\]
As \( W_n^*(z)/\phi_n^*(z) \) is an analytic and non-vanishing function in the closed unit disk, the logarithm of its absolute value is harmonic there. Thus, we get from the mean-value property that
\[
\log\alpha_n = -\log\left|\frac{W_n^*(0)}{\phi_n^*(0)}\right| = - \frac12\int\log\lambda_ndm,
\]
where one also needs to recall that \( \phi_n^*(0) = \alpha_n>0 \) and \( W_n^*(0)=1 \).  Hence, it follows from Lemma~\ref{lem:pw4} and the very definition of \( D_n(z) \) that
\[
\log\alpha_n \geq -\frac12\int \big( g_n+\log\upsilon_n\big) dm + o_{\mathcal E}(1) = - \log D_n(0) + o_{\mathcal E}(1). 
\]
On the other hand, we get from \eqref{ortho-circle} that
\[
0 = \log \int \left|\frac{\phi_n}{W_n}\right|^2 e^{g_n}d\sigma_n \geq \log \int \left|\frac{\phi_n}{W_n}\right|^2 e^{g_n} \upsilon_ndm.
\]
Therefore,  Jensen's inequality gives
\[
0 \geq \int \log \frac{e^{g_n}\upsilon_n}{\lambda_n} dm = 2\log\big(\alpha_nD_n(0) \big),
\]
which finishes the proof of the lemma.
\end{proof}

\begin{lemma}
\label{lem:pw6}
Under the conditions of Theorem~\ref{thm:pw1}, it holds that
\[
\int\left|\frac{\phi_n^*}{W_n^*}D_n-1\right|^2dm = o_{\mathcal E}(1)
\]
as \( n\to\infty \), where \( D_n(\xi) \) denotes above the non-tangential boundary values of \( D_n(z) \) on \( \T \) taken from within \( \D \). In particular, the first asymptotic formula in \eqref{A1-1} takes place.
\end{lemma}
\begin{proof}
Denote the integral in the statement of the lemma by \( I \). Then, we have that
\[
I = \int \left|\frac{\phi_n^*}{W_n^*}D_n\right|^2dm + 1 -2\int\re \left(\frac{\phi_n^*}{W_n^*}D_n\right)dm.
\]
Since \( |D_n(\xi)|^2=e^{g_n(\xi)}\upsilon_n(\xi) \) for almost every \( |\xi|=1 \), the mean-value property for harmonic functions yields that
\[
I = \int \left|\frac{\phi_n}{W_n}\right|^2 e^{g_n}\upsilon_ndm + 1 - 2\re \left(\frac{\phi_n^*(0)}{W_n^*(0)}D_n(0)\right) \leq 2 -2\alpha_nD_n(0).
\]
The first claim of the lemma now follows from \eqref{A1-2}. In particular, we have shown that the functions \( 1-\phi_n^*D_n/W_n^* \) belong to the Hardy space \( H^2(\D) \). Thus, the second claim of the lemma now follows from the first and the Cauchy integral formula for functions in \( H^2(\D) \).
\end{proof}

\begin{lemma}
\label{lem:pw7}
Under the conditions of Theorem~\ref{thm:pw1}, the second asymptotic formula in \eqref{A1-1} takes place.
\end{lemma}
\begin{proof}
It follows from the first asymptotic formula in \eqref{A1-1} that
\[
\frac{\phi_n(z)}{\phi_n^*(z)} = \frac{W_n^*(z)}{\phi_n^*(z)}  \frac{\phi_n(z)}{W_n^*(z)} = (1 +o_{\mathcal E}) \frac{\phi_n(z)D_n(z)}{W_n^*(z)}.
\]
Thus, it is sufficient for us to study the behavior of \( \phi_nD_n/W_n^* \) in the unit disk. Since these functions have integrable traces on \( \T \) and \( dm(\xi) = d\xi/(2\pi\ic\xi) \), it follows from the Cauchy integral formula that
\[
\frac{\phi_n(z)D_n(z)}{W_n^*(z)} = \int \frac{\phi_n(\xi)D_n(\xi)}{W_n^*(\xi)} \frac{dm(\xi)}{1-z\overline\xi} = \int B_n(\xi)U_n(\xi) \frac{\phi_n(\xi)\overline{D_n(\xi)}}{W_n(\xi)} \frac{dm(\xi)}{1-z\overline\xi},
\]
where \( B_n(z) = W_n(z)/W_n^*(z) \) and \( U_n(\xi) = D_n(\xi)/\overline{D_n(\xi)} \). Since  \( |B_n(\xi)U_n(\xi)| = 1 \) for \( |\xi|=1 \), it follows from the Cauchy-Schwarz inequality that
\[
\left| \int  \left(\frac{\phi_n(\xi)\overline{D_n(\xi)}}{W_n(\xi)} - 1 \right) B_n(\xi)U_n(\xi)\frac{dm(\xi)}{1-z\overline\xi}\right| \leq \frac{\|\phi_n\overline{D_n}/W_n -1 \|_{L^2(\T)}}{\sqrt{1-|z|^2}}.
\]
Thus, we deduce from the first claim of Lemma~\ref{lem:pw6} that
\[
\frac{\phi_n(z)D_n(z)}{W_n^*(z)} = o_{\mathcal E}(1) + \int B_n(\xi)U_n(\xi) \frac{dm(\xi)}{1-z\overline\xi},
\]
where \( o_{\mathcal E}(1) \) holds locally uniformly in the unit disk.\footnote{When \( U_n(\xi)=U(\xi) \), the last integral above can be understood as \( (\mathcal T_UB_n)(z) \), where \( \mathcal T_U \) is a Toeplitz operator with symbol \( U \). Since \( \mathcal T_U^* = \mathcal T_{\overline U} \) and the Blaschke products \( B_n(z) \) converge weakly to zero in \( H^2(\D) \) by \eqref{Blaschke}, the functions \( (\mathcal T_UB_n)(z) \) also converge weakly to zero in \( H^2(\D) \), which finishes the proof of the lemma in this case. This observation somewhat simplifies the arguments given in \cite[pages 182-188]{St00}.} Let \( \{b_{n,k} \}\) be the Fourier coefficients of \( B_n(\xi) \) and \( \{u_{n,k}(z) \}\) be the Fourier coefficients of \( U_n(\xi)/(1-z\overline\xi) \). Then,
\[
\int B_n(\xi)U_n(\xi) \frac{dm(\xi)}{1-z\overline\xi} = \sum_{k=0}^\infty b_{n,k}u_{n,-k}(z).
\]
Since both \( B_n(\xi) \) and \( U_n(\xi) \) are unimodular functions, 
\[
\sum_{k\in \Z}|b_{n,k}|^2=1, \quad \sum_{k\in \Z}|u_{n,k}|^2=\int \frac{dm(\xi)}{|1-z\bar{\xi}|^2}=\frac{1}{1-|z|^2}\,.
\]
The Cauchy-Schwarz inequality gives
\begin{eqnarray*}
\left| \int B_n(\xi)U_n(\xi) \frac{dm(\xi)}{1-z\overline\xi} \right| & \displaystyle \leq 
\left|\sum_{k=0}^{K} b_{n,k}u_{n,-k}(z)\right|+\left|\sum_{k=K+1}^{\infty} b_{n,k}u_{n,-k}(z)\right|
\\ & \displaystyle \leq \sqrt{\sum_{k=0}^K \frac{|b_{n,k}|^2}{1-|z|^2}} + \sqrt{\sum_{k=K+1}^\infty |u_{n,-k}(z)|^2}
\end{eqnarray*}
for any natural number \( K \). Asymptotic formula \eqref{Blaschke} yields that for any \( \epsilon>0 \) and any \( K \), there exists \( N_K(\epsilon) \) such that the first sum above is bounded by \( \epsilon/(1-|z|^2)^{1/2} \) for all \( n\geq N_K(\epsilon) \). Thus, to prove the lemma we are only left to argue that there exists \( K=K_{\mathcal E,z}(\epsilon) \), which is locally uniformly bounded with respect to \( z \), such that the second sum is bounded by \( \epsilon \). To prove this claim, it is enough to show that the functions (as functions in variable $\xi$)
\[
\frac1{1-z\overline\xi}\frac{D(e^g,\xi)}{\overline{D(e^g,\xi)}} \frac{D(\upsilon_n,\xi)}{\overline{D(\upsilon_n,\xi)}}
\]
form a precompact set in \( L^2(\T) \) for \( g\in\mathcal E \), \( n\in\N \), and \( z \) in any fixed compact subset of \( \D \). Let \( \mathcal H \) be the Hilbert transform. The functions above can be written as
\[
\big(1-z\overline\xi\big)^{-1} e^{2\ic(\mathcal H g)(\xi)} e^{2\ic(\mathcal H \log\upsilon_n)(\xi)}.
\]
Consider the following subsets of \( L^2(\T) \):
\[
\begin{cases}
S_1 & = \big\{\big(1-z\overline\xi\big)^{-1}:~ z\in \text{compact in }\D\big\}, \smallskip \\
S_2 & =\left\{ e^{2\ic(\mathcal H g)(\xi)}: g\in \mathcal{E}\right\}, \smallskip \\
S_3 & =\left\{e^{2\ic(\mathcal H \log\upsilon_n)(\xi)}: n\in \N\right\}.
\end{cases}
\]
The compactness of $S_1$ in $L^2(\T)$ is trivial. Since \( \mathcal E \) is compact in \( C(\T) \) it is also compact in \( L^2(\T) \). The operator \( \mathcal H \) is a bounded on \( L^2(\T) \). Hence, \( \mathcal H\mathcal E \) is a compact subset of real-valued functions in \( L^2(\T) \). Given two real numbers $f$ and $g$, we have
\begin{equation}
\label{2-2}
|e^{\ic f}-e^{\ic h}|^2=4\sin^2((f-h)/2) \leq 4|f-h|^p, \quad p\in(0,2].
\end{equation}
Therefore, $S_2$ is compact in \( L^2(\T) \). Finally, since \( \log\upsilon_n \to \log\upsilon \) in \( L^1 (\T) \) due to our condition \( (B_\T) \), it holds that \( \mathcal H\log\upsilon_n \to \mathcal H\log\upsilon \) in \( L^p (\T) \) for any \( p\in (0,1) \) by Kolmogorov's theorem. Using \eqref{2-2} with any such \( p \) yields that the functions \( \exp(2\ic\mathcal H \log\upsilon_n) \) converge to \( \exp(2\ic\mathcal H \log\upsilon ) \) in \( L^2(\T) \), which shows that $S_3$ is precompact in $L^2(\T)$. Since $S_1,S_2$, and $S_3$ are bounded in $L^\infty(\T)$, the product set $S_1S_2S_3$ is  precompact in $L^2(\T)$ as claimed. That finishes the proof of the lemma.
\end{proof}

\subsection{Orthogonality on \( [-1,1] \)}

We will now translate the results of Theorem~\ref{thm:pw1} to the case of polynomials orthonormal on the interval \( [-1,1] \). Here, we look at triples \( (\tilde\mu_n,h_n,\tau_n) \), where \( \tilde\mu_n \) is a finite positive Borel measure, \( h_n \) is a continuous function, and \( \tau_n \) is a polynomial of degree at most $2n$ with real coefficients that does not vanish on \( [-1,1] \) and is normalized to have value \( 1 \) at \( 0 \).  For each \( n \) we define an inner product on \( [-1,1] \) by
\[
\langle f,g\rangle_{[-1,1],n} = \int f(x)\overline{g(x)} \, \frac{e^{h_n(x)}d\tilde\mu_n(x)}{\tau_n(x)}.
\]
We are interested in the orthonormal polynomials with the above varying weights. That is, we study polynomials \( p_n(z) \), \( \deg p_n=n \), satisfying orthogonality relations
\begin{equation}
\label{ortho-interval}
\begin{cases}
\langle p_n,x^m\rangle_{[-1,1],n}=0, &  m\in\{0,1,\ldots,n-1\}, \smallskip \\
\langle p_n,p_n\rangle_{[-1,1],n} =1,
\end{cases}
\end{equation}
which are normalized to have positive leading coefficient that we denote by \( \gamma_n \), i.e., \( p_n(z) = \gamma_n z^n + \text{lower degree terms} \), \( \gamma_n>0 \).

The assumptions we made about the polynomials \( \tau_n \) can be equivalently stated in the following way. Let \( \{a_{n,1},a_{n,2},\ldots,a_{n,2n} \} \) be a conjugate-symmetric multi-set (points \( a_{n,j} \) can coincide and be either real or come in complex-conjugate pairs) such that \( a_{n,j}\notin[-1,1] \) (these points can be equal to \( \infty\)). Then,
\[
\tau_n(x) = \prod_{j=1}^{2n}\left(1-\frac{x}{a_{n,j}}\right),
\]
where we understand \( x/a_{n,j} \) as  \( 0 \) when \( a_{n,j} = \infty \). Below, we assume the following:
\emph{
\begin{itemize}
\item[$(A_{[-1,1]})$] there exists a finite measure $\mu$ on $[-1,1]$ such that for any non-negative continuous function \( f \) on \( [-1,1] \) it holds that
\[
\limsup_{n\to\infty} \int fd\tilde\mu_n \leq \int fd\mu;
\]
\item[$(B_{[-1,1]})$] the Radon-Nikodym derivatives of \( \tilde\mu_n \) and \( \mu \) with respect to \( \omega=\omega_{[-1,1]} \), say \( \tilde v_n \) and \( v \), see \eqref{arcsine} and \eqref{szego}, satisfy \( \|\log \tilde v_n- \log v\|_{L^1(\omega)}\to 0 \) as \( n\to\infty \); \smallskip
\item[$(C_{[-1,1]})$] the functions \( h_n \) belong to \( \mathcal K \), a fixed compact subset of \( C[-1,1] \); \smallskip
\item[$(D_{[-1,1]})$] the zeros \( \{a_{n,j}\} \) of the polynomials \( \tau_n \) satisfy \( \sum_{j=1}^{2n} (1-|\phi(a_{n,j})|)\to \infty \) as \( n\to\infty \), where \( \phi(z)=\phi_{[-1,1]}(z) \), see \eqref{phi-w}.
\end{itemize}
}

The next result is proven exactly as \cite[Theorem~1]{St00}. Nevertheless, we provide most of the details for the reader's convenience.

\begin{theorem} 
\label{thm:pw8} 
With the above definitions, assume that the triples $(\tilde\mu_n,h_n,\tau_n)$ satisfy conditions \( (A_{[-1,1]})-(D_{[-1,1]}) \). Let \( p_n(z) \) be as in \eqref{ortho-interval}. Then,
\begin{equation}
\label{A8-1}
2\widetilde G_n^2(z) \, \frac{p_n^2(z)}{\tau_n(z)} \, \prod_{j=1}^{2n} \frac{\phi(z)-\phi(a_{n,j})}{1-\overline{\phi(a_{n,j})}\phi(z)} = 1 + o_{\mathcal K}(1)
\end{equation}
locally uniformly in $\overline{\C}\setminus [-1,1]$ as \( n\to\infty \), where \( \widetilde G_n(z)=G(e^{h_n}\tilde\mu_n,z) \) and the meaning of \( o_{\mathcal K}(1) \) is the same as in Theorem~\ref{thm:vw1}. In particular, we have that
\begin{equation}
\label{A8-2}
\widetilde G_n^2(\infty)\frac{\gamma^2_n}{2^{2n-1}} \, \prod_{j: a_{n,j}\neq \infty} \big(2a_{n,j}\phi(a_{n,j})\big) = 1 + o_{\mathcal K}(1) \qasq n\to\infty.
\end{equation}
\end{theorem}
\begin{proof}
The results of this theorem follow from Theorem~\ref{thm:pw1} after we connect orthogonality on \( [-1,1] \) to the orthogonality on the circle by the Joukovski map
\[
J(z) = \big( z+z^{-1} \big)/2.
\]
To translate conditions \( (A_{[-1,1]})-(D_{[-1,1]}) \) into conditions \( (A_\T)-(D_\T) \), set \( b_{2n,j} = \phi(a_{n,j}) \), \( j\in\{1,2,\ldots,2n \} \). That is, \( J(b_{2n,j})=a_{n,j} \). Then
\[
\tau_n(z) = \prod_{a_{n,j}\neq\infty} \frac{(\zeta-b_{2n,j})(1-\zeta b_{2n,j})}{2a_{n,j}b_{2n,j}\zeta} = \frac{W_{2n}(\zeta)W_{2n}^*(\zeta)}{c_n\zeta^{2n}},
\]
\( z=J(\zeta) \), where \( c_n=\prod_{a_{n,j}\neq\infty}2a_{n,j}b_{2n,j} \), \( W_{2n}(\zeta) = \prod_{j=1}^{2n}(\zeta-b_{2n,j}) \), and we used the conjugate symmetry of the multi-set \( \{ b_{2n,1},b_{2n,2},\ldots,b_{2n,2n}\} \). Clearly, conditions \((D_{[-1,1]}) \) and \((D_\T)\) are equivalent to each other. Define \( \mathcal E=\{h\circ J:h\in\mathcal K\} \) and let \( g_{2n}=h_n\circ J \). It trivially holds that condition \( (C_{[-1,1]}) \) implies condition \( (C_\T) \). Every measure $\mu$ defined on $[-1,1]$ can be mapped to a measure $\sigma$ on $\T$ by the formula
\[
2\sigma(A)=\mu(J(A_+)) + \mu(J(A_-)),
\]
where $A$ is  any Borel subset in \( \T \), \( A_+=A\cap\{e^{\ic t}: t\in[0,\pi)\} \), and \( A_-=A\setminus A_+ \). For example, the mapping of the arcsine law $\omega$ results in the normalized Lebesgue measure $m$ on the circle. More generally, if one has a function  \( v \) that is integrable with respect to \( \omega \), the measure \(vd\omega\) is mapped to  \(\upsilon dm$, where \( \upsilon= v\circ J \). We use this map to define measures \( \sigma_{2n} \) on \( \T \) that correspond to the measures $\tilde\mu_n$ on \( [-1,1] \).  This gives $ d\sigma_{2n} = \upsilon_{2n} dm + d\sigma_{2n}^s $ and $d\sigma = \upsilon dm + d\sigma^s$, where  \( \upsilon_{2n}=\tilde v_n \circ J$ and $\upsilon=v \circ  J \). Notice that condition \( (B_{[-1,1]}) \) implies  \( (B_\T) \). Similarly, condition \( (A_{[-1,1]}) \) implies condition \( (A_\T) \). 

Let now \( \phi_{2n} \) be the polynomials satisfying orthogonality relations \eqref{ortho-circle} with the above defined \( (\sigma_{2n},g_{2n}, W_{2n} ) \). It is well-known, see \cite[Theorem~11.5]{Szego} or \cite[Lemma~4.13]{St00}, that
\[
p_n^2(z) = \frac{\phi_{2n}^*(\zeta)^2}{2c_n\zeta^{2n}}\frac{(1+\phi_{2n}(\zeta) / \phi_{2n}^*(\zeta))^2}{1+\phi_{2n}(0)/\phi_{2n}^*(0)} \qandq \gamma_n^2 = 2^{2n-1}\frac{\alpha_{2n}^2}{c_n}\left(1+\frac{\phi_{2n}(0)}{\phi_{2n}^*(0)}\right),
\]
where \( z=J(\zeta) \). It can be readily verified that \( \widetilde G_n(z) = D_{2n}(\zeta) \), see \cite[Lemma~4.3]{St00}. Thus, asymptotic formulae \eqref{A8-1} and \eqref{A8-2} follow from \eqref{A1-1} and \eqref{A1-2}, where one needs to use the identity
\[
\frac{\widetilde G_n^2(z)}{\tau_n(z)} \cdot \prod_{j=1}^{2n}\frac{\phi(z)-\phi(a_{n,j})}{1-\overline{\phi(a_{n,j})}\phi(z)}  = c_n\left(\frac{\zeta^nD_{2n}(\zeta)}{W_{2n}^*(\zeta)}\right)^2. \qedhere
\]
\end{proof}

\small

\bibliographystyle{plain}

\bibliography{st_ang}

\end{document}